\documentclass[11pt,reqno]{amsart}

\usepackage[left=1in,right=1in,top=1in,bottom=1.17in]{geometry}
\usepackage{amsmath,amssymb,amsthm,mathrsfs}
\usepackage[hidelinks]{hyperref}
\usepackage{xcolor}
\usepackage{enumerate}

\newcommand{\E}{\mathcal E}
\newcommand{\D}{\mathcal D}

\newtheorem{theorem}{Theorem}[section]

\newtheorem{lemma}[theorem]{Lemma}
\newtheorem{proposition}[theorem]{Proposition}
\numberwithin{equation}{section}

\theoremstyle{definition}

\newtheorem{remark}[theorem]{Remark}

\title[Uniform stability of degenerate shock--rarefaction waves]{Uniform Stability of degenerate Oleinik shock and rarefaction waves under large perturbations}

\author[M.-J. Kang, B. Kwon, and W. Shim]{Moon-Jin Kang, Bongsuk Kwon, and Wanyong Shim}

\address{(Moon-Jin Kang) Department of Mathematical Sciences, Korea Advanced Institute of Science and Technology, Daejeon, 34141, Korea}
\email{moonjinkang@kaist.ac.kr}

\address{(Bongsuk Kwon) Department of Mathematical Sciences, Ulsan National Institute of Science and Technology, Ulsan, 44919, Korea}
\email{bkwon@unist.ac.kr}

\address{(Wanyong Shim) Department of Mathematical Sciences, Korea Advanced Institute of Science and Technology, Daejeon, 34141, Korea}
\email{wyshim25@kaist.ac.kr}



\thanks{\textbf{Acknowledgments.}
M.K. was supported by  Samsung Science and Technology Foundation under Project Number SSTF-BA2102-01. B.K. was supported by the National Research Foundation of Korea (NRF) grant funded by the Korean government (MSIT) (00560003). W.S. was supported by the Basic Science Research Program through the National Research Foundation of Korea (NRF) funded by the Ministry of Education (RS-2025-25429122).}

\begin{document}

\begin{abstract}
We study the uniform stability of composite waves consisting of a degenerate Oleinik shock and a rarefaction wave for scalar viscous conservation laws. We consider cubic-type fluxes and arbitrarily large $L^2$ initial perturbations of the composite wave within the $H^1$ strong-solution class. We obtain a uniform $L^2$ estimate for such large perturbations. In particular, the estimate implies the contraction of a degenerate Oleinik shock, up to a dynamical shift. As a further consequence, the uniformity with respect to viscosity implies that the associated self-similar solution, composed of the degenerate shock and rarefaction, is stable in the class of inviscid limits. The main analytical difficulty is that the standard weighted Poincar\'e inequality underlying the $a$-contraction method fails to yield the required coercivity at the linear level. We overcome this obstruction by establishing a new degenerate Hardy--Poincar\'e inequality adapted to the degenerate structure of the shock profile. Combined with an exact nonlinear decomposition, this linear inequality yields a strict coercivity estimate that controls the quadratic energy together with the higher-order nonlinear terms, without any smallness assumption on the perturbation.
\end{abstract}

\maketitle


\section{Introduction}

We study the uniform stability of composite waves for the scalar viscous conservation law
\begin{equation} \label{eq:model}
u_t+f(u)_x=u_{xx},\qquad t>0,\quad x\in\mathbb R,
\end{equation}
where $u=u(t,x)$ is the conserved quantity and $f\in C^3(\mathbb R)$ is the flux. The initial data are prescribed by
\begin{equation}\label{eq:initial-data} 
u(0,x)=u_0(x),\qquad \lim_{x\to\pm\infty}u_0(x)=u_\pm. 
\end{equation}

The large-time behavior of the Cauchy problem \eqref{eq:model}--\eqref{eq:initial-data} is expected to be governed by the Riemann problem for the corresponding inviscid conservation law
\begin{equation} \label{eq:inviscid-equation}
u_t + f(u)_x = 0,
\end{equation}
with left and right states $u_-$ and $u_+$, respectively. For a strictly convex flux, the entropy solution of the Riemann problem consists of a single shock or rarefaction wave. A non-convex flux, however, may generate a Riemann solution composed of several elementary waves. In this paper, we focus on a particular composite configuration in which a degenerate Oleinik shock is attached to a rarefaction wave. The cubic flux
\begin{equation*}
f_0(u):=u^3
\end{equation*}
provides a canonical example of this configuration. We consider an open class of $C^3$ perturbations of $f_0$. We denote an \emph{admissible flux} in this class by $f$, while the exact cubic flux is denoted by $f_0$.

For the admissible fluxes considered here, the relevant states satisfy $u_-<0<u_m<u_+$ and
\begin{equation} \label{eq:perturbed-degenerate-compatibility}
\sigma:=\frac{f(u_m)-f(u_-)}{u_m-u_-}=f'(u_m).
\end{equation}
The shock connecting $u_-$ to $u_m$ is thus characteristic at its right end state, and its speed $\sigma$ coincides with the speed of the left edge of the rarefaction fan connecting $u_m$ to $u_+$. The two waves are therefore attached and together form the corresponding inviscid Riemann solution. For $f=f_0$, the compatibility relation gives $u_-=-2u_m$, and the viscous shock profile $U^S$ satisfies
\begin{equation*}
(U^S)'=(U^S+2u_m)(U^S-u_m)^2.
\end{equation*}
The double zero at \(u_m\) produces an algebraically decaying tail on the right, while the simple zero at \(u_-=-2u_m\) yields exponential decay on the left.

For the exact cubic flux $f_0$, Huang, Wang, and Zhang \cite{HWZ} established the first time-asymptotic stability result for this composite wave. Their result allows an arbitrarily strong degenerate Oleinik shock and a sufficiently weak rarefaction under a small $H^1$-type perturbation. Their proof uses a modified $a$-contraction method with a specially constructed shock-dependent weight. The $a$-contraction method  is introduced by Kang and Vasseur in the study on the stability of viscous shock for viscous conservation laws with convex flux \cite{KV,KV1}. 
Based on the typical $a$-contraction method, the leading order terms of small $L^2$-perturbations are controlled through the standard weighted Poincar\'e inequality \cite[Lemma 2.9]{KV1} (see for example  \cite{HKL,KU25,KVW}). In \cite{HWZ}, the standard weighted Poincar\'e inequality is applied only on a truncated portion of the shock profile, while the algebraic tail is treated separately. In their analysis, a common time-dependent shift is applied to the shock and the approximate rarefaction to preserve the attached configuration. We also refer to the recent result \cite{KU} for the time-asymptotic stability of small perturbations of the composition of Oleinik shock and rarefaction for more general flux.

The present paper establishes large-perturbation stability for the same attached configuration, with an arbitrarily strong degenerate Oleinik shock and a sufficiently weak rarefaction. For an open class of $C^3$ perturbations of the cubic flux, we prove a uniform $L^2$ estimate, up to a time-dependent shift, for arbitrarily large $L^2$ perturbations within the $H^1$ strong-solution class. This estimate can be formulated in terms of a shifted shock and an unshifted self-similar rarefaction. The estimate for the normalized equation \eqref{eq:model} is compatible with the parabolic scaling and yields a bound uniform in the viscosity. Taking the vanishing-viscosity limit, we obtain $L^2$ stability and uniqueness of the corresponding inviscid composite wave in a class of weak inviscid limits for initial perturbations in $L^2$.

The key ingredient of our analysis is a coercivity mechanism for the degenerate shock that operates simultaneously at the linear and nonlinear levels. At the linear level, the standard weighted Poincar\'e inequality underlying the $a$-contraction method  does not directly yield the required coercivity. We resolve this obstruction by establishing a new degenerate Hardy--Poincar\'e inequality adapted to the algebraic tail of the shock profile. At the nonlinear level, an exact decomposition for the cubic flux separates the shock contribution into a quadratic form controlled by this inequality and nonnegative nonlinear remainder terms. Together, these ingredients yield a strict nonlinear coercivity estimate that controls the quadratic and higher-order terms without any smallness assumption on the perturbation. This strict coercivity persists under sufficiently small perturbations of the flux. A more detailed description is given in Section~\ref{sec:intro-main-ideas}.

We next place these results in the context of the stability theory for scalar conservation laws. The large-time analysis of scalar viscous conservation laws goes back to Il'in and Oleinik \cite{IO}. Subsequent works established nonlinear stability and decay estimates for single viscous shocks under general or non-convex flux assumptions, including degenerate and partially linearly degenerate cases; see, among others, \cite{FS,HX,JGK,MN94,MY,Mei}. Together, these works provide a broad range of approaches to establishing convergence toward a single traveling wave.

A different line of work concerns stability under large initial perturbations. At the inviscid level, Leger proved $L^2$ stability of scalar shocks for convex fluxes up to a shift \cite{Leger}, while Cheng obtained an analogous result for concave--convex fluxes \cite{Cheng}. For viscous scalar conservation laws, Nishihara and Zhao obtained algebraic $L^\infty$ convergence rates toward a single shock profile, including characteristic degenerate shocks, for arbitrarily large perturbations under suitable spatial localization assumptions \cite{NZh}. Choi and Vasseur derived short-time $L^2$ estimates uniform in the viscosity for a class of large perturbations \cite{CV}. Kang and Vasseur established global exact $L^2$ contraction up to a shift for viscous shocks of fluxes close to Burgers, without smallness assumptions on either the shock or the perturbation \cite{KV}. Weighted $L^2$-type contraction for small shocks with general strictly convex fluxes was proved by Kang \cite{Kang2021}. Related large-perturbation stability and viscosity-uniform inviscid-limit results for the Navier--Stokes system can be found in \cite{EEK,KV1,KV2}.

The stability theory for composite waves consisting of a viscous shock and a rarefaction wave developed more recently. Classical approaches to the two component waves were not directly compatible: viscous shocks were typically treated by anti-derivative methods, whereas rarefaction waves were handled by direct energy estimates. Kang, Vasseur, and Wang established the time-asymptotic stability of such composite waves for the barotropic Navier--Stokes equations by combining the $a$-contraction method with energy estimates for the rarefaction wave \cite{KVW}; for related extensions, see \cite{EHK,GH, HanKim,KVW25,Shim,SW}. This framework was further extended to generic Riemann solutions containing a shock, a contact wave, and a rarefaction for the full compressible Navier--Stokes--Fourier system \cite{KVW25}. In the scalar non-convex setting, Huang, Wang, and Zhang treated the attached configuration of a degenerate Oleinik shock and a rarefaction wave for the exact cubic flux $f_0$ \cite{HWZ}. The present work combines this attached geometry with large-perturbation estimates, the vanishing-viscosity limit, and robustness under perturbations of the cubic flux.

\subsection{Component waves and reference profile}

We first introduce the component waves and the reference composite profile.

\subsubsection{Degenerate viscous Oleinik shock} \label{sec:1.1}

Let $\sigma$ be the shock speed defined in \eqref{eq:perturbed-degenerate-compatibility}. In the moving coordinate $\xi=x-\sigma t$, the viscous shock profile $u^S=u^S(\xi)$ satisfies
\begin{equation} \label{eq:shock-ode}
-\sigma u^S_\xi+f(u^S)_\xi=u^S_{\xi\xi},\quad \lim_{\xi\to-\infty}u^S(\xi)=u_-,\quad \lim_{\xi\to+\infty}u^S(\xi)=u_m,\quad u^S(0)=\frac{u_-+u_m}{2}.
\end{equation}
The last condition fixes the translation of the profile. Integrating the profile equation from \(+\infty\), we obtain
\begin{equation*}
u^S_\xi=f(u^S)-f(u_m)-f'(u_m)(u^S-u_m).
\end{equation*}
The right-hand side vanishes at both \(u_-\) and \(u_m\) by \eqref{eq:perturbed-degenerate-compatibility}, and its first derivative also vanishes at \(u_m\). Thus the equilibrium at the right end state is degenerate. The existence and quantitative properties of such shock profiles for admissible fluxes are established in Section~\ref{sec:shock-properties}.

The structure is explicit for the canonical flux \(f_0(u)=u^3\). In this case, the compatibility relation determines
\begin{equation*}
U_-:=-2u_m,\qquad \sigma_0:=f_0'(u_m)=3u_m^2.
\end{equation*}
Let \(U^S=U^S(\xi)\) denote the corresponding profile normalized by
\begin{equation*}
-\sigma_0 U^S_\xi+f_0(U^S)_\xi=U^S_{\xi\xi},\quad \lim_{\xi\to-\infty}U^S(\xi)=U_-,\quad \lim_{\xi\to+\infty}U^S(\xi)=u_m,\quad U^S(0)=\frac{U_-+u_m}{2}.
\end{equation*}
Its first-order equation factorizes exactly as
\begin{equation}\label{eq:cubic-shock-first-order}
U^S_\xi=(U^S-U_-)(U^S-u_m)^2>0.
\end{equation}
The simple factor at \(U_-\) and the double factor at \(u_m\) determine the exponential left tail and the algebraic right tail, respectively.

\subsubsection{Rarefaction wave and its smooth approximation} \label{sec:1.2}

Let $\lambda_m:=f'(u_m)$ and $\lambda_+:=f'(u_+)$. For the admissible fluxes and rarefaction strengths considered below, $f'$ is strictly increasing on $[u_m,u_+]$; see \eqref{eq:rarefaction-convexity}. The centered rarefaction wave connecting $u_m$ to $u_+$ is defined, for $t>0$, by
\begin{equation*} 
\bar u^r\left(\frac{x}{t}\right):=\begin{cases}u_m,&x\leq\lambda_m t,\\ (f')^{-1}(x/t),&\lambda_m t<x<\lambda_+t,\\ u_+,&x\geq\lambda_+t.\end{cases}
\end{equation*}
Its initial datum equals $u_m$ for $x<0$ and $u_+$ for $x>0$.

We next introduce the smooth approximate rarefaction wave. Let $w^R=w^R(t,x)$ be the unique global smooth solution of
\begin{equation}\label{eq:burgers-rarefaction}
w^R_t+w^Rw^R_x=0,\qquad w^R(0,x)=\frac{\lambda_m+\lambda_+}{2}+\frac{\lambda_+-\lambda_m}{2}\tanh x,
\end{equation}
and define
\begin{equation}\label{def_rarefaction}
\bar u^R(t,x):=(f')^{-1}\bigl(w^R(t,x)\bigr).
\end{equation}
Since the initial profile for $w^R$ is increasing, its characteristics do not intersect and $w^R_x>0$. Moreover, $w^R=f'(\bar u^R)$, and hence
\begin{equation*}
\bar u^R_t+f(\bar u^R)_x=0,\qquad \bar u^R_x>0,\qquad \lim_{x\to-\infty}\bar u^R(t,x)=u_m,\qquad \lim_{x\to+\infty}\bar u^R(t,x)=u_+.
\end{equation*}

We define the approximate rarefaction in the shock frame $\xi=x-\sigma t$ by
\begin{equation}\label{eq:rarefaction-shock-frame}
u^R(t,\xi):=\bar u^R(1+t,\xi+\sigma(1+t)).
\end{equation}
The time shift permits comparison with the centered rarefaction fan at a positive time when $t=0$. The spatial translation aligns the left edge of that fan at time $1+t$ with $\xi=0$, since $\sigma=\lambda_m$. In the original variables, this corresponds to $\bar u^R(1+t,x+\sigma)$. The profile $u^R$ satisfies
\begin{equation}\label{eq:rarefaction-frame-equation}
u^R_t-\sigma u^R_\xi+f(u^R)_\xi=0,\qquad u^R_\xi>0.
\end{equation}

\subsubsection{Reference composite profile and residual} \label{sec:1.3}

The stability estimate for \eqref{eq:model} is formulated relative to the composite wave
\begin{equation*}
u^S(x-\sigma t+X(t))+\bar u^r\left(\frac{x}{t}\right)-u_m
\end{equation*}
for a suitable time-dependent shift $X$. For our analysis, however, we use a smooth reference profile obtained by replacing the exact rarefaction wave with its smooth approximation $u^R$. We define
\begin{equation}\label{eq:composite-profile}
\widetilde u(t,\xi):=u^S(\xi)+u^R(t,\xi)-u_m.
\end{equation}
In the energy estimates, we compare $u(t,x)$ with $\widetilde u(t,x-\sigma t+X(t))$, so the same time-dependent shift is applied to both component waves. This common shift preserves the attached geometry. The rarefaction approximation estimates, together with the control of the shift, then allow us to replace the shifted approximate rarefaction by the unshifted self-similar wave $\bar u^r(x/t)$ in the stability estimate, yielding the form stated above.

For the unshifted reference profile, \eqref{eq:shock-ode} and \eqref{eq:rarefaction-frame-equation} give
\begin{equation}\label{eq:composite-residual-equation}
\widetilde u_t-\sigma\widetilde u_\xi+f(\widetilde u)_\xi-\widetilde u_{\xi\xi}=F,
\end{equation}
where
\begin{equation}\label{eq:residual}
F:=\bigl[f'(\widetilde u)-f'(u^S)\bigr]u^S_\xi+\bigl[f'(\widetilde u)-f'(u^R)\bigr]u^R_\xi-u^R_{\xi\xi}.
\end{equation}
The first two terms are generated by the nonlinear interaction of the component waves, while the last term is the viscous defect of the approximate rarefaction. For $f=f_0$, the residual is
\begin{equation*}
F=3(u^R-u_m)(2u^S+u^R-u_m)u^S_\xi+3(u^S-u_m)(u^S+2u^R-u_m)u^R_\xi-u^R_{\xi\xi}.
\end{equation*}

\subsection{Main results}

Let $f\in C^3(\mathbb R)$ and set $q(u):=f(u)-u^3$. We measure the perturbation by the seminorm
\begin{equation}\label{eq:flux-perturbation-size}
\varepsilon_f:=\sup_{u\in\mathbb R}\frac{|q''(u)|}{u_m+|u|}+\lVert q'''\rVert_{L^\infty(\mathbb R)}.
\end{equation}
The denominator $u_m+|u|$ makes this seminorm a global measure of closeness compatible with cubic growth: $q''$ may grow linearly in $|u|$, as does $(u^3)''$, but only with a small coefficient, while $q'''$ remains uniformly small. In particular, the class of fluxes for which $\varepsilon_f$ is sufficiently small contains all fluxes of the form
\begin{equation*}
f(u)=(1+\theta_3)u^3+\theta_2u_mu^2+au+b,
\end{equation*}
with $|\theta_2|+|\theta_3|$ sufficiently small and $a,b\in\mathbb R$ arbitrary; indeed, $\varepsilon_f\leq2|\theta_2|+12|\theta_3|$. The seminorm $\varepsilon_f$ does not measure constant or linear perturbations of the flux: an additive constant has no effect on the equation, while a linear perturbation can be removed by a change of moving frame.

\subsubsection{Uniform $L^2$ stability of the composite wave}

Before stating our first main result, we recall the reference composite profile $\widetilde u$ defined in \eqref{eq:composite-profile} and introduce the notion of a strong solution. For $T>0$, we say that $u$ is a strong solution of \eqref{eq:model}--\eqref{eq:initial-data} on $[0,T]$ if
\begin{equation} \label{strong_regularity}
u-\widetilde u(0,\cdot)\in C([0,T];H^1(\mathbb R))\cap L^2(0,T;H^2(\mathbb R)),\qquad
u_t\in L^2(0,T;L^2(\mathbb R)),
\end{equation}
and \eqref{eq:model} holds in $L^2(\mathbb R)$ for almost every $t\in(0,T)$, with $u(0)=u_0$.

\begin{theorem} \label{mainthm}
Fix $u_m>0$. There exist constants $\varepsilon_0>0$, $\delta_0>0$, and $C>0$ such that the following statement holds.

Let \(f\in C^3(\mathbb R)\) satisfy \(\varepsilon_f\leq\varepsilon_0\). Then there exists a unique state \(u_-\in(-5u_m/2,-3u_m/2)\) satisfying \eqref{eq:perturbed-degenerate-compatibility}. For this state, there exists a unique viscous shock profile \(u^S\in C^2(\mathbb R)\) satisfying \eqref{eq:shock-ode}. For any \(\delta_R>0\) with \(\delta_R/u_m\leq\delta_0\), set \(u_+:=u_m+\delta_R\). Then \(f'\) is strictly increasing on \([u_m,u_+]\), so that \(f'(u_m)<f'(u_+)\) and the rarefaction wave \(\bar u^r\) introduced in Section~\ref{sec:1.2} is well defined. Let \(\widetilde u\) denote the reference composite profile defined in Section~\ref{sec:1.3}.

Assume that $u_0-\widetilde u(0,\cdot)\in H^1(\mathbb R)$. Then the Cauchy problem \eqref{eq:model}--\eqref{eq:initial-data} admits a unique global strong solution $u$. Moreover, there exists a shift function $X\in C^1([0,\infty))$ with $X(0)=0$ such that
\begin{equation}\label{eq:normalized-contraction-power}
\begin{split}
&\left\|u(t,\cdot) - \left( u^S(\cdot-\sigma t + X(t)) + \bar{u}^r \left(\frac{\cdot}{t}\right) - u_m \right) \right\|_{L^2(\mathbb R)}^2 \\
& \qquad \leq  C \left(\left\|u_0-\widetilde u(0,\cdot)\right\|_{L^2(\mathbb R)}^2 + (\delta_R/u_m)^{8/33}\right)
\end{split}
\end{equation}
for all $t>0$. Furthermore, the shift \(X\) satisfies
\begin{equation}\label{eq:shift-velocity-bound}
\int_0^\infty |\dot X(t)|^2\,dt \leq C\left(\left\|u_0-\widetilde u(0,\cdot)\right\|_{L^2(\mathbb R)}^2+\left(\delta_R/u_m\right)^{8/33}\right).
\end{equation}
\end{theorem}

\begin{remark}[Contraction for viscous degenerate Oleinik shocks]
In the case $\delta_R=0$, our analysis yields a weighted $L^2$-contraction of a single viscous degenerate Oleinik shock. More precisely, there exists a shift $X^S$, with $X^S(0)=0$, such that
\begin{equation*}
\mathcal E^S(t)\leq \mathcal E^S(0)
\end{equation*}
for all $t\geq0$, where
\begin{equation*}
\mathcal E^S(t):=\frac12\int_{\mathbb R}\bigl(5u_m-u^S(x-\sigma t+X^S(t))\bigr)\bigl(u(t,x)-u^S(x-\sigma t+X^S(t))\bigr)^2\,dx.
\end{equation*}
This follows from the argument establishing the a priori estimate \eqref{eq:large-L2-apriori} in Proposition~\ref{prop:large-L2-apriori}: when $\delta_R=0$, we have $u^R\equiv u_m$, $\widetilde u=u^S$, and $F\equiv Q\equiv0$. The argument applies with $\eta=0$ in this case, since the positivity of $\eta$ is needed only to obtain strict coercivity of the rarefaction block, which is absent here; see Lemma~\ref{lem:cubic-rarefaction-coercivity}.
\end{remark}

\begin{remark}[Recovery of asymptotic stability for small perturbations]
Although Theorem~\ref{mainthm} itself is formulated as a uniform estimate for large perturbations, the coercivity analysis developed here can be combined with the small-perturbation framework of \cite{HWZ} to recover their asymptotic stability result. In fact, the quadratic form arising from the direct linearization is $\mathcal H$ plus nonnegative quadratic terms. Therefore, the coercivity of $\mathcal H$ established by Proposition~\ref{Poincare} and Lemma~\ref{lem:shock-endpoint-recovery} implies the linear coercivity required in the small-perturbation analysis.

More precisely, the linear coercivity estimate can be used to simplify the coercivity argument in \cite[Lemmas~4.3--4.5]{HWZ} by avoiding the truncation of the shock profile and the separate treatment of the algebraic tail. With this simplification, the remaining argument follows the framework of \cite{HWZ}: the $H^1$ estimate is closed by absorbing the nonlinear terms through the smallness of the perturbation and controlling the wave-interaction errors. Their time-asymptotic argument then yields the asymptotic stability of the composite wave.
\end{remark}

\subsubsection{$L^2$ stability in vanishing-viscosity limits}

The normalized estimate in Theorem~\ref{mainthm} is compatible with the parabolic scaling and therefore yields estimates for equations with arbitrary viscosity. For $\nu>0$, consider
\begin{equation}\label{eq:model-nu}
u^\nu_t+f(u^\nu)_x=\nu u^\nu_{xx},\qquad u^\nu(0,x)=u_0^\nu(x),\qquad t>0,\quad x\in\mathbb R.
\end{equation}
Under the parabolic scaling, the centered self-similar rarefaction wave is unchanged, while the viscous shock profile is rescaled. As $\nu\to0$, the rescaled viscous shock approaches the corresponding inviscid shock. Consequently, the reference profile $\widetilde u$, after rescaling, converges to the Riemann solution of the inviscid conservation law \eqref{eq:inviscid-equation}.

For the admissible fluxes, the Riemann solution of the inviscid conservation law \eqref{eq:inviscid-equation} with initial states $(u_-,u_+)$ consists of an Oleinik shock $\bar u^s$ connecting $u_-$ to $u_m$ and the rarefaction wave $\bar u^r(x/t)$ connecting $u_m$ to $u_+$. Define
\begin{equation*}
\bar u^s(x-\sigma t):=
\begin{cases}
u_-,&x<\sigma t,\\
u_m,&x\geq\sigma t.
\end{cases}
\end{equation*}
Then the Riemann solution is given by
\begin{equation*}
\bar u(t,x):=\bar u^s(x-\sigma t)+\bar u^r\left(\frac{x}{t}\right)-u_m,\qquad t>0,
\end{equation*}
with initial datum
\begin{equation*}
\bar u_0(x):=\begin{cases}u_-,&x<0,\\u_+,&x>0.\end{cases}
\end{equation*}

The following theorem establishes the stability and uniqueness of the Riemann solution $\bar u$ of \eqref{eq:inviscid-equation}.

\begin{theorem}\label{cor:inviscid-stability}
Fix $u_m>0$, and let $\varepsilon_0>0$ and $\delta_0>0$ be chosen as in Theorem~\ref{mainthm}. Let $f\in C^3(\mathbb R)$ satisfy $\varepsilon_f\leq\varepsilon_0$, let $u_-$ and $\sigma$ be the corresponding left state and shock speed, and set $u_+:=u_m+\delta_R$ for $0<\delta_R/u_m\leq\delta_0$.

Then, for any initial datum $u_0$ of \eqref{eq:inviscid-equation} satisfying
\begin{equation*}
u_0-\bar u_0\in L^2(\mathbb R),
\end{equation*}
the following is true.

\begin{enumerate}[(i)]
\item There exists a sequence of smooth functions $\{u_0^\nu\}_{\nu>0}$ such that
\begin{equation}\label{eq:well-prepared-initial-data}
u_0^\nu-\widetilde u(0,\cdot/\nu)\in H^1(\mathbb R),\qquad \lim_{\nu\to0}\|u_0^\nu-u_0\|_{L^2(\mathbb R)}=0.
\end{equation}
\item Let $\{u_0^\nu\}_{\nu>0}$ be any sequence of smooth initial data satisfying \eqref{eq:well-prepared-initial-data}, and let $\{u^\nu\}_{\nu>0}$ be the corresponding global strong solutions to \eqref{eq:model-nu}. Then there exists a limit $u_\infty$ such that, as $\nu\to0$ up to a subsequence,
\begin{equation}\label{eq:weak-inviscid-limit}
u^\nu\rightharpoonup u_\infty\quad\text{in }L^2_{\mathrm{loc}}((0,T)\times\mathbb R)
\end{equation}
for any $T>0$.
\item For every subsequential limit $u_\infty$ as in \eqref{eq:weak-inviscid-limit}, there exist a shift $X_\infty\in H^1_{\mathrm{loc}}([0,\infty))$, with $X_\infty(0)=0$, and a constant $C>0$ such that, for any $T>0$,
\begin{equation}\label{eq:inviscid-unweighted-stability}
\left\|u_\infty(t)-\left(\bar u^s(\cdot-\sigma t+X_\infty(t))+\bar u^r\left(\frac{\cdot}{t}\right)-u_m\right)\right\|_{L^2(\mathbb R)}^2\leq C\|u_0-\bar u_0\|_{L^2(\mathbb R)}^2
\end{equation}
for almost every $t\in(0,T)$. Moreover, the shift $X_\infty$ satisfies
\begin{equation}\label{eq:inviscid-shift-bound}
|X_\infty(t)|\leq C\sqrt{t}\,\|u_0-\bar u_0\|_{L^2(\mathbb R)}
\end{equation}
for all $t\in[0,T]$.
\end{enumerate}

Therefore, the Riemann solution $\bar u$ of \eqref{eq:inviscid-equation} is orbitally stable and unique in the class of weak inviscid limits of solutions to \eqref{eq:model-nu}.
\end{theorem}

\subsection{Main ideas and outline of the paper} \label{sec:intro-main-ideas}

The proof proceeds in two main stages. First, we work with the viscosity-normalized equation \eqref{eq:model} and derive an $L^2$ estimate relative to a suitably shifted reference composite profile $\widetilde u$, from which the stability estimate in Theorem~\ref{mainthm} follows. Second, we restore the viscosity scale and take the vanishing-viscosity limit, obtaining Theorem~\ref{cor:inviscid-stability} for the inviscid composite wave.

The main point in the first stage is to identify the nonlinear coercive structure of the weighted energy:
\begin{equation*}
\mathcal E(t) := \frac12 \int_\mathbb{R} \big( (5-\eta)u_m - \widetilde u \big) \phi^2 \, d\xi, \qquad \text{where} \quad  \phi(t,\xi):=u(t,\xi+\sigma t-X(t))-\widetilde u(t,\xi).
\end{equation*}
Here, $\xi=x-\sigma t$ and $\eta>0$ is a sufficiently small parameter. The first observation is that the weight in $\mathcal{E}$ is adapted to the large-perturbation problem for the composite wave. Its dependence on the full profile places the nonlinear contributions of both waves in a common algebraic structure. This provides the basic structure underlying the nonlinear coercivity, in which both the shock and the rarefaction contributions are exploited.

This structure is first exposed for the canonical cubic flux $f_0(u)=u^3$. For this flux, the weighted shock and rarefaction contributions admit exact completions of squares. In the shock contribution, the completion separates the nonlinear functional into nonnegative square terms and a remaining quadratic form. This separation is the essential point of the cubic calculation: once the nonlinear terms have been organized, the coercivity problem is reduced to understanding a single quadratic contribution associated with the shock.

To analyze this quadratic form, we introduce the shock-profile coordinate
\begin{equation*}
y:= \frac{U^S-U_-}{u_m-U_-} \in(0,1), \qquad \Phi(y):= \frac{\phi(\xi(y))}{u_m}
\end{equation*}
for $U^S$ and $U_-$ described in Section~\ref{sec:1.1}. This change of variables compactifies the whole-line problem and rewrites the quadratic form in an endpoint-degenerate Sturm--Liouville form. In particular, the algebraic tail of the shock appears as a degeneracy of the differential weight at $y=1$, and the standard weighted Poincar\'e inequality underlying the $a$-contraction method does not control this endpoint degeneracy. We reveal that the appropriate replacement is a Hardy--Poincar\'e inequality: there exists a constant $c_P>1$ such that
\begin{equation*}
\begin{split}
& \int_0^1 (7-3y)y(1-y)^2 \Phi_y^2\,dy + \frac{10}{3}\left(\int_0^1 (3-2y) \Phi\,dy\right)^2 \geq c_P \int_0^1 (1-y)(12-y) \Phi^2\,dy.
\end{split}
\end{equation*}
The rank-one moment term is generated by the shift and controls the one-dimensional obstruction associated with translation invariance; in the compactified quadratic form, the unstable direction is represented by the constants. The remaining component is controlled through a Hardy inequality adapted to the endpoint degeneracies. Thus the shift and the Hardy estimate have distinct and complementary roles: the former removes the finite-dimensional obstruction, while the latter supplies coercivity at the degenerate endpoint. Combined with the exact nonlinear decomposition, this yields a strict nonlinear coercivity estimate for the shock contribution, together with the corresponding coercive contribution from the rarefaction.

The strictness of this estimate is what allows the argument to move beyond the exact cubic flux. Writing $q:=f-f_0$, we measure the perturbation by
\begin{equation*}
\varepsilon_f := \sup_{u\in\mathbb R} \frac{|q''(u)|}{u_m+|u|} + \lVert q''' \rVert_{L^\infty(\mathbb R)}.
\end{equation*}
This seminorm is adapted to the quantities that arise in the nonlinear energy identity and in the factorization of the degenerate shock profile. It ignores affine changes of the flux, which do not affect the relevant nonlinear structure, and controls precisely the additional terms produced by a perturbation of the cubic flux. The strict coercivity gap for the reference problem absorbs these terms when $\varepsilon_f$ is sufficiently small. In this way, the exact cubic calculation identifies the coercive structure, while the result itself is formulated for an open class of fluxes for which that structure persists uniformly.

Once the nonlinear coercivity is established, the remaining difficulty is the interaction of the weak rarefaction with the algebraic tail of the degenerate shock. The two wave-interaction estimates needed here were obtained for the cubic flux in \cite[Lemma~4.8]{HWZ}. They rely on the shock tail bounds and the standard decay estimates for the approximate rarefaction. We show that these estimates remain uniform over the admissible flux class, and use them to derive the required residual bounds. The normalized estimate in Theorem~\ref{mainthm} is therefore obtained by combining the strict nonlinear coercivity with these interaction and residual estimates in the $L^2$-equivalent energy.

To prove Theorem~\ref{cor:inviscid-stability}, we follow the inviscid-limit argument of \cite[Section~5]{KV2}. We construct well-prepared initial data and use the viscosity-uniform estimate to obtain subsequential weak $L^2_{\mathrm{loc}}$ limits of the viscous solutions. The uniform stability and shift estimates then yield the $L^2$ stability estimate for the inviscid limits, while the zero-perturbation case gives uniqueness of the Riemann solution within this class.

The paper is organized as follows. The preliminary section establishes the uniform structure and tail bounds of the perturbed degenerate shock, records the properties of the approximate rarefaction wave, and states the degenerate Hardy--Poincar\'e inequality. Section~\ref{sec:3} derives the weighted energy identity and separates the shock, rarefaction, interaction, and residual contributions. Section~\ref{sec:4} identifies the nonlinear coercive mechanism for the cubic flux and proves its persistence under perturbations of the flux. Section~\ref{sec:5} establishes the uniform wave-interaction estimates and closes the a priori estimate. Section~\ref{sec:proof-mainthm} proves Theorems~\ref{mainthm} and~\ref{cor:inviscid-stability}. The appendices provide the proofs deferred from Section~2 and establish the global well-posedness result used in the proof of Theorem~\ref{mainthm}.

\section{Preliminaries}
In this section, we record the structural and quantitative properties of the component waves and state the degenerate Hardy--Poincar\'e inequality.

\subsection{Structure and properties of the degenerate viscous shock} \label{sec:shock-properties}

The following two lemmas establish the existence and uniqueness of the left state \(u_-\) and the corresponding viscous shock profile asserted in Theorem~\ref{mainthm}. They also record the factorization and pointwise bounds used in the coercivity and wave-interaction estimates.

\begin{lemma}[Persistence of the degenerate structure] \label{lem:perturbed-degenerate-profile}
There exist constants $\varepsilon_0>0$, $c>0$, and $C>0$ such that, if $\varepsilon_f\leq\varepsilon_0$, then there exists a unique state
\begin{equation*}
u_-\in(-5u_m/2,-3u_m/2)
\end{equation*}
satisfying \eqref{eq:perturbed-degenerate-compatibility}. Moreover, 
\begin{equation}\label{eq:perturbed-left-state}
|u_-+2u_m|\leq C\varepsilon_fu_m,
\end{equation}
and there exists a function \(b\in C([u_-,u_m])\) such that
\begin{equation}\label{eq:perturbed-shock-factorization}
f(u)-f(u_m)-f'(u_m)(u-u_m)=(u-u_-)(u-u_m)^2b(u),\qquad u\in[u_-,u_m],
\end{equation}
with
\begin{equation}\label{eq:perturbed-factor-bound}
|b(u)-1|\leq C\varepsilon_f,\qquad \frac12\leq b(u)\leq\frac32.
\end{equation}
Furthermore,
\begin{equation}\label{eq:rarefaction-convexity}
f''(u)\geq cu_m \quad \text{for all } u \geq u_m.
\end{equation} 
\end{lemma}

\begin{lemma} [Existence and properties of the degenerate viscous shock profile] \label{lem:existence_shock}
Let $u_-$ be the state given in Lemma~\ref{lem:perturbed-degenerate-profile}. There exists a unique viscous shock profile \(u^S\in C^2(\mathbb R)\) satisfying \eqref{eq:shock-ode}. The profile $u^S$ satisfies
\begin{equation}\label{eq:perturbed-shock-ode}
u^S_\xi=(u^S-u_-)(u^S-u_m)^2b(u^S)>0,\qquad \xi\in\mathbb R.
\end{equation}
Moreover,
\begin{equation}\label{eq:perturbed-shock-left-tail}
0<u^S(\xi)-u_-\leq C\delta_Se^{-c\delta_S^2|\xi|},\qquad 0<u^S_\xi(\xi)\leq C\delta_S^3e^{-c\delta_S^2|\xi|},\qquad \xi\leq0,
\end{equation}
and
\begin{equation}\label{eq:perturbed-shock-right-tail}
0<u_m-u^S(\xi)\leq\frac{C\delta_S}{1+c\delta_S^2\xi},\qquad 0<u^S_\xi(\xi)\leq\frac{C\delta_S^3}{(1+c\delta_S^2\xi)^2},\qquad \xi\geq0,
\end{equation}
where \(\delta_S:=u_m-u_-\).
\end{lemma}

The proofs of Lemmas~\ref{lem:perturbed-degenerate-profile}--\ref{lem:existence_shock} are given in  Appendix~\ref{app:profile}.

\subsection{Estimates for the approximate rarefaction wave}

Set
\begin{equation*}
\lambda_R:=f'(u_+)-f'(u_m)>0.
\end{equation*}
We denote by \(u^r=u^r(t,\xi)\) the exact rarefaction wave at time \(1+t\), expressed in the shock frame:
\begin{equation}\label{eq:exact-rarefaction-shock-frame}
u^r(t,\xi):=
\begin{cases}
u_m,&\xi\leq0,\\ 
(f')^{-1}\!\left(f'(u_m)+\dfrac{\xi}{1+t}\right),&0\leq\xi\leq\lambda_R(1+t),\\ u_+,&\xi\geq\lambda_R(1+t).
\end{cases}
\end{equation}
The following lemma records the derivative, tail, and approximation estimates for \(u^R\).

\begin{lemma}[Properties of the approximate rarefaction wave]\label{lem:approximate-rarefaction}
There exist constants $\varepsilon_0>0$ and $\delta_0>0$ such that, if $\varepsilon_f\leq\varepsilon_0$ and $\delta_R/u_m\leq\delta_0$, then the approximate rarefaction wave $u^R$ defined by \eqref{eq:burgers-rarefaction}, \eqref{def_rarefaction}, and \eqref{eq:rarefaction-shock-frame} satisfies the following properties:
\begin{enumerate}[(i)]
\item For all \(t\geq0\) and \(\xi\in\mathbb R\),
\begin{equation*}
u_m<u^R(t,\xi)<u_+,\qquad u^R_\xi(t,\xi)>0.
\end{equation*}
\item For any \(p\in[1,\infty]\), there exists \(C_p>0\) such that, for all \(t\geq0\),
\begin{equation}\label{eq:approx-rarefaction-derivatives}
\begin{split}
& \|u^R_\xi(t)\|_{L^p(\mathbb R)}\leq C_p\min\!\left\{\delta_R,\delta_R^{1/p}(1+t)^{-1+1/p}\right\},\\
& \|u^R_{\xi\xi}(t)\|_{L^p(\mathbb R)}\leq C_p\min\!\left\{\delta_R,(1+t)^{-1}\right\}.
\end{split}
\end{equation}
\item The approximate rarefaction wave converges uniformly to \(u^r\):
\begin{equation*}
\lim_{t\to\infty}\sup_{\xi\in\mathbb R}|u^R(t,\xi)-u^r(t,\xi)|=0.
\end{equation*}
\item There exists \(C>0\) such that, for all \(t\geq0\),
\begin{equation}\label{eq:approx-rarefaction-tails}
\begin{aligned}
& |u^R(t,\xi)-u_m|\leq C\delta_Re^{-2|\xi|}, & &\quad  \xi\leq0, \\
& |u^R(t,\xi)-u_+|\leq C\delta_Re^{-2|\xi-\lambda_R(1+t)|}, & &\quad  \xi\geq\lambda_R(1+t).
\end{aligned}
\end{equation}
\item For any \( \theta \in(0,1)\), there exists \(C_\theta>0\) such that, for all $t\geq0$,
\begin{equation}\label{eq:approx-rarefaction-refined-tails}
\begin{aligned}
& |u^R(t,\xi)-u_m|\leq C_\theta\delta_R^{2\theta/(2+\theta)}(1+t)^{-1+\theta}e^{-\theta|\xi|}, & & \quad \xi\leq0, \\
&|u^R(t,\xi)-u_+|\leq C_\theta\delta_R^{2\theta/(2+\theta)}(1+t)^{-1+\theta}e^{-\theta|\xi-\lambda_R(1+t)|}, & &\quad  \xi\geq\lambda_R(1+t).
\end{aligned}
\end{equation}
\item For any \(\theta\in(0,1)\), there exists \(C_\theta>0\) such that, for all \(t\geq0\),
\begin{equation}\label{eq:approx-rarefaction-fan-error}
|u^R(t,\xi)-u^r(t,\xi)|\leq C_\theta\delta_R^\theta(1+t)^{-1+\theta}, \quad 0\leq\xi\leq\lambda_R(1+t).
\end{equation}
\end{enumerate}
The constants \(C_p\), \(C\), and \(C_\theta\) are uniform over the admissible fluxes satisfying \(\varepsilon_f \leq \varepsilon_0\).
\end{lemma}

\begin{proof}
The estimates in the statement follow from \cite[Lemma~3.1]{HWZ} by evaluating the approximate rarefaction wave at time \(1+t\) in the moving frame \(\xi=x-\sigma t\), as in \eqref{eq:rarefaction-shock-frame}. Thus, it remains only to verify that the constants in \eqref{eq:approx-rarefaction-derivatives} and \eqref{eq:approx-rarefaction-tails}--\eqref{eq:approx-rarefaction-fan-error} can be chosen independently of \(f\).

By \eqref{eq:flux-perturbation-size}, \eqref{eq:rarefaction-convexity}, and the assumption \(\delta_R/u_m\leq\delta_0\), there exist constants \(c>0\) and \(C>0\), independent of \(f\), such that
\begin{equation} \label{f''f''bd}
cu_m\leq f''(u)\leq Cu_m,\qquad |f'''(u)|\leq C,\qquad u\in[u_m,u_+].
\end{equation}
Let \(g=(f')^{-1}\). Then it follows from \eqref{f''f''bd} that
\begin{equation} \label{g'g''bd}
g'(w)=\frac{1}{f''(g(w))} \leq \frac{C}{u_m},\qquad |g''(w)|= \left|-\frac{f'''(g(w))}{(f''(g(w)))^3} \right| \leq \frac{C}{u_m^3}
\end{equation}
for \(w\in[f'(u_m),f'(u_+)]\). Moreover,
\begin{equation} \label{lambdaR}
cu_m\delta_R\leq\lambda_R=f'(u_+)-f'(u_m) =\int_{u_m}^{u_+}f''(v)\,dv\leq Cu_m\delta_R.
\end{equation}

We recall from \eqref{def_rarefaction} and \eqref{eq:rarefaction-shock-frame} that
\begin{equation*}
u^R(t,\xi)=g\!\left(w^R(1+t,\xi+\sigma(1+t))\right).
\end{equation*}
In particular,
\begin{equation*}
u^R_\xi=g'(w^R)w^R_x,\qquad
u^R_{\xi\xi}=g''(w^R)(w^R_x)^2+g'(w^R)w^R_{xx},
\end{equation*}
where \(w^R\), \(w^R_x\), and \(w^R_{xx}\) are evaluated at \((1+t,\xi+\sigma(1+t))\). Writing \(s=1+t\), the estimates in \cite[Lemma~2.1(ii)--(iii)]{MN}, together with \eqref{g'g''bd} and \eqref{lambdaR}, give
\begin{equation}\label{eq:uniform-rarefaction-first-derivative}
\begin{aligned}
\|u^R_\xi(t)\|_{L^p} &\leq \|g'(w^R)\|_{L^\infty}\|w^R_x(s)\|_{L^p} \leq \frac{C}{u_m}C_p\min\left\{\lambda_R,\lambda_R^{1/p}s^{-1+1/p}\right\}\\
&\leq C_p\min\left\{\delta_R,\delta_R^{1/p}s^{-1+1/p}\right\}.
\end{aligned}
\end{equation}
Similarly,
\begin{equation}\label{eq:uniform-rarefaction-second-derivative}
\begin{aligned}
\|u^R_{\xi\xi}(t)\|_{L^p}
&\leq \|g''(w^R)\|_{L^\infty}\|(w^R_x(s))^2\|_{L^p}
+\|g'(w^R)\|_{L^\infty}\|w^R_{xx}(s)\|_{L^p}\\
&\leq \frac{C}{u_m^3}\|w^R_x(s)\|_{L^\infty}\|w^R_x(s)\|_{L^p}
+\frac{C}{u_m}\|w^R_{xx}(s)\|_{L^p}\\
&\leq C_p\left(\frac{\lambda_R}{u_m^3}+\frac{1}{u_m}\right)
\min\left\{\lambda_R,s^{-1}\right\} \leq C_p\min\left\{\delta_R,s^{-1}\right\}.
\end{aligned}
\end{equation}
In the last lines of \eqref{eq:uniform-rarefaction-first-derivative} and \eqref{eq:uniform-rarefaction-second-derivative}, we used \eqref{lambdaR}; the additional factor \(\lambda_R\) appearing in \eqref{eq:uniform-rarefaction-second-derivative} is uniformly bounded since \(u_m\) is fixed and \(\delta_R/u_m\leq\delta_0\). Thus, the constant \(C_p\) in \eqref{eq:approx-rarefaction-derivatives} can be chosen independently of \(f\).

For \eqref{eq:approx-rarefaction-tails}--\eqref{eq:approx-rarefaction-fan-error}, note first that \(g(f'(u_*))=u_*\) for \(u_*\in\{u_m,u_+\}\). Hence, by the mean value theorem and \eqref{g'g''bd},
\begin{equation}\label{eq:rarefaction-endstate-comparison}
|u^R(t,\xi)-u_*|
\leq \frac{C}{u_m}\left|w^R(s,\xi+\sigma s)-f'(u_*)\right|,
\qquad u_*\in\{u_m,u_+\}.
\end{equation}
Moreover, by \eqref{eq:exact-rarefaction-shock-frame}, for \(0\leq\xi\leq\lambda_Rs\),
\begin{equation*}
f'(u^r(t,\xi))=f'(u_m)+\frac{\xi}{s}.
\end{equation*}
Therefore,
\begin{equation}\label{eq:rarefaction-fan-comparison}
|u^R(t,\xi)-u^r(t,\xi)|
\leq\frac{C}{u_m}
\left|w^R(s,\xi+\sigma s)-f'(u_m)-\frac{\xi}{s}\right|.
\end{equation}
Applying the estimates for \(w^R\) used in the proof of \cite[Lemma~3.1(4)--(6)]{HWZ} to the right-hand sides of \eqref{eq:rarefaction-endstate-comparison} and \eqref{eq:rarefaction-fan-comparison}, and using \eqref{lambdaR}, we find that \(C\) in \eqref{eq:approx-rarefaction-tails} and \(C_\theta\) in \eqref{eq:approx-rarefaction-refined-tails} and \eqref{eq:approx-rarefaction-fan-error} can be chosen uniformly in \(f\). This completes the proof.
\end{proof}

\subsection{Degenerate Hardy--Poincar\'e inequality}

The coercivity analysis uses the following degenerate Hardy--Poincar\'e inequality.

\begin{proposition}[Degenerate Hardy--Poincar\'e inequality]\label{Poincare}
There exists a constant \(c_P>1\) such that
\begin{equation*}
\int_0^1(7-3y)y(1-y)^2\Phi_y^2\,dy + \frac{10}{3}\left(\int_0^1(3-2y)\Phi\,dy\right)^2 \geq c_P \int_0^1(1-y)(12-y)\Phi^2\,dy
\end{equation*}
for any \(\Phi\in H^1_{\mathrm{loc}}(0,1)\cap L^1(0,1)\) satisfying
\begin{equation*}
\int_0^1(7-3y)y(1-y)^2\Phi_y^2\,dy < \infty,  \qquad \int_0^1(1-y)(12-y)\Phi^2\,dy<\infty.
\end{equation*}
\end{proposition}

For general criteria for weighted one-dimensional Poincar\'e inequalities and related Hardy-type inequalities, together with estimates for their optimal constants, we refer to \cite{CW}. For the given degenerate weights, a direct application of these general bounds does not establish the strict bound $c_P>1$ with the coefficient of the rank-one moment term fixed at $10/3$. We establish this bound through quantitative estimates adapted to these weights, thereby obtaining the positive coercivity gap needed in our analysis. The proof of Proposition~\ref{Poincare} requires a separate analysis and is therefore deferred to Appendix~\ref{sec:poincare}.

\section{Statement of the a priori estimate} \label{sec:3}

In this section, we state the a priori estimate and derive the weighted energy identity that underlies its proof.

\subsection{The a priori estimate} \label{sec:3.1}

Under the change of variables \(\xi=x-\sigma t\), the viscous conservation law \eqref{eq:model} takes the form
\begin{equation*}
u_t-\sigma u_\xi+f(u)_\xi=u_{\xi\xi}.
\end{equation*}
We define the perturbation around the reference composite profile \(\widetilde u\) by
\begin{equation}\label{eq:perturbation}
\phi(t,\xi):=u(t,\xi-X(t))-\widetilde u(t,\xi),
\end{equation}
where \(X\) is a shift to be specified below. By \eqref{eq:composite-residual-equation}, the perturbation $\phi$ satisfies
\begin{equation}\label{eq:perturbation-equation}
\phi_t-\sigma\phi_\xi+\left(f(\phi+\widetilde u)-f(\widetilde u)\right)_\xi+\dot X(\phi_\xi+\widetilde u_\xi)-\phi_{\xi\xi}=-F,
\end{equation}
where \(F\) is as in \eqref{eq:residual}.

We introduce the weight function, for a parameter $\eta>0$,
\begin{equation}\label{eq:weight}
a_\eta(t,\xi):=\alpha_\eta-\widetilde u(t,\xi), \qquad \text{where} \quad \alpha_\eta:=(5-\eta)u_m,
\end{equation}
and the associated weighted energy
\begin{equation}\label{eq:weighted-energy}
\mathcal E(t):=\frac12\int_{\mathbb R}a_\eta\phi^2\,d\xi.
\end{equation}
For sufficiently small \(\eta\) and \(\delta_R/u_m\), the bounds \(u_-\leq u^S\leq u_m\) and \(u_m<u^R<u_+\), together with \eqref{eq:perturbed-left-state}, imply
\begin{equation} \label{eq:weight-bounds}
c_\eta u_m\leq a_\eta(t,\xi)\leq C_\eta u_m,
\end{equation}
where \(c_\eta>0\) and \(C_\eta>0\) can be chosen independently of \(f\). Consequently,
\begin{equation} \label{eq:weighted-energy-equivalence}
\frac{c_\eta u_m}{2}\|\phi(t)\|_{L^2(\mathbb R)}^2\leq\mathcal E(t)\leq\frac{C_\eta u_m}{2}\|\phi(t)\|_{L^2(\mathbb R)}^2.
\end{equation}

We define the shift function \(X\) by
\begin{equation}\label{eq:shift-ode}
\dot X(t)=\kappa\mathcal Y(t),\qquad \mathcal Y(t):=\int_{\mathbb R}\widetilde u_\xi\left(a_\eta\phi+\frac12\phi^2\right) \, d\xi,
\end{equation}
with \(X(0)=0\), where \(\kappa>5/u_m\). For a strong solution \(u\) on \([0,T]\), the right-hand side of \eqref{eq:shift-ode}, viewed as a function of \((t,X)\) through \eqref{eq:perturbation}, is continuous in \(t\) and locally Lipschitz in \(X\). Hence, the Picard--Lindel\"of theorem yields a unique local \(C^1\) solution \(X\) of \eqref{eq:shift-ode}, which extends to \([0,T]\) by the boundedness of the right-hand side.

The a priori estimate is stated as follows.

\begin{proposition}[A priori estimate] \label{prop:large-L2-apriori}
Fix \(u_m>0\) and \(\kappa>5/u_m\). There exists \(\eta_0>0\) such that, for every \(0<\eta<\eta_0\), there exist constants \(\varepsilon_0>0\), \(\delta_0>0\), \(c>0\), and \(C>0\) for which the following statement holds.

Let $T>0$. Let \(f\in C^3(\mathbb R)\) satisfy \(\varepsilon_f\leq\varepsilon_0\), let $\widetilde u$ be the reference composite profile described in Section~\ref{sec:1.3}, and let \(u\) be a strong solution of \eqref{eq:model} on \([0,T]\). Let \(X\) solve \eqref{eq:shift-ode} with \(X(0)=0\). If \(0<\delta_R/u_m\leq\delta_0\), then
\begin{equation} \label{eq:large-L2-apriori}
\E(t)+c\int_0^t\D(s)\,ds\leq e^{CQ(t)}\bigl(\E(0)+CQ(t)\bigr)
\end{equation}
for all \(t\in[0,T]\), where $Q$ and $\mathcal{D}$ are defined by
\begin{align}
Q(t)&:=\int_0^t\left[\left(\int_{\mathbb R}|u^S-u_m|u^R_\xi\,d\xi\right)^2+\left(\int_{\mathbb R}|F|\,d\xi\right)^{4/3}\right] \, ds, \label{eq:interaction-size} \\
\mathcal D(t)&:=u_m\int_{\mathbb R}\phi_\xi^2\,d\xi+\int_{\mathbb R}\widetilde u_\xi\bigl(u_m^2\phi^2+\phi^4\bigr)\,d\xi+|\dot X(t)|^2. \label{eq:dissipation}
\end{align}
In particular,
\begin{equation}\label{eq:large-L2-apriori-power}
\E(t)+c\int_0^t\D(s)\,ds\leq e^{C(\delta_R/u_m)^{8/33}}\left(\E(0)+C(\delta_R/u_m)^{8/33}\right)
\end{equation}
for all \(t\in[0,T]\).
\end{proposition}

\subsection{Energy identity}

We begin the proof of Proposition~\ref{prop:large-L2-apriori} by deriving an identity for the weighted energy \(\mathcal E\). We first record the regularity needed for the computation. Since
\begin{equation*}
\widetilde u(t,\cdot)-\widetilde u(0,\cdot)=u^R(t,\cdot)-u^R(0,\cdot),
\end{equation*}
the equation \eqref{eq:rarefaction-frame-equation}, together with the estimates \eqref{eq:approx-rarefaction-derivatives}, yields
\begin{equation*}\widetilde u-\widetilde u(0,\cdot)\in H^1(0,T;L^2(\mathbb R))\cap L^2(0,T;H^2(\mathbb R)).
\end{equation*}
Combining this with the regularity \eqref{strong_regularity} of a strong solution and $X \in C^1([0,T])$, we obtain
\begin{equation*}
\phi\in H^1(0,T;L^2(\mathbb R))\cap L^2(0,T;H^2(\mathbb R)).
\end{equation*}
Together with the smoothness of \(a_\eta\), this implies that \(t\mapsto\mathcal E(t)\) is absolutely continuous and justifies the spatial integrations by parts below for almost every \(t\in(0,T)\).

To state the energy identity in a convenient form, for a flux \(g\) and
\(v,z\in\mathbb R\), we define
\begin{equation}\label{eq:general-flux-nonlinear-form}
\mathcal B_g^\eta(v,z):=(\alpha_\eta-v+z)\bigl(g(v+z)-g(v)\bigr)-\int_0^z\bigl(g(v+s)-g(v)\bigr)\,ds-(\alpha_\eta-v)g'(v)z.
\end{equation}
With this notation, the weighted energy satisfies the following identity.

\begin{lemma}[Energy identity] \label{prop:energy-identity}
Under the assumptions in Proposition~\ref{prop:large-L2-apriori}, it holds that
\begin{equation}\label{eq:weighted-energy-identity}
\begin{split}
\frac{d}{dt}\mathcal E + \int_{\mathbb R}a_\eta\phi_\xi^2 \, d\xi +\frac12 \int_{\mathbb R}\left(\widetilde u_t-\sigma\widetilde u_\xi+\widetilde u_{\xi\xi}\right)\phi^2 \, d\xi +\int_{\mathbb R}\widetilde u_\xi\mathcal B_f^\eta(\widetilde u,\phi) \, d\xi + \kappa \mathcal Y^2 = -\int_{\mathbb R}a_\eta\phi F \, d\xi
\end{split}
\end{equation}
for almost every \(t\in(0,T)\), where $\mathcal{B}_f^\eta$ is defined by \eqref{eq:general-flux-nonlinear-form}.
\end{lemma}

\begin{proof}
We take the \(L^2\)-inner product of \eqref{eq:perturbation-equation} with \(a_\eta\phi\):
\begin{equation} \label{L2inner}
\begin{split}
& \int_{\mathbb R}a_\eta\phi\phi_t \, d\xi -\sigma\int_{\mathbb R}a_\eta\phi\phi_\xi \, d\xi + \int_{\mathbb R}a_\eta\phi\bigl(f(\widetilde u+\phi)-f(\widetilde u)\bigr)_\xi \, d\xi\\
&\quad + \dot X \int_{\mathbb R}a_\eta\phi(\phi_\xi+\widetilde u_\xi) \, d\xi -\int_{\mathbb R}a_\eta\phi\phi_{\xi\xi}\,d\xi = -\int_{\mathbb R}a_\eta\phi F \, d\xi.
\end{split}
\end{equation}
For the third term on the left-hand side, we first integrate by parts:
\begin{equation*}
\begin{split}
I_{\mathrm{conv}} &:=\int_{\mathbb R}a_\eta\phi\bigl(f(\widetilde u+\phi)-f(\widetilde u)\bigr)_\xi\,d\xi \\
&=\int_{\mathbb R}\widetilde u_\xi\phi\bigl(f(\widetilde u+\phi)-f(\widetilde u)\bigr)\,d\xi-\int_{\mathbb R}a_\eta\phi_\xi\bigl(f(\widetilde u+\phi)-f(\widetilde u)\bigr)\,d\xi,
\end{split}
\end{equation*}
where we used \((a_\eta)_\xi=-\widetilde u_\xi\). To handle the second term, we observe that
\begin{equation*}
\begin{split}
\left(\int_0^\phi\bigl(f(\widetilde u+s)-f(\widetilde u)\bigr)\,ds\right)_\xi &= \phi_\xi \bigl(f(\widetilde u+\phi)-f(\widetilde u)\bigr) + \widetilde u_\xi\int_0^\phi \bigl( f'(\widetilde u+s)-f'(\widetilde u)\bigr) \, ds \\
& = \phi_\xi \big( f(\widetilde u+\phi)-f(\widetilde u)\big) + \widetilde u_\xi\bigl(f(\widetilde u+\phi) - f(\widetilde u)-f'(\widetilde u)\phi \bigr).
\end{split}
\end{equation*}
Using this identity, we have
\begin{equation*}
\begin{split}
I_{\mathrm{conv}} &=\int_{\mathbb R}\widetilde u_\xi\phi\bigl(f(\widetilde u+\phi)-f(\widetilde u)\bigr)\,d\xi -\int_{\mathbb R}a_\eta\left(\int_0^\phi\bigl(f(\widetilde u+s)-f(\widetilde u)\bigr)\,ds\right)_\xi \, d\xi \\
&\quad+\int_{\mathbb R}a_\eta\widetilde u_\xi\bigl(f(\widetilde u+\phi)-f(\widetilde u)-f'(\widetilde u)\phi\bigr)\, d\xi.
\end{split}
\end{equation*}
Integrating the second term by parts and using \((a_\eta)_\xi=-\widetilde u_\xi\), we obtain
\begin{equation} \label{Iconv}
\begin{split}
I_{\mathrm{conv}}
&=\int_{\mathbb R}\widetilde u_\xi \bigg[ \phi\bigl(f(\widetilde u+\phi)-f(\widetilde u)\bigr)-\int_0^\phi\bigl(f(\widetilde u+s)-f(\widetilde u)\bigr)\,ds \\
&\qquad \qquad +a_\eta\bigl(f(\widetilde u+\phi)-f(\widetilde u)-f'(\widetilde u)\phi\bigr)\bigg] \, d\xi\\
&=\int_{\mathbb R}\widetilde u_\xi\mathcal B_f^\eta(\widetilde u,\phi)\,d\xi.
\end{split}
\end{equation}
The remaining terms in \eqref{L2inner} are computed using the product rule, integration by parts, \(a_\eta=\alpha_\eta-\widetilde u\), and \eqref{eq:shift-ode}:
\begin{equation} \label{fise}
\int_{\mathbb R}a_\eta\phi\phi_t\,d\xi =\frac{d}{dt}\mathcal E+\frac12\int_{\mathbb R}\widetilde u_t\phi^2\,d\xi, \qquad -\sigma\int_{\mathbb R}a_\eta\phi\phi_\xi\,d\xi =-\frac{\sigma}{2}\int_{\mathbb R}\widetilde u_\xi\phi^2\,d\xi,
\end{equation}
\begin{equation} \label{dXY}
\dot X\int_{\mathbb R}a_\eta\phi(\phi_\xi+\widetilde u_\xi)\,d\xi =\dot X\mathcal Y=\kappa\mathcal Y^2,
\end{equation}
and
\begin{equation} \label{las}
-\int_{\mathbb R}a_\eta\phi\phi_{\xi\xi}\,d\xi =\int_{\mathbb R}a_\eta\phi_\xi^2\,d\xi+\frac12\int_{\mathbb R}\widetilde u_{\xi\xi}\phi^2\,d\xi.
\end{equation}
Substituting \eqref{Iconv}--\eqref{las} into \eqref{L2inner}, we obtain the desired identity \eqref{eq:weighted-energy-identity}.
\end{proof}

\subsection{Decomposition of the energy identity}

We next decompose \eqref{eq:weighted-energy-identity} into shock, rarefaction, interaction, and residual contributions.

To combine the third and fourth terms on the left-hand side of \eqref{eq:weighted-energy-identity}, for a flux \(g\) we define
\begin{equation}\label{eq:general-flux-A-form}
\mathcal A_g^\eta(v,z):=\mathcal B_g^\eta(v,z)+\frac12g'(v)z^2, \qquad v,z \in \mathbb{R}.
\end{equation}
By \eqref{eq:shock-ode}, \eqref{eq:rarefaction-frame-equation}, and \eqref{eq:composite-profile}, we have
\begin{equation*}
\begin{split}
\widetilde u_t-\sigma\widetilde u_\xi+\widetilde u_{\xi\xi}
&=\bigl(f'(u^S)-2\sigma\bigr)u^S_\xi-f'(u^R)u^R_\xi+u^R_{\xi\xi}\\
&=\bigl(f'(\widetilde u)-2\sigma\bigr)u^S_\xi+\bigl(f'(\widetilde u)-2f'(u^R)\bigr)u^R_\xi-F,
\end{split}
\end{equation*}
where the second equality follows from the definition of \(F\) in \eqref{eq:residual}. Using this identity and \eqref{eq:general-flux-A-form}, we have
\begin{equation*}
\begin{split}
&\frac12\int_{\mathbb R}\left(\widetilde u_t-\sigma\widetilde u_\xi+\widetilde u_{\xi\xi}\right)\phi^2\,d\xi+\int_{\mathbb R} \widetilde u_\xi\mathcal B_f^\eta(\widetilde u,\phi)\,d\xi \\
& \quad = \frac{1}{2} \int_\mathbb{R} \left( \bigl(f'(\widetilde u)-2\sigma\bigr)u^S_\xi+\bigl(f'(\widetilde u)-2f'(u^R)\bigr)u^R_\xi-F \right) \phi^2 \, d\xi + \int_\mathbb{R} ( u^S_\xi + u^R_\xi ) \mathcal B_f^\eta(\widetilde u,\phi)\,d\xi \\
& \quad = \int_{\mathbb R} u^S_\xi \left( \mathcal A_f^\eta(\widetilde u,\phi)-\sigma\phi^2\right) \, d\xi+\int_{\mathbb R} u^R_\xi \left( \mathcal A_f^\eta(\widetilde u,\phi)-f'(u^R)\phi^2\right) \, d\xi-\frac12\int_{\mathbb R}F\phi^2\,d\xi.
\end{split}
\end{equation*}
Next, using \eqref{eq:shift-ode}, \eqref{eq:weight}, and \eqref{eq:composite-profile}, we decompose \(\mathcal Y\) as
\begin{equation*}
\mathcal Y=\mathcal Y_S^\eta[u^S;\phi]+\mathcal Z,
\end{equation*}
where, for a sufficiently regular profile \(V\),
\begin{equation}\label{eq:shock-shift-functional}
\mathcal Y_S^\eta[V;\phi]:=\int_{\mathbb R}V_\xi\left((\alpha_\eta-V)\phi+\frac12\phi^2\right) \, d\xi,
\end{equation}
and
\begin{equation*}
\mathcal Z:=-\int_{\mathbb R}(u^R-u_m)u^S_\xi\phi\,d\xi+\int_{\mathbb R}u^R_\xi\left(a_\eta\phi+\frac12\phi^2\right) \, d\xi.
\end{equation*}
Similarly,
\begin{equation*}
\int_{\mathbb R}a_\eta\phi_\xi^2\,d\xi=\int_{\mathbb R}(\alpha_\eta-u^S)\phi_\xi^2\,d\xi-\int_{\mathbb R}(u^R-u_m)\phi_\xi^2\,d\xi.
\end{equation*}
For a flux \(g\) and a sufficiently regular profile \(V\), we define
\begin{equation}\label{eq:perturbed-shock-block}
\begin{split}
\mathcal D_{S,g}^\eta[V;\phi]:={}&\int_{\mathbb R}(\alpha_\eta-V)\phi_\xi^2\,d\xi
+\int_{\mathbb R}V_\xi\left[\mathcal B_g^\eta(V,\phi)+\frac12\bigl(g'(V)-2g'(u_m)\bigr)\phi^2\right] \, d\xi\\
&+\frac5{u_m}\left(\mathcal Y_S^\eta[V;\phi]\right)^2.
\end{split}
\end{equation}
For the rarefaction profile \(u^R\), we define
\begin{equation}\label{eq:perturbed-rarefaction-block}
\mathcal D_{R,g}^\eta[u^R;\phi]:=\int_{\mathbb R}u^R_\xi\left[\mathcal B_g^\eta(u^R,\phi)-\frac12g'(u^R)\phi^2\right] \, d\xi.
\end{equation}
In the subsequent analysis, we refer to \(\mathcal D_{S,g}^\eta\) and \(\mathcal D_{R,g}^\eta\) as the shock and rarefaction blocks, respectively. 

Substituting the above decompositions into \eqref{eq:weighted-energy-identity}, we obtain
\begin{equation}\label{eq:master-decomposition}
\frac{d}{dt}\mathcal E(t)+\mathcal G_f(t)+\mathcal I_f(t)=\mathcal R(t),
\end{equation}
where
\begin{equation}\label{eq:general-good-interaction-remainder}
\begin{split}
\mathcal G_f(t)&:=\mathcal D_{S,f}^\eta[u^S;\phi]+\mathcal D_{R,f}^\eta[u^R;\phi]-\int_{\mathbb R}(u^R-u_m)\phi_\xi^2\,d\xi+\kappa(\mathcal Y_S^\eta[u^S;\phi]+\mathcal Z)^2 \\
&\quad -\frac5{u_m}\left(\mathcal Y_S^\eta[u^S;\phi]\right)^2 +\int_{\mathbb R}u^S_\xi\bigl(\mathcal A_f^\eta(\widetilde u,\phi)-\mathcal A_f^\eta(u^S,\phi)\bigr)\,d\xi,\\
\mathcal I_f(t)&:=\int_{\mathbb R}u^R_\xi\bigl(\mathcal A_f^\eta(\widetilde u,\phi)-\mathcal A_f^\eta(u^R,\phi)\bigr)\,d\xi,\\
\mathcal R(t)&:=-\int_{\mathbb R}F\left(a_\eta\phi-\frac12\phi^2\right) \, d\xi.
\end{split}
\end{equation}
Sections~\ref{sec:4}--\ref{sec:5} provide estimates for $\mathcal{G}_f$, $\mathcal{I}_f$, and $\mathcal{R}$ and complete the proof of Proposition~\ref{prop:large-L2-apriori}.

\section{The main coercivity estimate} \label{sec:4}

This section proves the coercivity estimate for the functional $\mathcal{G}_f$. We first establish the coercivity of the reference shock block associated with the cubic flux $f_0$ and the corresponding shock profile $U^S$, at $\eta=0$, using the degenerate Hardy--Poincar\'e inequality. The resulting estimate is then extended to a small positive $\eta$, the profile $u^S$, and the perturbed flux $f$. The rarefaction block is treated by a pointwise estimate for the cubic flux $f_0$, followed by a perturbation estimate for the flux $f$. Finally, the remaining terms in $\mathcal{G}_f$ are controlled by the smallness of $\delta_R/u_m$ and $\varepsilon_f$, and all the estimates are combined to obtain the desired coercivity.

The following proposition is the main result of this section.

\begin{proposition}\label{prop:main-coercivity}
Under the assumptions in Proposition~\ref{prop:large-L2-apriori}, there exists a constant \(c>0\), independent of $f$, such that
\begin{equation}\label{eq:perturbed-good-term}
\mathcal G_f(t)\geq c\left(u_m\int_{\mathbb R}\phi_\xi^2\,d\xi+\int_{\mathbb R}\widetilde u_\xi\left(u_m^2\phi^2+\phi^4\right) \, d\xi + |\dot X(t)|^2\right)
\end{equation}
for all \(t\in[0,T]\).
\end{proposition}

The remainder of this section is devoted to the proof of Proposition~\ref{prop:main-coercivity}.

\subsection{Coercivity of the reference shock block}

In this and the next subsection, the analysis is carried out for arbitrary \(\phi\in H^1(\mathbb R)\). We define
\begin{equation*}
\begin{split}
\mathcal D_{S,f_0}^{0}[U^S;\phi]:={}&\int_{\mathbb R}(5u_m-U^S)\phi_\xi^2\,d\xi\\
&+\int_{\mathbb R}U^S_\xi\left((15u_mU^S-3u_m^2)\phi^2+(5u_m+U^S)\phi^3+\frac34\phi^4\right) \, d\xi\\
&+\frac5{u_m}\left(\mathcal Y_S^0[U^S;\phi]\right)^2,
\end{split}
\end{equation*}
where
\begin{equation} \label{def_YS0}
\mathcal Y_S^0[U^S;\phi]:=\int_{\mathbb R}U^S_\xi\left((5u_m-U^S)\phi+\frac12\phi^2\right) \, d\xi.
\end{equation}
The proposition below provides the coercivity of the reference shock block $\mathcal{D}_{S,f_0}^0[U^S;\phi]$.

\begin{proposition}\label{prop:reference-shock-coercivity}
Fix \(u_m>0\). There exists a constant \(c>0\) such that
\begin{equation}\label{eq:reference-shock-coercivity-statement}
\mathcal D_{S,f_0}^{0}[U^S;\phi]\geq c\left(u_m\int_{\mathbb R}\phi_\xi^2\,d\xi+\int_{\mathbb R}U^S_\xi\left(u_m^2\phi^2+\phi^4\right)d\xi\right)
\end{equation}
for any \(\phi\in H^1(\mathbb R)\).
\end{proposition}

To establish this proposition, we proceed in two steps. We first rewrite the functional $\mathcal D_{S,f_0}^{0}[U^S;\phi]$ in shock-profile coordinates and isolate the quadratic form to which the degenerate Hardy--Poincar\'e inequality applies; we then recover control of the unweighted \(L^2\)- and \(L^4\)-norms needed for the full coercivity estimate.

\subsubsection{Profile coordinates and nonlinear decomposition}

We introduce the shock-profile coordinate
\begin{equation} \label{eq:profile-coordinate}
y:=\frac{U^S-U_-}{u_m-U_-}=\frac{U^S+2u_m}{3u_m}\in(0,1),\qquad \Phi(y):=\frac{\phi(\xi(y))}{u_m}.
\end{equation}
Then, using \eqref{eq:cubic-shock-first-order} and $U_-=-2u_m$,
\begin{equation} \label{eq:profile-jacobian}
\begin{split}
y_\xi &= \frac{U^S_\xi}{3u_m} = \frac{1}{3u_m} (U^S-U_-)(U^S-u_m)^2 = 9u_m^2 y (1-y)^2, \qquad U^S_\xi d \xi = 3u_m dy.
\end{split}
\end{equation}
Under this change of coordinates, we have
\begin{equation*}
\mathcal D_{S,f_0}^{0}[U^S;\phi] = 9 u_m^5 \mathcal{F}[\Phi],
\end{equation*}
where
\begin{equation*}
\begin{aligned}
\mathcal F[\Phi] & :=\int_0^1(7-3y)y(1-y)^2\Phi_y^2 \, dy+\int_0^1\left((15y-11)\Phi^2+(1+y)\Phi^3+\frac14\Phi^4\right) \, dy \\
&\quad +5\left(\int_0^1\left((7-3y)\Phi+\frac12\Phi^2\right) \, dy \right)^2.
\end{aligned}
\end{equation*}
Here we used the identities
\begin{equation*}
5u_m - U^S = u_m (7-3y), \quad 15u_m U^S - \sigma_0 = 3u_m^2 (15 y - 11),  \quad 5u_m +U^S = 3u_m (1+y).
\end{equation*}

The following exact decomposition separates the quadratic part of \(\mathcal F\) from two nonnegative nonlinear remainder terms.

\begin{lemma}\label{lem:reference-shock-nonlinear-decomposition}
For \(\phi\in H^1(\mathbb R)\), let \(\Phi\) be defined by \eqref{eq:profile-coordinate}. Then
\begin{equation}\label{eq:cubic-full-nonlinear-decomposition}
\mathcal F[\Phi]=\mathcal H[\Phi]+\int_0^1(g-G)^2 \, dy + 6 \left( G + \frac53 M \right)^2,
\end{equation}
where \(\mathcal H\), \(g\), \(G\), and \(M\) are defined by
\begin{equation*}
\mathcal H [\Phi]:=\int_0^1(7-3y)y(1-y)^2\Phi_y^2 \, dy-\int_0^1(1-y)(12-y)\Phi^2 \, dy +\frac{10}{3} \left(\int_0^1(3-2y)\Phi\, dy \right)^2,
\end{equation*}
and
\begin{equation} \label{gGM}
g(y):=\frac12\Phi(y)^2+(1+y)\Phi(y),\qquad G:=\int_0^1g(y)\, d y,\qquad M:=\int_0^1(3-2y)\Phi(y)\, dy.
\end{equation}
\end{lemma}

\begin{proof}
Completing the square in the integrand of the second term of $\mathcal{F}[\Phi]$, we obtain
\begin{equation*}
\begin{split}
&\int_0^1\left((15y-11)\Phi^2+(1+y)\Phi^3+\frac14\Phi^4\right)\,dy\\
&\quad= -\int_0^1(1-y)(12-y)\Phi^2\,dy+\int_0^1\left(\frac12\Phi^2+(1+y)\Phi\right)^2\,dy.
\end{split}
\end{equation*}
For the third term of $\mathcal{F}[\Phi]$, we have
\begin{equation*}
\int_0^1\left((7-3y)\Phi+\frac12\Phi^2\right)\,dy = G+2M.
\end{equation*}
Hence,
\begin{equation*}
\begin{aligned}
\mathcal F[\Phi]&=\int_0^1(7-3y)y(1-y)^2\Phi_y^2\,dy-\int_0^1(1-y)(12-y)\Phi^2\,dy\\
&\quad+\int_0^1g^2\,dy+5(G+2M)^2.
\end{aligned}
\end{equation*}
To rewrite the last two terms, we use $\int_0^1g^2 \, dy = \int_0^1 (g-G)^2\,dy + G^2$:
\begin{equation*}
\begin{split}
\int_0^1g^2\,dy+5(G+2M)^2&=\int_0^1(g-G)^2\,dy+G^2+5(G+2M)^2\\
&=\int_0^1(g-G)^2\,dy+6\left(G+\frac53M\right)^2+\frac{10}{3}M^2.
\end{split}
\end{equation*}
Combining the last two identities yields \eqref{eq:cubic-full-nonlinear-decomposition}.
\end{proof}

\begin{remark}
The calculation in the proof of Lemma~\ref{lem:reference-shock-nonlinear-decomposition} explains the choice of \(5\) in \(\alpha_\eta=(5-\eta)u_m\). For the cubic flux at \(\eta=0\), this choice yields the simple polynomial weights
\begin{equation*}
(7-3y)y(1-y)^2,\qquad (1-y)(12-y),\qquad 3-2y
\end{equation*}
in the quadratic form \(\mathcal H[\Phi]\). The value \(5\) is used for this algebraic simplification; changing it would alter the reduced weights and require a separate Hardy--Poincar\'e analysis.
\end{remark}

\subsubsection{Completion of the coercivity estimate}

By Proposition~\ref{Poincare} and the definition of \(\mathcal H[\Phi]\), we have
\begin{equation} \label{eq:new-poincare}
\mathcal H[\Phi] \geq c \left( \int_0^1 (7-3y)y(1-y)^2 \Phi_y^2 \, dy + \int_0^1 (1-y)(12-y) \Phi^2 \, dy \right)
\end{equation}
for some $c >0$. The weight in the second integral degenerates at \(y=1\) and therefore does not directly control the unweighted \(L^2\)-norm. Consequently, the coercivity of the reference shock block $\mathcal D_{S,f_0}^{0}[U^S;\phi]$ reduces to recovering the unweighted \(L^2\)- and \(L^4\)-norms of \(\Phi\) from the quadratic form \(\mathcal H\) and the nonlinear remainder terms in \(\mathcal F\).

\begin{lemma} \label{lem:shock-endpoint-recovery}
There exists a constant \(C>0\) such that
\begin{equation} \label{eq:F-controls-Psi-L2}
\int_0^1 \Phi^2 \, dy \leq C \mathcal H[\Phi]
\end{equation}
for any \(\phi\in H^1(\mathbb R)\), with \(\Phi\) defined by \eqref{eq:profile-coordinate}.
\end{lemma}

\begin{proof}
For \(0<\varepsilon<1/2\), we write
\begin{equation} \label{L2con_1}
\int_0^{1-\varepsilon}\Phi^2\,dy=\int_0^{1/2}\Phi^2\,dy+\int_{1/2}^{1-\varepsilon}\Phi^2\,dy.
\end{equation}
We estimate the two terms on the right-hand side uniformly in \(\varepsilon\), and then let \(\varepsilon\downarrow0\).

For the first term, the bound
\begin{equation*}
(1-y)(12-y)\ge\frac{23}{4} \qquad\text{on }\left[0,1/2\right]
\end{equation*}
yields
\begin{equation} \label{L2con_11}
\int_0^{1/2}\Phi^2\,dy \le \frac4{23} \int_0^1(1-y)(12-y)\Phi^2\,dy.
\end{equation}
Next, we observe that
\begin{equation*}
\left((1-y)\Phi^2\right)_y=-\Phi^2+2(1-y)\Phi\Phi_y.
\end{equation*}
Integrating over \([1/2,1-\varepsilon]\), we obtain
\begin{equation*}
\begin{split}
\int_{1/2}^{1-\varepsilon}\Phi^2\,dy + \varepsilon\Phi^2(1-\varepsilon) & = \frac12 \Phi\left(\frac12\right)^2 + 2\int_{1/2}^{1-\varepsilon} (1-y)\Phi\Phi_y\,dy.
\end{split}
\end{equation*}
By Young's inequality, the last term is estimated as
\begin{equation*}
2 \int_{1/2}^{(1-\varepsilon)} (1-y)\Phi\Phi_y \, dy \leq \frac12 \int_{1/2}^{(1-\varepsilon)} \Phi^2 \, dy + 2 \int_{1/2}^{(1-\varepsilon)} (1-y)^2\Phi_y^2 \, dy.
\end{equation*}
Thus, we have
\begin{equation} \label{L2con_12}
\int_{1/2}^{1-\varepsilon }\Phi^2 \,dy \leq \Phi\left(\frac12\right)^2 + 4\int_{1/2}^{1-\varepsilon}(1-y)^2\Phi_y^2\,dy.
\end{equation}
It remains to control the term $\Phi^2(1/2)$. Let $J:=\left(\frac14,\frac12\right)$. By the fundamental theorem of calculus,
\begin{equation*}
\Phi\left(\frac12\right) = \Phi(z)+\int_z^{1/2}\Phi_y(s)\,ds, \qquad  z\in J.
\end{equation*}
Using $(a+b)^2 \leq 2a^2 + 2b^2$ and applying the Cauchy--Schwarz inequality, we have
\begin{equation*}
\begin{split}
\Phi\left(\frac12\right)^2  & \le 2 \Phi(z)^2 + 2 \bigg( \int_z^{1/2} \Phi_y(s) \, ds \bigg)^2 \\
& \leq 2 \Phi(z)^2 + 2 \left( \frac{1}{2} - z \right) \int_z^{1/2} \Phi_y(s)^2 \, ds \\
& \leq 2 \Phi(z)^2 + \frac{1}{2} \int_J \Phi_y^2 \, dy.
\end{split}
\end{equation*}
Integrating both sides over $J$ and using $|J| = 1/4$, we obtain
\begin{equation} \label{L2con_13}
\Phi \left(\frac12 \right)^2 \leq 8 \int_J \Phi^2 \, dy + \frac12 \int_J \Phi_y^2 \, dy.
\end{equation}

Combining \eqref{L2con_1}--\eqref{L2con_13}, we have
\begin{equation} \label{combPhiL2}
\int_0^{1-\varepsilon}\Phi^2\,dy \le \frac4{23} \int_0^1(1-y)(12-y)\Phi^2\,dy  + 8 \int_J \Phi^2 \, dy + \frac12 \int_J \Phi_y^2 \, dy + 4\int_{1/2}^{1-\varepsilon}(1-y)^2\Phi_y^2\,dy.
\end{equation}
To obtain a bound uniform in $\varepsilon$, we estimate the last three terms by the weighted integrals appearing in $\mathcal{H}[\Phi]$. Since \(J\subset(0,1/2)\), \eqref{L2con_11} gives
\begin{equation*}
8\int_J\Phi^2\,dy\leq\frac{32}{23}\int_0^1(1-y)(12-y)\Phi^2\,dy.
\end{equation*}
Moreover, the bounds
\begin{equation*}
(7-3y)y(1-y)^2\geq\frac{11}{32}\quad\text{on }J,\qquad (7-3y)y\geq\frac{11}{4}\quad\text{on }[1/2,1]
\end{equation*}
yield
\begin{equation*}
\frac12\int_J\Phi_y^2\,dy+4\int_{1/2}^{1-\varepsilon}(1-y)^2\Phi_y^2\,dy\leq\frac{32}{11}\int_0^1(7-3y)y(1-y)^2\Phi_y^2\,dy.
\end{equation*}
Substituting these estimates into \eqref{combPhiL2}, we obtain the uniform bound
\begin{equation*}
\int_0^{1-\varepsilon}\Phi^2\,dy \leq \frac{36}{23}\int_0^1(1-y)(12-y)\Phi^2\,dy + \frac{32}{11}\int_0^1(7-3y)y(1-y)^2\Phi_y^2\,dy.
\end{equation*}
Hence, by letting \(\varepsilon\downarrow0\) and applying \eqref{eq:new-poincare}, we arrive at \eqref{eq:F-controls-Psi-L2}.
\end{proof}

We next use the nonlinear remainder terms in \(\mathcal F\) to control the \(L^4\)-norm.

\begin{lemma}
There exists a constant \(C>0\) such that
\begin{equation}\label{eq:F-controls-Psi-L4}
\int_0^1 \Phi^4 \, dy \leq C\mathcal F[\Phi]
\end{equation}
for any \(\phi\in H^1(\mathbb R)\), with \(\Phi\) defined by \eqref{eq:profile-coordinate}.
\end{lemma}

\begin{proof}
Recall from \eqref{gGM} that
\begin{equation*}
g(y) = \frac12 \Phi^2+(1+y)\Phi, \qquad G =\int_0^1 g(y) \, dy, \qquad M =\int_0^1(3-2y)\Phi \, dy.
\end{equation*}
The definition of \(g\) gives
\begin{equation*}
\Phi^4 =4\left(g-(1+y)\Phi\right)^2 \leq 8g^2+8(1+y)^2\Phi^2 \leq 8g^2+32\Phi^2.
\end{equation*}
Thus, in view of \eqref{eq:F-controls-Psi-L2} and \eqref{eq:cubic-full-nonlinear-decomposition}, it suffices to show that
\begin{equation}\label{eq:F-controls-g-L2}
\int_0^1g^2\,dy\leq C\mathcal F[\Phi].
\end{equation}

We first observe that
\begin{equation*}
\int_0^1 g^2\,dy=\int_0^1(g-G)^2\,dy+G^2
\end{equation*}
and, by \eqref{eq:cubic-full-nonlinear-decomposition},
\begin{equation*}
\int_0^1(g-G)^2\,dy\leq\mathcal F[\Phi], \qquad \left(G+\frac53M\right)^2\leq\frac16\mathcal F[\Phi].
\end{equation*}
Using these relations and Young's inequality, we obtain
\begin{equation*}
\int_0^1 g^2 \, dy \leq \mathcal{F}[\Phi] + \left[ \left(G+\frac53M\right)-\frac53M \right]^2 \leq C \left( \mathcal{F}[\Phi] + M^2 \right).
\end{equation*}
The remaining term $M^2$ is estimated using the definition of \(M\) and the Cauchy--Schwarz inequality:
\begin{equation*}
M^2 \leq\left(\int_0^1(3-2y)^2\,dy\right) \left(\int_0^1\Phi^2\,dy\right) =\frac{13}{3}\int_0^1\Phi^2\,dy \leq C\mathcal F[\Phi].
\end{equation*}
Combining the last two estimates yields \eqref{eq:F-controls-g-L2}, and hence \eqref{eq:F-controls-Psi-L4}.
\end{proof}

\begin{proof}[Proof of Proposition~\ref{prop:reference-shock-coercivity}]
Together with \eqref{eq:new-poincare}, the estimates \eqref{eq:F-controls-Psi-L2} and \eqref{eq:F-controls-Psi-L4} show that there exists a constant \(c>0\) such that
\begin{equation}\label{eq:nonlinear-profile-coercivity}
\mathcal F[\Phi]\geq c \left(\int_0^1(7-3y)y(1-y)^2\Phi_y^2\,dy+\int_0^1\Phi^2\,dy+\int_0^1\Phi^4\,dy\right).
\end{equation}
We now return to the original variables. By \eqref{eq:profile-coordinate}--\eqref{eq:profile-jacobian},
\begin{equation*}
\begin{split}
& 9u_m^5\int_0^1(7-3y)y(1-y)^2\Phi_y^2\,dy= \int_{\mathbb R}(5u_m-U^S)\phi_\xi^2\,d\xi, \\
& 9u_m^5\int_0^1\Phi^2\,dy=3\int_{\mathbb R}U^S_\xi u_m^2\phi^2\,d\xi,
\end{split}
\end{equation*}
and
\begin{equation*}
9u_m^5\int_0^1\Phi^4\,dy=3\int_{\mathbb R}U^S_\xi\phi^4\,d\xi.
\end{equation*}
Since \(\mathcal D_{S,f_0}^{0}[U^S;\phi]=9u_m^5\mathcal F[\Phi]\), the estimate \eqref{eq:nonlinear-profile-coercivity} gives
\begin{equation*}
\mathcal D_{S,f_0}^{0}[U^S;\phi]\geq c\left(\int_{\mathbb R}(5u_m-U^S)\phi_\xi^2\,d\xi+\int_{\mathbb R}U^S_\xi\left(u_m^2\phi^2+\phi^4\right)\,d\xi\right)
\end{equation*}
for some \(c>0\). Since \(-2u_m\leq U^S\leq u_m\), one has \(5u_m-U^S\geq4u_m\), which proves \eqref{eq:reference-shock-coercivity-statement}.
\end{proof}

\subsection{Extension of the reference shock coercivity to positive \texorpdfstring{\(\eta\)}{eta}}

The next lemma quantifies the change in the reference shock block when \(\eta=0\) is replaced by a small positive \(\eta\).

\begin{lemma}\label{lem:exact-cubic-shock-eta-coercivity}
There exists a constant \(C>0\) such that
\begin{equation}\label{eq:exact-cubic-shock-eta-error}
\left|\mathcal D_{S,f_0}^{\eta}[U^S;\phi]-\mathcal D_{S,f_0}^{0}[U^S;\phi]\right|\leq C\eta\left(u_m\int_{\mathbb R}\phi_\xi^2\,d\xi+\int_{\mathbb R}U^S_\xi\left(u_m^2\phi^2+\phi^4\right)d\xi\right)
\end{equation}
for any \(\phi\in H^1(\mathbb R)\) and \(0\leq\eta\leq1\).
\end{lemma}

\begin{proof}
By \eqref{eq:shock-shift-functional} and \(\alpha_\eta=(5-\eta)u_m\),
\begin{equation*}
\mathcal Y_S^\eta[U^S;\phi] = \mathcal Y_S^0[U^S;\phi] -\eta u_m\int_{\mathbb R}U^S_\xi\phi\,d\xi.
\end{equation*}
Using this identity in \eqref{eq:perturbed-shock-block}, we obtain
\begin{equation} \label{temp_id}
\begin{split}
\mathcal D_{S,f_0}^{\eta}[U^S;\phi]-\mathcal D_{S,f_0}^{0}[U^S;\phi] &= -\eta u_m\int_{\mathbb R}\phi_\xi^2 \, d\xi -\eta u_m \int_{\mathbb R} U^S_\xi \left(3U^S\phi^2+\phi^3\right) \, d\xi\\
&\quad - 10 \eta \left(\int_{\mathbb R} U^S_\xi \phi \, d\xi\right) \mathcal Y_S^0[U^S;\phi]  + 5\eta^2u_m \left(\int_{\mathbb R} U^S_\xi \phi \, d\xi \right)^2.
\end{split}
\end{equation}

The second term on the right-hand side is estimated using \(|U^S|\leq2u_m\) and \(u_m|\phi|^3\leq(u_m^2\phi^2+\phi^4)/2\) as follows:
\begin{equation} \label{secoterm}
\left| \eta u_m \int_{\mathbb R} U^S_\xi \left(3U^S\phi^2+\phi^3\right) \, d\xi \right| \leq C \eta \int_{\mathbb R} U^S_\xi \left( u_m^2\phi^2+\phi^4\right) \, d\xi.
\end{equation}
We next estimate the third term. By Young's inequality,
\begin{equation*}
\left|\left(\int_{\mathbb R}U^S_\xi\phi\,d\xi\right)\mathcal Y_S^0[U^S;\phi]\right|
\leq \frac{u_m}{2} \left( \int_\mathbb{R} U^S_\xi \phi \, d\xi \right)^2 + \frac{\left| \mathcal{Y}_S^0[U^S;\phi] \right|^2}{2u_m}.
\end{equation*}
Since $\int_{\mathbb R}U^S_\xi\,d\xi=3u_m$, the Cauchy--Schwarz inequality gives
\begin{equation} \label{um2USxi}
u_m\left(\int_{\mathbb R}U^S_\xi\phi\,d\xi\right)^2 \leq u_m \left( \int_\mathbb{R} U^S_\xi \, d\xi \right) \left( \int_\mathbb{R} U^S_\xi \phi^2 \, d\xi \right) = 3u_m^2\int_{\mathbb R}U^S_\xi\phi^2\,d\xi.
\end{equation}
On the other hand, by \eqref{def_YS0}, the Cauchy--Schwarz inequality, and $|5u_m - U^S| \leq Cu_m$,
\begin{equation*}
\begin{split}
\frac{\left|\mathcal Y_S^0[U^S;\phi]\right|^2}{u_m} &= \frac{1}{u_m} \bigg[ \int_\mathbb{R} U^S_\xi \left( (5u_m - U^S) \phi + \frac12 \phi^2 \right) \, d\xi \bigg]^2 \\
& \leq \frac{1}{u_m} \left( \int_\mathbb{R} U^S_\xi \, d\xi \right) \bigg[ \int_\mathbb{R} U^S_\xi \left( (5u_m - U^S) \phi + \frac12 \phi^2 \right)^2 \, d\xi \bigg] \\
& \leq C \int_\mathbb{R} U^S_\xi (u_m^2 \phi^2 + \phi^4) \, d\xi.
\end{split}
\end{equation*}
We therefore obtain
\begin{equation} \label{tempthirdterm}
\left|\left(\int_{\mathbb R}U^S_\xi\phi\,d\xi\right)\mathcal Y_S^0[U^S;\phi]\right|
\leq C\int_\mathbb{R} U^S_\xi (u_m^2 \phi^2 + \phi^4) \, d\xi.
\end{equation}
Finally, since \(0\leq\eta\leq1\), \eqref{um2USxi} also gives
\begin{equation*} 
\eta^2u_m \left(\int_{\mathbb R}U^S_\xi\phi\,d\xi\right)^2 \leq C\eta\int_{\mathbb R}U^S_\xi u_m^2\phi^2\,d\xi.
\end{equation*}
Combining \eqref{temp_id}, \eqref{secoterm}, and \eqref{tempthirdterm}, we obtain \eqref{eq:exact-cubic-shock-eta-error}.
\end{proof}

\subsection{Coercivity of the rarefaction block for the cubic flux}

\begin{lemma}\label{lem:cubic-rarefaction-coercivity}
Under the assumptions in Proposition~\ref{prop:large-L2-apriori}, there exists a constant \(c_\eta>0\) such that
\begin{equation}\label{eq:rarefaction-coercivity}
\mathcal D_{R,f_0}^{\eta}[u^R;\phi]\geq c_\eta \int_{\mathbb R}u^R_\xi\left(u_m^2\phi^2+\phi^4\right) \, d\xi
\end{equation}
for all \(t\in[0,T]\).
\end{lemma}

\begin{proof}
By the definition \eqref{eq:perturbed-rarefaction-block} with \(f_0(u)=u^3\), we have
\begin{equation*}
\mathcal D_{R,f_0}^{\eta}[u^R;\phi] =\int_{\mathbb R}u^R_\xi\phi^2 \left(\frac34\phi^2+(\alpha_\eta+u^R)\phi +3u^R(\alpha_\eta-u^R)\right) \, d\xi.
\end{equation*}
For the quadratic factor in the integrand, completing the square gives
\begin{equation}\label{eq:rarefaction-completion}
\frac34\phi^2+(\alpha_\eta+u^R)\phi+3u^R(\alpha_\eta-u^R) =\frac34\left(\phi+\frac{2(\alpha_\eta+u^R)}3\right)^2 +\frac{(5u^R-\alpha_\eta)(\alpha_\eta-2u^R)}3.
\end{equation}
Since \(u_m\leq u^R\leq u_m+\delta_R\),
\begin{equation*}
5u^R-\alpha_\eta\geq\eta u_m,
\qquad
\alpha_\eta-2u^R
\geq(3-\eta)u_m-2\delta_R
\geq c u_m,
\end{equation*}
provided that \(\delta_R/u_m\) is sufficiently small. Hence we have
\begin{equation*}
\frac{(5u^R-\alpha_\eta)(\alpha_\eta-2u^R)}3 \geq \frac{c}{3} \eta u_m^2.
\end{equation*}
Moreover, since \(\alpha_\eta+u^R\leq Cu_m\), Young's inequality gives, for any \(0<\gamma<1\),
\begin{equation*}
\left(\phi+\frac{2(\alpha_\eta+u^R)}3\right)^2 \geq (1-\gamma)\phi^2 -C\frac{1-\gamma}{\gamma}u_m^2.
\end{equation*}
We choose \(\gamma\) sufficiently close to \(1\) so that $C\frac{1-\gamma}{\gamma}\leq\frac{c}{6}\eta$. Then, by \eqref{eq:rarefaction-completion},
\begin{equation*}
\begin{split}
& \frac34\phi^2+(\alpha_\eta+u^R)\phi+3u^R(\alpha_\eta-u^R) \\
&\quad \geq \frac34(1-\gamma)\phi^2 +\left(\frac{c}{3}\eta-\frac34C\frac{1-\gamma}{\gamma}\right)u_m^2 \geq \frac34(1-\gamma)\phi^2+\frac{5c}{24}\eta u_m^2 \geq c_\eta\left(u_m^2+\phi^2\right)
\end{split}
\end{equation*}
for some \(c_\eta>0\). Substituting this bound into the expression for \(\mathcal D_{R,f_0}^{\eta}[u^R;\phi]\) and using \(u^R_\xi\geq0\) yields \eqref{eq:rarefaction-coercivity}.
\end{proof}

\subsection{Perturbation estimates for the shock and rarefaction blocks}

Let
\begin{equation}\label{eq:perturbed-shock-coordinate}
L:=u_m-u_-,\qquad y:=\frac{u^S-u_-}{L},\qquad \Phi(y):=\frac{\phi(\xi(y))}{u_m},
\end{equation}
and let
\begin{equation}\label{eq:reference-shock-pullback}
y_0(\xi):=\frac{U^S(\xi)-U_-}{3u_m},\qquad
\phi^U(\xi):=u_m\Phi(y_0(\xi)).
\end{equation}
By \eqref{eq:profile-jacobian}, \eqref{eq:perturbed-shock-ode}, and \eqref{eq:perturbed-factor-bound}, we have \(\phi^U\in H^1(\mathbb R)\) whenever \(\phi\in H^1(\mathbb R)\).

To organize the perturbation argument, we use the following decompositions:
\begin{equation} \label{eq:shock-perturbative-decomposition}
\begin{split}
\mathcal D_{S,f}^{\eta}[u^S;\phi] & = \mathcal D_{S,f_0}^{0}[U^S;\phi^U] +\left(\mathcal D_{S,f_0}^{\eta}[U^S;\phi^U]-\mathcal D_{S,f_0}^{0}[U^S;\phi^U]\right)\\
& \quad +\left(\mathcal D_{S,f_0}^{\eta}[u^S;\phi]-\mathcal D_{S,f_0}^{\eta}[U^S;\phi^U]\right) +\left(\mathcal D_{S,f}^{\eta}[u^S;\phi]-\mathcal D_{S,f_0}^{\eta}[u^S;\phi]\right)
\end{split}
\end{equation}
and
\begin{equation}\label{eq:rarefaction-perturbative-decomposition}
\mathcal D_{R,f}^{\eta}[u^R;\phi] =\mathcal D_{R,f_0}^{\eta}[u^R;\phi] +\left(\mathcal D_{R,f}^{\eta}[u^R;\phi]-\mathcal D_{R,f_0}^{\eta}[u^R;\phi]\right).
\end{equation}
The first two terms on the right-hand side of \eqref{eq:shock-perturbative-decomposition} are controlled by Proposition~\ref{prop:reference-shock-coercivity} and Lemma~\ref{lem:exact-cubic-shock-eta-coercivity}, respectively, while Lemma~\ref{lem:cubic-rarefaction-coercivity} controls the reference rarefaction block. It therefore remains to estimate the perturbations of the shock profile and the flux; these are treated separately below.

\subsubsection{Perturbation of the shock profile}

We first estimate the effect of replacing the reference profile \(U^S\) by \(u^S\) while keeping the cubic flux fixed.

\begin{lemma}\label{lem:shock-profile-perturbation}
Under the assumptions in Proposition~\ref{prop:large-L2-apriori}, there exists a constant \(C>0\) such that
\begin{equation}\label{eq:shock-profile-functional-perturbation}
\left|\mathcal D_{S,f_0}^{\eta}[u^S;\phi]-\mathcal D_{S,f_0}^{\eta}[U^S;\phi^U]\right|
\leq C\varepsilon_f\left(u_m\int_{\mathbb R}\phi_\xi^2 \, d\xi + \int_{\mathbb R}u^S_\xi\left(u_m^2\phi^2+\phi^4\right) \, d\xi \right)
\end{equation}
for all \(t\in[0,T]\), where \(\phi^U\) is defined in \eqref{eq:reference-shock-pullback}.
\end{lemma}

\begin{proof}
We compare the two functionals in \eqref{eq:shock-profile-functional-perturbation} by expressing both in terms of the same function \(\Phi\) on \((0,1)\), using the coordinates introduced in \eqref{eq:perturbed-shock-coordinate} and \eqref{eq:reference-shock-pullback}. Set
\begin{equation*}
\ell:=\frac{L}{3u_m},\qquad \vartheta(y):=\frac{u^S(\xi(y))}{u_m},\qquad \beta(y):=b(u^S(\xi(y))),\qquad \vartheta_0(y):=3y-2.
\end{equation*}
From \eqref{eq:perturbed-shock-coordinate} and \eqref{eq:perturbed-shock-ode}, we have
\begin{equation}\label{eq:perturbed-profile-coordinate-jacobian}
y_\xi=L^2y(1-y)^2\beta(y),\qquad u^S_\xi\,d\xi=L\,dy.
\end{equation}
Moreover, \eqref{eq:perturbed-left-state} and \eqref{eq:perturbed-factor-bound} give
\begin{equation}\label{eq:perturbed-profile-coordinate-closeness}
\ell=1+O(\varepsilon_f),\qquad \vartheta=\vartheta_0+O(\varepsilon_f),\qquad \beta=1+O(\varepsilon_f)
\end{equation}
uniformly on \((0,1)\).

To express the two functionals in their respective profile coordinates, we first write \eqref{eq:perturbed-shock-block} explicitly for \(f_0(u)=u^3\):
\begin{equation}\label{eq:explicit-cubic-shock-block}
\begin{split}
\mathcal D_{S,f_0}^{\eta}[V;\psi] & = \int_{\mathbb R}(\alpha_\eta-V)\psi_\xi^2 \, d\xi +\int_{\mathbb R}V_\xi \left[ \left(3\alpha_\eta V-3u_m^2\right)\psi^2 +(\alpha_\eta+V)\psi^3 +\frac34\psi^4 \right] \, d\xi \\
& \quad +\frac5{u_m}\left(\mathcal Y_S^\eta[V;\psi]\right)^2.
\end{split}
\end{equation}
Set $\widehat\alpha_\eta:=\alpha_\eta/u_m=5-\eta$ and
\begin{equation*}
A:=\int_0^1 \left[ (\widehat\alpha_\eta-\vartheta)\Phi+\frac12\Phi^2 \right] \, dy,
\qquad A_0:=\int_0^1 \left[ (\widehat\alpha_\eta-\vartheta_0)\Phi+\frac12\Phi^2 \right] \, dy.
\end{equation*}
We apply the formula \eqref{eq:explicit-cubic-shock-block} with \((V,\psi)=(u^S,\phi)\). Using \eqref{eq:perturbed-profile-coordinate-jacobian}, we obtain
\begin{equation}\label{eq:perturbed-cubic-block-coordinate}
\begin{split}
\frac{\mathcal D_{S,f_0}^{\eta}[u^S;\phi]}{9u_m^5} & = \int_0^1 \ell^2(\widehat\alpha_\eta-\vartheta)\beta y(1-y)^2\Phi_y^2 \, dy\\
& \quad +\int_0^1 \left[ \ell(\widehat\alpha_\eta\vartheta-1)\Phi^2 +\frac{\ell}{3}(\widehat\alpha_\eta+\vartheta)\Phi^3 +\frac{\ell}{4}\Phi^4 \right] \, dy +5\ell^2A^2.
\end{split}
\end{equation}
For the reference profile $U^S$, the definition of \(\phi^U\) in \eqref{eq:reference-shock-pullback} ensures that the same function \(\Phi\) appears in the coordinate \(y_0\). Thus, using \eqref{eq:profile-jacobian} and renaming \(y_0\) as \(y\) in the integrals, we similarly obtain
\begin{equation}\label{eq:reference-cubic-block-coordinate}
\begin{split}
\frac{\mathcal D_{S,f_0}^{\eta}[U^S;\phi^U]}{9u_m^5} & = \int_0^1 (\widehat\alpha_\eta-\vartheta_0) y(1-y)^2\Phi_y^2 \, dy\\
& \quad +\int_0^1 \left[ (\widehat\alpha_\eta\vartheta_0-1)\Phi^2 +\frac13(\widehat\alpha_\eta+\vartheta_0)\Phi^3 +\frac14\Phi^4 \right] \, dy +5A_0^2.
\end{split}
\end{equation}

Subtracting \eqref{eq:reference-cubic-block-coordinate} from
\eqref{eq:perturbed-cubic-block-coordinate}, we find
\begin{equation}\label{eq:shock-profile-coordinate-difference}
\begin{split}
& \frac{ \mathcal D_{S,f_0}^{\eta}[u^S;\phi] -\mathcal D_{S,f_0}^{\eta}[U^S;\phi^U] }{9u_m^5} \\
& \quad = \int_0^1 \left[ \ell^2(\widehat\alpha_\eta-\vartheta)\beta -(\widehat\alpha_\eta-\vartheta_0) \right] y(1-y)^2\Phi_y^2 \, dy  +\int_0^1 \left[ \ell(\widehat\alpha_\eta\vartheta-1) -(\widehat\alpha_\eta\vartheta_0-1) \right]\Phi^2\,dy\\
& \qquad +\frac13\int_0^1 \left[ \ell(\widehat\alpha_\eta+\vartheta) -(\widehat\alpha_\eta+\vartheta_0) \right]\Phi^3 \, dy  +\frac14(\ell-1)\int_0^1\Phi^4\,dy +5\left(\ell^2A^2-A_0^2\right).
\end{split}
\end{equation}
We first estimate the four coefficient differences in \eqref{eq:shock-profile-coordinate-difference}. Since \(\widehat\alpha_\eta\) is uniformly bounded for \(0\leq\eta\leq\eta_0\), the estimates \eqref{eq:perturbed-profile-coordinate-closeness} give
\begin{equation} \label{eq:shock-profile-coefficient-difference-bound}
\begin{split}
&\left| \ell^2(\widehat\alpha_\eta-\vartheta)\beta -(\widehat\alpha_\eta-\vartheta_0) \right| +\left| \ell(\widehat\alpha_\eta\vartheta-1) -(\widehat\alpha_\eta\vartheta_0-1) \right|\\
&\quad +\left| \ell(\widehat\alpha_\eta+\vartheta) -(\widehat\alpha_\eta+\vartheta_0) \right| +|\ell-1| \leq C\varepsilon_f.
\end{split}
\end{equation}
We next estimate the last term in \eqref{eq:shock-profile-coordinate-difference}. We decompose
\begin{equation*}
\ell^2A^2-A_0^2 =(\ell^2-1)A^2+(A-A_0)(A+A_0).
\end{equation*}
By \eqref{eq:perturbed-profile-coordinate-closeness} and the Cauchy--Schwarz inequality, we have
\begin{equation*}
|A-A_0| = \left| \int_0^1(\vartheta_0-\vartheta)\Phi\,dy \right| \leq C\varepsilon_f \left(\int_0^1\Phi^2\,dy\right)^{1/2}.
\end{equation*}
Since \(\widehat\alpha_\eta\), \(\vartheta\), and \(\vartheta_0\) are uniformly bounded, the Cauchy--Schwarz inequality gives
\begin{equation*}
\begin{split}
|A|^2+|A_0|^2 &\leq \int_0^1\left[(\widehat\alpha_\eta-\vartheta)\Phi+\frac12\Phi^2\right]^2\,dy+\int_0^1\left[(\widehat\alpha_\eta-\vartheta_0)\Phi+\frac12\Phi^2\right]^2\,dy\\
&\leq C\int_0^1(\Phi^2+\Phi^4)\,dy,
\end{split}
\end{equation*}
where, in the second inequality, we used $(a+b)^2 \leq 2a^2 + 2b^2$. Together with \(|\ell^2-1|\leq C\varepsilon_f\), these estimates give
\begin{equation}\label{eq:shock-profile-shift-difference-bound}
\left|\ell^2A^2-A_0^2\right| \leq C\varepsilon_f \int_0^1(\Phi^2+\Phi^4)\,dy.
\end{equation}
Substituting \eqref{eq:shock-profile-coefficient-difference-bound} and \eqref{eq:shock-profile-shift-difference-bound} into \eqref{eq:shock-profile-coordinate-difference}, and using \(|\Phi|^3\leq(\Phi^2+\Phi^4)/2\), we obtain
\begin{equation}\label{eq:shock-profile-coordinate-estimate}
\left| \mathcal D_{S,f_0}^{\eta}[u^S;\phi] -\mathcal D_{S,f_0}^{\eta}[U^S;\phi^U] \right| \leq C\varepsilon_fu_m^5 \left( \int_0^1y(1-y)^2\Phi_y^2\,dy +\int_0^1(\Phi^2+\Phi^4)\,dy \right).
\end{equation}
Finally, \eqref{eq:perturbed-profile-coordinate-jacobian} gives
\begin{equation}\label{eq:shock-profile-coordinate-norms}
\begin{aligned}
u_m\int_{\mathbb R}\phi_\xi^2\,d\xi &= 9\ell^2u_m^5 \int_0^1\beta\,y(1-y)^2\Phi_y^2\,dy,\\
\int_{\mathbb R}u^S_\xi \left(u_m^2\phi^2+\phi^4\right) \, d\xi &= 3\ell u_m^5 \int_0^1(\Phi^2+\Phi^4)\,dy.
\end{aligned}
\end{equation}
By \eqref{eq:perturbed-profile-coordinate-closeness}, \(\ell\) and \(\beta\) are uniformly bounded above and below. Hence, \eqref{eq:shock-profile-coordinate-estimate} and \eqref{eq:shock-profile-coordinate-norms} yield \eqref{eq:shock-profile-functional-perturbation}.
\end{proof}

\subsubsection{Perturbation of the flux}

We next keep the profiles fixed and estimate the effect of replacing \(f_0\) by \(f\), including perturbative estimates for $\mathcal{A}_f^\eta$ needed in the remainder estimate.

\begin{lemma}\label{lem:shock-rarefaction-functional-perturbation}
Under the assumptions in Proposition~\ref{prop:large-L2-apriori}, there exists a constant \(C>0\), independent of \(f\), such that
\begin{align}
&\left|\mathcal D_{S,f}^{\eta}[u^S;\phi]-\mathcal D_{S,f_0}^{\eta}[u^S;\phi]\right|
\leq C\varepsilon_f \int_{\mathbb R}u^S_\xi\left(u_m^2\phi^2+\phi^4\right) \, d\xi,\label{eq:shock-functional-perturbation}\\
&\left|\mathcal D_{R,f}^{\eta}[u^R;\phi]-\mathcal D_{R,f_0}^{\eta}[u^R;\phi]\right|
\leq C\varepsilon_f \int_{\mathbb R}u^R_\xi\left(u_m^2\phi^2+\phi^4\right) \, d\xi \label{eq:rarefaction-functional-perturbation}
\end{align}
for all \(t\in[0,T]\), and
\begin{align}
&\left|\bigl[\mathcal A_f^\eta(\widetilde u,\phi)-\mathcal A_{f_0}^\eta(\widetilde u,\phi)\bigr]-\bigl[\mathcal A_f^\eta(u^S,\phi)-\mathcal A_{f_0}^\eta(u^S,\phi)\bigr]\right|
\leq C\varepsilon_f (u_m^2\phi^2+\phi^4),\label{eq:q-shock-side-mixed-symbol}\\
&\left|\bigl[\mathcal A_f^\eta(\widetilde u,\phi)-\mathcal A_{f_0}^\eta(\widetilde u,\phi)\bigr]-\bigl[\mathcal A_f^\eta(u^R,\phi)-\mathcal A_{f_0}^\eta(u^R,\phi)\bigr]\right|
\leq C\varepsilon_f (u_m^2\phi^2+\phi^4)\label{eq:q-rarefaction-side-mixed-symbol}
\end{align}
for all \((t,\xi)\in[0,T]\times\mathbb R\).
\end{lemma}

\begin{proof}
From the definitions of $\mathcal D_{S,f}^{\eta}$, $\mathcal D_{R,f}^{\eta}$, $\mathcal D_{S,f_0}^{\eta}$, and $\mathcal D_{R,f_0}^{\eta}$, we have
\begin{equation}\label{eq:DS-exact-difference}
\begin{split}
& \mathcal D_{S,f}^{\eta}[u^S;\phi]-\mathcal D_{S,f_0}^{\eta}[u^S;\phi] \\
& \quad = \int_{\mathbb R}u^S_\xi\Bigl[\mathcal B_f^\eta(u^S,\phi)-\mathcal B_{f_0}^\eta(u^S,\phi) + \frac12\bigl(f'(u^S)-f_0'(u^S)-2f'(u_m)+2f_0'(u_m)\bigr)\phi^2 \Bigr]\,d\xi
\end{split}
\end{equation}
and
\begin{equation}\label{eq:DR-exact-difference}
\mathcal D_{R,f}^{\eta}[u^R;\phi]-\mathcal D_{R,f_0}^{\eta}[u^R;\phi] =\int_{\mathbb R}u^R_\xi\left[ \mathcal B_f^\eta(u^R,\phi)-\mathcal B_{f_0}^\eta(u^R,\phi) -\frac12\bigl(f'(u^R)-f_0'(u^R)\bigr)\phi^2 \right]\,d\xi.
\end{equation}
We now estimate the quantities in brackets in \eqref{eq:DS-exact-difference} and \eqref{eq:DR-exact-difference} by rewriting them in terms of the flux perturbation \(q:=f-f_0\).

We recall that
\begin{equation*}
\mathcal B_f^\eta(v,z) =(\alpha_\eta-v+z)\bigl(f(v+z)-f(v)\bigr)-\int_0^z\bigl(f(v+s)-f(v)\bigr)\,ds-(\alpha_\eta-v)f'(v)z.
\end{equation*}
Since \(\mathcal B_g^\eta\) is linear with respect to \(g\), we have
\begin{equation*}
\mathcal B_f^\eta(v,z)-\mathcal B_{f_0}^\eta(v,z)=\mathcal B_q^\eta(v,z).
\end{equation*}
By Taylor's formula,
\begin{equation*}
q(v+z)-q(v)=q'(v)z+\int_0^z(z-s)q''(v+s)\,ds
\end{equation*}
and, by integration by parts,
\begin{equation*}
\int_0^z\bigl(q(v+s)-q(v)\bigr)\,ds =\frac12q'(v)z^2+\frac12\int_0^z(z-s)^2q''(v+s)\,ds.
\end{equation*}
Substituting these identities into the definition of \(\mathcal B_q^\eta\), we find
\begin{equation}\label{eq:Bg-representation}
\begin{split}
\mathcal B_q^\eta(v,z) &=\frac12 q'(v)z^2 + \int_0^z \left( (\alpha_\eta-v+z) (z-s) - \frac12 (z-s)^2 \right) q''(v+s)\,ds \\
&=\frac12 q'(v)z^2 +z^2\int_0^1(1-\theta)\left(\alpha_\eta-v+\frac{1+\theta}{2}z\right)q''(v+\theta z)\, d\theta,
\end{split}
\end{equation}
where the second equality follows from the change of variables $s=\theta z$ and the identity
\begin{equation*}
(\alpha_\eta-v+z)(z-s)-\frac12(z-s)^2=(z-s)\left(\alpha_\eta-v+\frac{z+s}{2}\right).
\end{equation*}

Using \eqref{eq:Bg-representation}, the bracket in \eqref{eq:DS-exact-difference} becomes
\begin{equation}\label{eq:shock-bracket-identity}
\begin{split}
&\mathcal B_f^\eta(u^S,\phi)-\mathcal B_{f_0}^\eta(u^S,\phi) +\frac12\bigl(f'(u^S)-f_0'(u^S)-2f'(u_m)+2f_0'(u_m)\bigr)\phi^2\\
&\quad =\bigl(q'(u^S)-q'(u_m)\bigr)\phi^2 +\phi^2\int_0^1(1-\theta) \left(\alpha_\eta-u^S+\frac{1+\theta}{2}\phi\right) q''(u^S+\theta\phi)\,d\theta.
\end{split}
\end{equation}
On the other hand, the \(q'(u^R)\phi^2/2\) term cancels in the rarefaction part, and hence
\begin{equation}\label{eq:rarefaction-bracket-identity}
\begin{split}
&\mathcal B_f^\eta(u^R,\phi)-\mathcal B_{f_0}^\eta(u^R,\phi) -\frac12\bigl(f'(u^R)-f_0'(u^R)\bigr)\phi^2\\
&\quad =\phi^2\int_0^1(1-\theta) \left(\alpha_\eta-u^R+\frac{1+\theta}{2}\phi\right) q''(u^R+\theta\phi)\,d\theta.
\end{split}
\end{equation}

By \eqref{eq:flux-perturbation-size},
\begin{equation} \label{q''bound}
|q''(w)|\leq\varepsilon_f(u_m+|w|), \quad \text{for all } w \in \mathbb{R}.
\end{equation}
This, together with the bounds $|u^S|+|u^R| \leq Cu_m$, yields
\begin{equation}\label{eq:q-second-derivative-along-phi}
\begin{split}
|q''(u^S+\theta\phi)|+|q''(u^R+\theta\phi)|
&\leq\varepsilon_f\bigl(2u_m+|u^S+\theta\phi|+|u^R+\theta\phi|\bigr)\\
&\leq C\varepsilon_f(u_m+|\phi|).
\end{split}
\end{equation}
Moreover, since \(|\alpha_\eta-u^S|+|\alpha_\eta-u^R|\leq Cu_m\),
\begin{equation}\label{eq:bracket-coefficient-bound}
\left|\alpha_\eta-u^S+\frac{1+\theta}{2}\phi\right|+\left|\alpha_\eta-u^R+\frac{1+\theta}{2}\phi\right|\leq C(u_m+|\phi|).
\end{equation}
In addition, the first term on the right-hand side of \eqref{eq:shock-bracket-identity} satisfies
\begin{equation}\label{eq:q-prime-shock-bound}
\begin{split}
|q'(u^S)-q'(u_m)|
&\leq |u^S-u_m|\int_0^1\left|q''\bigl(u_m+\theta(u^S-u_m)\bigr)\right|\,d\theta\\
&\leq C\varepsilon_fu_m|u^S-u_m| \leq C\varepsilon_fu_m^2.
\end{split}
\end{equation}

Using \eqref{eq:q-second-derivative-along-phi}--\eqref{eq:q-prime-shock-bound} in \eqref{eq:shock-bracket-identity} and applying Young's inequality, we obtain
\begin{equation}\label{eq:shock-bracket-bound}
\begin{split}
&\left|\mathcal B_f^\eta(u^S,\phi)-\mathcal B_{f_0}^\eta(u^S,\phi)+\frac12\bigl(f'(u^S)-f_0'(u^S)-2f'(u_m)+2f_0'(u_m)\bigr)\phi^2\right|\\
&\quad\leq C\varepsilon_f\left(u_m^2\phi^2+u_m|\phi|^3+\phi^4\right) \leq C\varepsilon_f\left(u_m^2\phi^2+\phi^4\right).
\end{split}
\end{equation}
Similarly, applying \eqref{eq:q-second-derivative-along-phi} and \eqref{eq:bracket-coefficient-bound} to \eqref{eq:rarefaction-bracket-identity}, we obtain
\begin{equation}\label{eq:rarefaction-bracket-bound}
\left|\mathcal B_f^\eta(u^R,\phi)-\mathcal B_{f_0}^\eta(u^R,\phi)-\frac12\bigl(f'(u^R)-f_0'(u^R)\bigr)\phi^2\right| \leq C\varepsilon_f\left(u_m^2\phi^2+\phi^4\right).
\end{equation}
Substituting \eqref{eq:shock-bracket-bound} and \eqref{eq:rarefaction-bracket-bound} into \eqref{eq:DS-exact-difference} and \eqref{eq:DR-exact-difference}, respectively, proves \eqref{eq:shock-functional-perturbation} and \eqref{eq:rarefaction-functional-perturbation}.

We now prove \eqref{eq:q-shock-side-mixed-symbol} and \eqref{eq:q-rarefaction-side-mixed-symbol}. From the definition \eqref{eq:general-flux-A-form} and the representation \eqref{eq:Bg-representation}, we have
\begin{equation*}
\mathcal A_f^\eta(v,z)-\mathcal A_{f_0}^\eta(v,z) = \mathcal A_q^\eta(v,z)=q'(v)z^2+z^2\int_0^1(1-\theta)\left(\alpha_\eta-v+\frac{1+\theta}{2}z\right)q''(v+\theta z)\,d\theta.
\end{equation*}
Hence
\begin{equation}\label{eq:Aq-base-point-difference}
\begin{split}
\mathcal A_q^\eta(\widetilde u,\phi)-\mathcal A_q^\eta(u^S,\phi) & =\bigl(q'(\widetilde u)-q'(u^S)\bigr)\phi^2 + (u^S-\widetilde u)\phi^2 \int_0^1(1-\theta)q''(\widetilde u+\theta \phi)\,d\theta\\
&\quad + \phi^2\int_0^1(1-\theta) \left(\alpha_\eta-u^S+\frac{1+\theta}{2}\phi\right) \bigl(q''(\widetilde u+\theta \phi)-q''(u^S+\theta \phi)\bigr)\,d\theta.
\end{split}
\end{equation}
Since \(|\widetilde u-u^S|=|u^R-u_m|\leq\delta_R\leq\delta_0u_m\), \eqref{q''bound} and \(\lVert q'''\rVert_{L^\infty}\leq\varepsilon_f\) give
\begin{equation*}
\begin{split}
& |q'(\widetilde u)-q'(u^S)| \leq |\widetilde u - u^S| \int_0^1 \left|q''\bigl(u^S + s(\widetilde u - u^S)\bigr)\right|\,ds \leq C\varepsilon_f u_m^2,\\
& |q''(\widetilde u+\theta \phi)| \leq C\varepsilon_f(u_m+|\phi|),\\
& |q''(\widetilde u+\theta \phi)-q''(u^S+\theta \phi)| \leq |\widetilde u - u^S| \lVert q'''\rVert_{L^\infty}\leq C \varepsilon_f u_m.
\end{split}
\end{equation*}
Substituting these estimates and the bound for $u^S$ in \eqref{eq:bracket-coefficient-bound} into \eqref{eq:Aq-base-point-difference}, and applying Young's inequality, we obtain
\begin{equation*}
|\mathcal A_q^\eta(\widetilde u,\phi)-\mathcal A_q^\eta(u^S,\phi)|\leq C\varepsilon_f \left(u_m^2 \phi^2 + \phi^4\right).
\end{equation*}
This proves \eqref{eq:q-shock-side-mixed-symbol}.

It remains to prove \eqref{eq:q-rarefaction-side-mixed-symbol}. Repeating the calculation leading to \eqref{eq:Aq-base-point-difference} with \(u^R\) in place of \(u^S\), and using \(|\widetilde u-u^R|=|u^S-u_m|\), \eqref{q''bound}, and \(\lVert q'''\rVert_{L^\infty}\leq\varepsilon_f\), we obtain
\begin{equation*}
|q'(\widetilde u)-q'(u^R)| \leq C\varepsilon_fu_m^2, \qquad |q''(\widetilde u+\theta\phi)-q''(u^R+\theta\phi)| \leq C \varepsilon_f u_m.
\end{equation*}
Together with \(|q''(\widetilde u+\theta\phi)|\leq C\varepsilon_f(u_m+|\phi|)\) and \eqref{eq:bracket-coefficient-bound}, these estimates yield
\begin{equation*}
|\mathcal A_q^\eta(\widetilde u,\phi)-\mathcal A_q^\eta(u^R,\phi)|
\leq C\varepsilon_f \left(u_m^2\phi^2 + \phi^4\right),
\end{equation*}
which proves \eqref{eq:q-rarefaction-side-mixed-symbol}.
\end{proof}

\subsection{Remaining terms and coercivity of \texorpdfstring{\(\mathcal G_f\)}{Gf}}

We finally close the coercivity estimate for $\mathcal{G}_f$ by estimating the remaining terms and combining all the preceding bounds.

\subsubsection{Estimate of the remaining terms}

\begin{lemma}\label{lem:remaining-good-terms}
Under the assumptions in Proposition~\ref{prop:large-L2-apriori}, there exists a constant \(C>0\) such that
\begin{equation}\label{eq:remaining-good-term-estimate}
\begin{split}
&\mathcal G_f-\mathcal D_{S,f}^{\eta}[u^S;\phi]-\mathcal D_{R,f}^{\eta}[u^R;\phi]\\
&\quad\geq\frac1{2\kappa^2}\left(\kappa-\frac5{u_m}\right)|\dot X|^2 -C\left(\frac{\delta_R}{u_m}+\varepsilon_f\right) \left(u_m\int_{\mathbb R}\phi_\xi^2\,d\xi+\int_{\mathbb R}\widetilde u_\xi(u_m^2\phi^2+\phi^4) \, d\xi\right)
\end{split}
\end{equation}
for all \(t\in[0,T]\).
\end{lemma}

\begin{proof}
By \eqref{eq:general-good-interaction-remainder}, we have
\begin{equation} \label{forestimate}
\begin{split}
&\mathcal G_f-\mathcal D_{S,f}^{\eta}[u^S;\phi]-\mathcal D_{R,f}^{\eta}[u^R;\phi]\\
& \quad =-\int_{\mathbb R}(u^R-u_m)\phi_\xi^2 \, d\xi
+\int_{\mathbb R}(u^R-u_m)u^S_\xi\phi^2(3\alpha_\eta+\phi) \, d\xi\\
&\qquad+\kappa\left(\mathcal Y_S^\eta[u^S;\phi]+\mathcal Z\right)^2 -\frac5{u_m}\left(\mathcal Y_S^\eta[u^S;\phi]\right)^2\\
&\qquad+\int_{\mathbb R}u^S_\xi\Bigl(\bigl[\mathcal A_f^\eta(\widetilde u,\phi)-\mathcal A_{f_0}^\eta(\widetilde u,\phi)\bigr] -\bigl[\mathcal A_f^\eta(u^S,\phi)-\mathcal A_{f_0}^\eta(u^S,\phi)\bigr]\Bigr) \, d\xi,
\end{split}
\end{equation}
where we used the decomposition
\begin{equation*}
\begin{split}
\mathcal A_f^\eta(\widetilde u,\phi)-\mathcal A_f^\eta(u^S,\phi)
&=\mathcal A_{f_0}^\eta(\widetilde u,\phi)-\mathcal A_{f_0}^\eta(u^S,\phi)\\
&\quad+\bigl[\mathcal A_f^\eta(\widetilde u,\phi)-\mathcal A_{f_0}^\eta(\widetilde u,\phi)\bigr]-\bigl[\mathcal A_f^\eta(u^S,\phi)-\mathcal A_{f_0}^\eta(u^S,\phi)\bigr]
\end{split}
\end{equation*}
together with 
\begin{equation*}
\mathcal A_{f_0}^\eta(\widetilde u,\phi)-\mathcal A_{f_0}^\eta(u^S,\phi)=(u^R-u_m)\phi^2(3\alpha_\eta+\phi).
\end{equation*}
We estimate the terms on the right-hand side of \eqref{forestimate} separately.

Since \(0\leq u^R-u_m\leq\delta_R\) and \(\alpha_\eta>0\), using \(u_m|\phi|^3\leq(u_m^2\phi^2+\phi^4)/2\), we obtain
\begin{equation*}
\begin{split}
&-\int_{\mathbb R}(u^R-u_m)\phi_\xi^2\,d\xi+\int_{\mathbb R}(u^R-u_m)u^S_\xi\phi^2(3\alpha_\eta+\phi)\,d\xi\\
&\quad\geq-\frac{\delta_R}{u_m}\left(u_m\int_{\mathbb R}\phi_\xi^2\,d\xi\right)-\frac{\delta_R}{2u_m}\int_{\mathbb R}u^S_\xi(u_m^2\phi^2+\phi^4)\,d\xi.
\end{split}
\end{equation*}
For the third and fourth terms, we use the identity \(\mathcal Y_S^\eta[u^S;\phi]=\mathcal Y-\mathcal Z\), the shift ODE \eqref{eq:shift-ode}, and Young's inequality to obtain
\begin{equation}\label{YsZ}
\begin{split}
\kappa\mathcal Y^2-\frac5{u_m}(\mathcal Y-\mathcal Z)^2 & = \left( \kappa - \frac{5}{u_m} \right) \mathcal{Y}^2 + \frac{10}{u_m} \mathcal{Y} \mathcal{Z} - \frac{5}{u_m} \mathcal{Z}^2 \\
& \geq\frac1{2\kappa^2}\left(\kappa-\frac5{u_m}\right)|\dot X|^2 - \frac{5\left(\kappa+\frac5{u_m}\right)}{\kappa-\frac5{u_m}} \frac{\mathcal Z^2}{u_m}.
\end{split}
\end{equation}
The last term in \eqref{YsZ} is estimated using the definition of $\mathcal{Z}$:
\begin{equation*} 
\mathcal Z=-\int_{\mathbb R}(u^R-u_m)u^S_\xi\phi\,d\xi+\int_{\mathbb R}u^R_\xi\left(a_\eta\phi+\frac12\phi^2\right) \, d\xi.
\end{equation*}
Since \(\int_{\mathbb R}u^S_\xi\,d\xi\leq Cu_m\), \(\int_{\mathbb R}u^R_\xi\,d\xi=\delta_R\), and \(|a_\eta|\leq Cu_m\), the Cauchy--Schwarz inequality gives
\begin{equation*}
\begin{split}
\frac{\mathcal Z^2}{u_m} &\leq\frac2{u_m}\left(\int_{\mathbb R}(u^R-u_m)u^S_\xi\phi\,d\xi\right)^2+\frac2{u_m}\left(\int_{\mathbb R}u^R_\xi\left(a_\eta\phi+\frac12\phi^2\right)\,d\xi\right)^2\\
&\leq C\frac{\delta_R^2}{u_m^2}\int_{\mathbb R}u^S_\xi u_m^2\phi^2\,d\xi+C\frac{\delta_R}{u_m}\int_{\mathbb R}u^R_\xi(u_m^2\phi^2+\phi^4)\,d\xi.
\end{split}
\end{equation*}
Finally, it follows directly from \eqref{eq:q-shock-side-mixed-symbol} that
\begin{equation*}
\left|\int_{\mathbb R}u^S_\xi\Bigl(\bigl[\mathcal A_f^\eta(\widetilde u,\phi)-\mathcal A_{f_0}^\eta(\widetilde u,\phi)\bigr]
-\bigl[\mathcal A_f^\eta(u^S,\phi)-\mathcal A_{f_0}^\eta(u^S,\phi)\bigr]\Bigr)\,d\xi\right|
\leq C\varepsilon_f\int_{\mathbb R}u^S_\xi\left(u_m^2\phi^2+\phi^4\right)\,d\xi.
\end{equation*}
Combining all the above estimates proves \eqref{eq:remaining-good-term-estimate}.
\end{proof}

\subsubsection{Completion of the proof of Proposition~\ref{prop:main-coercivity}}

By \eqref{eq:reference-shock-pullback}, \eqref{eq:profile-jacobian}, \eqref{eq:perturbed-profile-coordinate-closeness}, and \eqref{eq:shock-profile-coordinate-norms}, we have
\begin{equation}\label{eq:reference-perturbed-shock-norm-equivalence}
\begin{split}
&u_m\int_{\mathbb R}(\phi^U_\xi)^2\,d\xi+\int_{\mathbb R}U^S_\xi\left(u_m^2(\phi^U)^2+(\phi^U)^4\right)\,d\xi\\
&\quad\sim u_m\int_{\mathbb R}\phi_\xi^2\,d\xi+\int_{\mathbb R}u^S_\xi\left(u_m^2\phi^2+\phi^4\right)\,d\xi.
\end{split}
\end{equation}
Here, $A\sim B$ means that $cB \leq A \leq CB$ for constants $c>0$ and $C>0$. Using the decomposition \eqref{eq:shock-perturbative-decomposition}, together with the estimates \eqref{eq:reference-shock-coercivity-statement}, \eqref{eq:exact-cubic-shock-eta-error}, \eqref{eq:shock-profile-functional-perturbation}, \eqref{eq:shock-functional-perturbation}, and \eqref{eq:reference-perturbed-shock-norm-equivalence}, we obtain
\begin{equation*}
\mathcal D_{S,f}^{\eta}[u^S;\phi] \geq (c-C\eta-C\varepsilon_f)\left(u_m\int_{\mathbb R}\phi_\xi^2 \, d\xi + \int_{\mathbb R}u^S_\xi\left(u_m^2\phi^2+\phi^4\right) \, d\xi \right).
\end{equation*}
Likewise, \eqref{eq:rarefaction-perturbative-decomposition}, \eqref{eq:rarefaction-coercivity}, and \eqref{eq:rarefaction-functional-perturbation} give, for some \(c_\eta>0\),
\begin{equation*}
\mathcal D_{R,f}^{\eta}[u^R;\phi]\geq(c_\eta-C\varepsilon_f)\int_{\mathbb R}u^R_\xi\left(u_m^2\phi^2+\phi^4\right)\,d\xi.
\end{equation*}
Combining these two estimates with \eqref{eq:remaining-good-term-estimate}, we obtain
\begin{equation*}
\begin{split}
\mathcal G_f(t) &\geq \left(c-C\left(\eta+\varepsilon_f+\frac{\delta_R}{u_m}\right)\right)\left(u_m\int_{\mathbb R}\phi_\xi^2 \, d\xi + \int_{\mathbb R} u^S_\xi \left(u_m^2\phi^2+\phi^4\right) \, d\xi \right)\\
&\quad + \left(c_\eta-C\left(\varepsilon_f+\frac{\delta_R}{u_m}\right)\right)\int_{\mathbb R}u^R_\xi\left(u_m^2\phi^2+\phi^4\right) \, d\xi + \frac1{2\kappa^2}\left(\kappa-\frac5{u_m}\right) |\dot X(t)|^2.
\end{split}
\end{equation*}
We choose \(\eta_0>0\) sufficiently small and then, for each \(0<\eta<\eta_0\), choose \(\varepsilon_0>0\) and \(\delta_0>0\) sufficiently small. With these choices, the coercivity estimate \eqref{eq:perturbed-good-term} follows from \(\kappa>5/u_m\) and \(\widetilde u_\xi=u^S_\xi+u^R_\xi\).

\section{Interaction errors and completion of the a priori estimate} \label{sec:5}

This section is devoted to estimating the interaction and residual terms \(\mathcal I_f\) and \(\mathcal R\) in \eqref{eq:master-decomposition} and to completing the proof of Proposition~\ref{prop:large-L2-apriori}. We first derive the uniform wave interaction estimates and the resulting integrated bounds for the profile residual \(F\). We then estimate \(\mathcal I_f\) and \(\mathcal R\) and combine these estimates with the coercivity result of Section~\ref{sec:4} to close the a priori estimate.

\subsection{Uniform wave interaction estimates}

The degeneracy of the shock at the right state \(u_m\) gives rise to an algebraically decaying tail on the rarefaction side, while the shock and rarefaction remain attached. The resulting wave interaction estimates therefore require particular attention. However, for the cubic flux \(f_0(u)=u^3\), the two estimates needed here were established in \cite[Lemma~4.8]{HWZ}. The following lemma shows that the same estimates hold uniformly over the present admissible flux class.

\begin{lemma}[Uniform wave interaction estimates] \label{lem:HWZ-profile-interactions}
Under the assumptions in Proposition~\ref{prop:large-L2-apriori}, there exists a constant $C>0$, independent of \(f\), \(t\), and \(\delta_R\), such that
\begin{align} 
\int_{-\infty}^0|u^S-u_m|u^R_\xi\,d\xi & \leq C\delta_S\delta_R^{2/11}(1+t)^{-4/5}, \label{eq:HWZ-A-left} \\ 
\int_0^\infty|u^S-u_m|u^R_\xi\,d\xi & \leq C\delta_S^{-3/5}\delta_R^{1/5}(1+t)^{-4/5} \label{eq:HWZ-A-right} 
\end{align}
for all \(t\geq0\).
\end{lemma}

\begin{proof}
For \(f=f_0\), the estimates \eqref{eq:HWZ-A-left} and \eqref{eq:HWZ-A-right} are established in \cite[Lemma~4.8]{HWZ}. It therefore remains to verify that the constants in these estimates can be chosen independently of \(f\).

The proof of \eqref{eq:HWZ-A-left} in \cite[Lemma~4.8]{HWZ} uses the estimate \eqref{eq:approx-rarefaction-refined-tails} with \(\theta=1/5\) and the bound \(|u^S-u_m|\leq\delta_S\). The proof of \eqref{eq:HWZ-A-right} uses \eqref{eq:perturbed-shock-right-tail} and the \(L^5\)-estimate for \(u^R_\xi\) in \eqref{eq:approx-rarefaction-derivatives}. Since the constants in the shock-profile bounds and in Lemma~\ref{lem:approximate-rarefaction} can be chosen independently of \(f\), the corresponding calculations in \cite[Lemma~4.8]{HWZ} yield \eqref{eq:HWZ-A-left} and \eqref{eq:HWZ-A-right} with uniform constants.
\end{proof}

\subsection{Integrated bounds for the profile residual}

\begin{lemma} \label{lem:profile-interaction-input}
Under the assumptions in Proposition~\ref{prop:large-L2-apriori}, there exists a constant \(C>0\), independent of \(f\) and \(\delta_R\), such that
\begin{equation*} 
\sup_{t\ge0}\int_{\mathbb R}|F(t,\xi)|\,d\xi\le1,\qquad \lim_{t\to+\infty}Q(t) \le C\left(\frac{\delta_R}{u_m}\right)^{8/33}, 
\end{equation*}
where \(F\) and \(Q\) are defined in \eqref{eq:residual} and \eqref{eq:interaction-size}, respectively.
\end{lemma}

\begin{proof}
Since \(\widetilde u-u^S=u^R-u_m\), \(\widetilde u-u^R=u^S-u_m\), and $|u^S|+|u^R|+|\widetilde u|\leq Cu_m$, we have
\begin{equation*} 
\left|u^S+\theta(u^R-u_m)\right|+\left|u^R+\theta(u^S-u_m)\right|\leq Cu_m,\qquad 0 \leq \theta \leq1. 
\end{equation*}
Writing \(q=f-f_0\), \eqref{eq:flux-perturbation-size} gives
\begin{equation*} 
|f''(v)|\leq6|v|+|q''(v)|\leq6|v|+\varepsilon_f(u_m+|v|)\leq Cu_m 
\end{equation*}
for \(v=u^S+\theta(u^R-u_m)\) or \(v=u^R+\theta(u^S-u_m)\). Hence, by \eqref{eq:residual} and the fundamental theorem of calculus, we have
\begin{equation}\label{eq:F-by-ABV} 
\begin{split} 
\int_{\mathbb R}|F(t,\xi)|\,d\xi &= \int_{\mathbb R}\left|\bigl[f'(\widetilde u)-f'(u^S)\bigr] u^S_\xi + \bigl[f'(\widetilde u)-f'(u^R)\bigr]u^R_\xi-u^R_{\xi\xi}\right|\,d\xi\\
&\leq \int_{\mathbb R} \bigg[ |u^R-u_m|u^S_\xi\int_0^1\left|f''\bigl(u^S+\theta(u^R-u_m)\bigr)\right|\,d\theta \\ 
& \qquad \quad +|u^S-u_m|u^R_\xi\int_0^1\left|f''\bigl(u^R+\theta(u^S-u_m)\bigr)\right|\,d\theta+|u^R_{\xi\xi}| \bigg] \, d\xi\\ &\leq C\bigl(A(t)+B(t)+V(t)\bigr), 
\end{split} 
\end{equation}
where
\begin{equation*} 
A(t) := \int_{\mathbb R}|u^S-u_m|u^R_\xi \, d\xi, \qquad B(t):=\int_{\mathbb R} |u^R-u_m|u^S_\xi \, d\xi, \qquad V(t):=\int_{\mathbb R}|u^R_{\xi\xi}| \, d\xi.
\end{equation*}
Since \(u^S\leq u_m\leq u^R\) and \((u_m-u^S)(u^R-u_m)\to0\) as \(\xi\to\pm\infty\), integration by parts gives
\begin{equation*} 
A(t)=\int_{\mathbb R}(u_m-u^S)u^R_\xi\,d\xi=\int_{\mathbb R}(u^R-u_m)u^S_\xi\,d\xi=B(t).
\end{equation*}
By \eqref{eq:interaction-size}, \eqref{eq:F-by-ABV}, and \(A=B\),
\begin{equation}\label{eq:Q-by-ABV} 
\begin{split} \lim_{\tau\to+\infty}Q(\tau) &=\int_0^\infty \bigg[ A(t)^2+\left(\int_{\mathbb R}|F(t,\xi)|\,d\xi\right)^{4/3} \bigg] \,dt\\ 
&\leq\int_0^\infty A(t)^2\,dt+C\int_0^\infty A(t)^{4/3}\,dt+C\int_0^\infty V(t)^{4/3}\,dt. 
\end{split} 
\end{equation}
We estimate the three integrals on the right-hand side of \eqref{eq:Q-by-ABV}.

Since \(\delta_S\sim u_m\) by Lemma~\ref{lem:perturbed-degenerate-profile}, \eqref{eq:HWZ-A-left} and \eqref{eq:HWZ-A-right} give
\begin{equation}\label{eq:A-pointwise} 
A(t)=B(t)\leq C\left(\delta_R^{2/11}+\delta_R^{1/5}\right)(1+t)^{-4/5}\leq C\delta_R^{2/11}(1+t)^{-4/5}. 
\end{equation}
Consequently, we obtain the bounds for the first two integrals in \eqref{eq:Q-by-ABV}:
\begin{equation}\label{eq:A-square-integral} 
\int_0^\infty A(t)^2\,dt\leq C\delta_R^{4/11}\int_0^\infty(1+t)^{-8/5}\,dt\leq C\delta_R^{4/11} 
\end{equation}
and
\begin{equation}\label{eq:A-four-thirds-integral} 
\int_0^\infty A(t)^{4/3}\,dt\leq C\delta_R^{8/33}\int_0^\infty(1+t)^{-16/15}\,dt\leq C\delta_R^{8/33}.
\end{equation}

It remains to estimate \(V(t)\). By \eqref{eq:approx-rarefaction-derivatives} with \(p=1\), we have
\begin{equation*} 
V(t) = \lVert u^R_{\xi\xi}(t) \rVert_{L^1} \leq C\min\{\delta_R,(1+t)^{-1}\}.
\end{equation*}
Splitting the time integral at \(1+t=\delta_R^{-1}\), we obtain
\begin{equation}\label{eq:V-four-thirds-integral} 
\int_0^\infty V(t)^{4/3}\,dt \leq C\delta_R^{4/3}\int_0^{\delta_R^{-1}-1} \, dt + C\int_{\delta_R^{-1}-1}^\infty(1+t)^{-4/3}\,dt\leq C\delta_R^{1/3}. 
\end{equation}

Substituting \eqref{eq:A-square-integral}, \eqref{eq:A-four-thirds-integral}, and \eqref{eq:V-four-thirds-integral} into \eqref{eq:Q-by-ABV}, we conclude that
\begin{equation}\label{eq:Q-infinity-bound} 
\lim_{t\to+\infty}Q(t) \leq C\left(\delta_R^{4/11}+\delta_R^{8/33}+\delta_R^{1/3}\right)\leq C \delta_R^{8/33} \leq C\left(\frac{\delta_R}{u_m}\right)^{8/33} 
\end{equation}
for sufficiently small $\delta_R>0$, where \(u_m^{8/33}\) is absorbed into \(C\) in the last inequality.

Finally, we establish the uniform bound on $\lVert F(t) \rVert_{L^1}$. By \eqref{eq:A-pointwise} and \eqref{eq:approx-rarefaction-derivatives} with \(p=1\), we have
\begin{equation*} 
A(t)+B(t)+V(t)\leq C\left(\delta_R^{2/11}+\delta_R\right)\leq C\delta_R^{2/11}.
\end{equation*}
Hence, \eqref{eq:F-by-ABV} gives, for sufficiently small \(\delta_R>0\),
\begin{equation*} 
\sup_{t\geq0}\int_{\mathbb R}|F(t,\xi)|\,d\xi\leq C\delta_R^{2/11}\leq 1. 
\end{equation*}
This, together with \eqref{eq:Q-infinity-bound}, completes the proof.
\end{proof}

\subsection{Estimate of \texorpdfstring{$\mathcal I_f$}{If}}

We now estimate the interaction term \(\mathcal I_f\) arising from the rarefaction part of \eqref{eq:master-decomposition}.

\begin{lemma}\label{lem:interaction-term-estimate}
Under the assumptions in Proposition~\ref{prop:large-L2-apriori}, for every \(\gamma>0\), there exists a constant \(C_\gamma>0\), independent of $f$, such that
\begin{equation}\label{eq:perturbed-interaction-estimate}
\begin{split}
|\mathcal I_f(t)|
&\leq \gamma\left(u_m\int_{\mathbb R}\phi_\xi^2\,d\xi+\int_{\mathbb R}u^R_\xi\left(u_m^2\phi^2+\phi^4\right)\,d\xi\right)\\
&\quad+C_\gamma\left(\int_{\mathbb R}(u_m-u^S)u^R_\xi\,d\xi\right)^2\mathcal E(t)
+ C_\gamma\varepsilon_f\int_{\mathbb R}u^R_\xi\left(u_m^2\phi^2+\phi^4\right) \, d\xi
\end{split}
\end{equation}
for all \(t\in[0,T]\), where \(\mathcal I_f\) is defined in \eqref{eq:general-good-interaction-remainder}.
\end{lemma}

\begin{proof}
By the definition of \(\mathcal I_f\) in \eqref{eq:general-good-interaction-remainder},
\begin{equation*}
\mathcal I_f=\int_{\mathbb R}u^R_\xi\left[\mathcal A_f^\eta(\widetilde u,\phi)-\mathcal A_f^\eta(u^R,\phi)\right]\,d\xi.
\end{equation*}
We decompose the difference in the integrand as
\begin{equation*}
\begin{split}
\mathcal A_f^\eta(\widetilde u,\phi)-\mathcal A_f^\eta(u^R,\phi)
&=\mathcal A_{f_0}^\eta(\widetilde u,\phi)-\mathcal A_{f_0}^\eta(u^R,\phi)\\
&\quad+\bigl[\mathcal A_f^\eta(\widetilde u,\phi)-\mathcal A_{f_0}^\eta(\widetilde u,\phi)\bigr]-\bigl[\mathcal A_f^\eta(u^R,\phi)-\mathcal A_{f_0}^\eta(u^R,\phi)\bigr].
\end{split}
\end{equation*}
For \(f_0(u)=u^3\), \eqref{eq:general-flux-nonlinear-form} and \eqref{eq:general-flux-A-form} give
\begin{equation*}
\mathcal A_{f_0}^\eta(v,z)=3\alpha_\eta vz^2+(\alpha_\eta+v)z^3+\frac34z^4.
\end{equation*}
Since \(\widetilde u-u^R=u^S-u_m\) by \eqref{eq:composite-profile}, it follows that
\begin{equation*}
\mathcal A_{f_0}^\eta(\widetilde u,\phi)-\mathcal A_{f_0}^\eta(u^R,\phi)=(u^S-u_m)\phi^2(3\alpha_\eta+\phi).
\end{equation*}
Substituting this identity into the above decomposition, we obtain
\begin{equation}\label{eq:rarefaction-side-mixed-decomposition}
\begin{split}
\mathcal I_f &= \int_{\mathbb R}(u^S-u_m)u^R_\xi\phi^2(3\alpha_\eta+\phi)\,d\xi\\
&\quad+\int_{\mathbb R}u^R_\xi\Bigl(\bigl[\mathcal A_f^\eta(\widetilde u,\phi)-\mathcal A_{f_0}^\eta(\widetilde u,\phi)\bigr]-\bigl[\mathcal A_f^\eta(u^R,\phi)-\mathcal A_{f_0}^\eta(u^R,\phi)\bigr]\Bigr)\,d\xi.
\end{split}
\end{equation}
Note that, by \eqref{eq:q-rarefaction-side-mixed-symbol}, the second integral is bounded as
\begin{equation}\label{eq:q-rarefaction-side-mixed-term}
\begin{split}
& \left|\int_{\mathbb R}u^R_\xi\Bigl(\bigl[\mathcal A_f^\eta(\widetilde u,\phi)-\mathcal A_{f_0}^\eta(\widetilde u,\phi)\bigr]-\bigl[\mathcal A_f^\eta(u^R,\phi)-\mathcal A_{f_0}^\eta(u^R,\phi)\bigr]\Bigr) \, d\xi\right| \\
& \quad \leq C \varepsilon_f \int_{\mathbb R} u^R_\xi\left(u_m^2 \phi^2+ \phi^4 \right) \, d\xi.
\end{split}
\end{equation}

We now estimate the first integral in \eqref{eq:rarefaction-side-mixed-decomposition}. Since \(u^S\leq u_m\), we have \((u_m-u^S)u^R_\xi\geq0\) and hence
\begin{equation} \label{If_first}
\left|\int_{\mathbb R}(u^S-u_m)u^R_\xi\phi^2(3\alpha_\eta+\phi)\,d\xi\right|
\leq3\alpha_\eta\int_{\mathbb R}(u_m-u^S)u^R_\xi\phi^2\,d\xi
+\int_{\mathbb R}(u_m-u^S)u^R_\xi|\phi|^3\,d\xi.
\end{equation}
By the one-dimensional Sobolev inequality and Young's inequality, we have
\begin{equation} \label{oSY}
\begin{split}
u_m\int_{\mathbb R}(u_m-u^S)u^R_\xi\phi^2\,d\xi &\leq u_m\|\phi\|_{L^\infty}^2\int_{\mathbb R}(u_m-u^S)u^R_\xi\,d\xi\\
&\leq2u_m\|\phi\|_{L^2}\|\phi_\xi\|_{L^2} \int_{\mathbb R}(u_m-u^S)u^R_\xi \, d\xi \\
&\leq\gamma u_m\|\phi_\xi\|_{L^2}^2 + C_\gamma \left(\int_{\mathbb R}(u_m-u^S)u^R_\xi\,d\xi\right)^2\mathcal E(t),
\end{split}
\end{equation}
where, in the last inequality, we used $\mathcal{E}(t) \sim u_m \lVert \phi \rVert_{L^2}^2$. Moreover, since \(0\leq u_m-u^S\leq C u_m\), the Cauchy--Schwarz inequality yields
\begin{equation*}
\begin{split}
\int_{\mathbb R}(u_m-u^S)u^R_\xi|\phi|^3\,d\xi
&\leq\left(\int_{\mathbb R}(u_m-u^S)u^R_\xi\phi^2\,d\xi\right)^{1/2}
\left(\int_{\mathbb R}(u_m-u^S)u^R_\xi\phi^4\,d\xi\right)^{1/2}\\
&\leq\left(Cu_m\int_{\mathbb R}(u_m-u^S)u^R_\xi\phi^2\,d\xi\right)^{1/2}
\left(\int_{\mathbb R}u^R_\xi\phi^4\,d\xi\right)^{1/2}.
\end{split}
\end{equation*}
Using Young's inequality and \eqref{oSY}, we have
\begin{equation} \label{oSY1}
\begin{split}
\int_{\mathbb R}(u_m-u^S)u^R_\xi|\phi|^3 \, d\xi & \leq \gamma \int_{\mathbb R}u^R_\xi\phi^4 \, d\xi + C_\gamma u_m\int_{\mathbb R}(u_m-u^S)u^R_\xi\phi^2 \, d\xi \\
& \leq \tilde\gamma \left(\int_{\mathbb R}u^R_\xi\phi^4 \, d\xi + u_m\|\phi_\xi\|_{L^2}^2 \right) + C_{\tilde\gamma} \left(\int_{\mathbb R}(u_m-u^S)u^R_\xi\,d\xi\right)^2\mathcal E(t).
\end{split}
\end{equation}
Since \(\alpha_\eta\leq Cu_m\), combining \eqref{eq:q-rarefaction-side-mixed-term} and \eqref{If_first}--\eqref{oSY1}, and relabeling the Young parameters, gives \eqref{eq:perturbed-interaction-estimate}.
\end{proof}

\subsection{Estimate of \texorpdfstring{$\mathcal{R}$}{R}}

The residual contribution \(\mathcal R\) is estimated as follows.

\begin{lemma}\label{lem:remainder-estimate}
Under the assumptions in Proposition~\ref{prop:large-L2-apriori}, for every \(\gamma>0\), there exists a constant \(C_\gamma>0\), independent of $f$, such that
\begin{equation}\label{eq:remainder-estimate}
|\mathcal R(t)|\leq\gamma u_m \|\phi_\xi\|_{L^2}^2+C_\gamma \left( \int_{\mathbb R}|F(t,\xi)|\,d\xi \right)^{4/3}\left(1+\mathcal E(t)\right)
\end{equation}
for all \(t\in[0,T]\), where \(\mathcal R\) is defined in \eqref{eq:general-good-interaction-remainder}.
\end{lemma}

\begin{proof}
Recalling the definition of $\mathcal{R}$ from \eqref{eq:general-good-interaction-remainder} and using \eqref{eq:weight-bounds}, we have
\begin{equation*}
\begin{split}
|\mathcal{R}(t)| & = \left| \int_\mathbb{R} F\left( a_\eta \phi - \frac{1}{2} \phi^2 \right) \, d\xi \right| \\
&\leq C u_m \lVert \phi \rVert_{L^\infty} \int_\mathbb{R} |F| \, d\xi + \frac{1}{2} \lVert \phi \rVert_{L^\infty}^2 \int_\mathbb{R} |F| \, d\xi =: I + II.
\end{split}
\end{equation*}
By the one-dimensional Sobolev inequality and Young's inequality with exponents \(4\) and \(4/3\),
\begin{equation} \label{RI_est}
|I|\leq Cu_m\|\phi\|_{L^2}^{1/2}\|\phi_\xi\|_{L^2}^{1/2}\int_{\mathbb R}|F|\,d\xi\leq\gamma u_m \|\phi_\xi\|_{L^2}^2+C_\gamma u_m\|\phi\|_{L^2}^{2/3}\left(\int_{\mathbb R}|F|\,d\xi\right)^{4/3}.
\end{equation}
Applying \(\|\phi\|_{L^2}^{2/3}\leq1+\|\phi\|_{L^2}^2\) together with \eqref{eq:weighted-energy-equivalence} to the last term, we obtain
\begin{equation*}
|I|\leq\gamma u_m \|\phi_\xi\|_{L^2}^2+C_\gamma\left(\int_{\mathbb R}|F|\,d\xi\right)^{4/3}(1+\mathcal E(t)).
\end{equation*}
Similarly to \eqref{RI_est},
\begin{equation*}
|II| \leq C \lVert \phi \rVert_{L^2} \lVert \phi_\xi \rVert_{L^2} \int_\mathbb{R} |F| \, d\xi \leq \gamma u_m  \lVert \phi_\xi \rVert_{L^2}^2 + C_\gamma \lVert \phi \rVert_{L^2}^2 \left( \int_\mathbb{R} |F| \, d\xi \right)^2.
\end{equation*}
Since \( \int_{\mathbb R}|F(t,\xi)|\,d\xi \leq1\) by Lemma~\ref{lem:profile-interaction-input}, we have 
\begin{equation*}
\left( \int_{\mathbb R}|F(t,\xi)| \, d\xi \right)^2\leq \left( \int_{\mathbb R}|F(t,\xi)|\, d\xi \right)^{4/3}.
\end{equation*}
Collecting all the estimates, we obtain \eqref{eq:remainder-estimate}.
\end{proof}

\subsection{Proof of \texorpdfstring{Proposition~\ref{prop:large-L2-apriori}}{the uniform a priori estimate}}

Combining the exact decomposition \eqref{eq:master-decomposition} with the coercivity estimate \eqref{eq:perturbed-good-term}, the interaction estimate \eqref{eq:perturbed-interaction-estimate}, and the bound \eqref{eq:remainder-estimate} for \(\mathcal R\), we obtain after choosing \(\gamma>0\) sufficiently small and then \(\varepsilon_0>0\) sufficiently small that
\begin{equation}\label{E_inequ}
\begin{split}
\frac{d}{dt}\mathcal E + c \mathcal{D}(t) &\leq C\left(\int_{\mathbb R}|F(t,\xi)|\,d\xi\right)^{4/3} + C \mathcal E(t) Q'(t),
\end{split}
\end{equation}
where $\mathcal{D}$ and $Q$ are defined by \eqref{eq:dissipation} and \eqref{eq:interaction-size}, respectively. Multiplying \eqref{E_inequ} by \(e^{-CQ(t)}\), we obtain
\begin{equation*}
\frac{d}{dt}\left(e^{-CQ(t)}\mathcal E(t)\right) +ce^{-CQ(t)}\mathcal D(t)
\leq Ce^{-CQ(t)} \left(\int_{\mathbb R}|F(t,\xi)|\,d\xi\right)^{4/3}.
\end{equation*}
Since \(Q(0)=0\) and \(Q\) is nondecreasing by \eqref{eq:interaction-size}, we have \(Q(s)\geq0\) and \(Q(t)-Q(s)\geq0\) for \(0\leq s\leq t\). Integrating over \([0,t]\) and multiplying by \(e^{CQ(t)}\), we therefore obtain
\begin{equation*}
\begin{split}
\mathcal E(t)+c\int_0^t\mathcal D(s)\,ds &\leq e^{CQ(t)} \left[ \mathcal E(0) +C\int_0^t\left(\int_{\mathbb R}|F(s,\xi)|\,d\xi\right)^{4/3} \, ds \right]\\
& \leq e^{CQ(t)}\left(\mathcal E(0)+CQ(t)\right).
\end{split}
\end{equation*}
This proves \eqref{eq:large-L2-apriori}. Moreover, since \(Q\) is nondecreasing, Lemma~\ref{lem:profile-interaction-input} gives
\begin{equation*}
Q(t)\leq C\left(\frac{\delta_R}{u_m}\right)^{8/33},\qquad t\geq0.
\end{equation*}
Substituting this bound into \eqref{eq:large-L2-apriori} yields \eqref{eq:large-L2-apriori-power}.

\section{Proofs of the main theorems} \label{sec:proof-mainthm}

\subsection{Proof of Theorem~\ref{mainthm}}

We divide the proof into five steps.

\subsubsection*{Step 1. Choice of parameters and construction of the component waves}

We set \(\kappa:=6/u_m\), so that \(\kappa>5/u_m\). Let \(\eta_0>0\) be the constant in Proposition~\ref{prop:large-L2-apriori}, and fix \(\eta:=\eta_0/2\). We then choose \(\varepsilon_0>0\) and \(\delta_0>0\) sufficiently small so that Proposition~\ref{prop:large-L2-apriori} and Lemmas~\ref{lem:perturbed-degenerate-profile}, \ref{lem:existence_shock}, \ref{lem:approximate-rarefaction}, and~\ref{lem:profile-interaction-input} apply.

The existence and uniqueness of the state \(u_-\in(-5u_m/2,-3u_m/2)\) satisfying \eqref{eq:perturbed-degenerate-compatibility} follow from Lemma~\ref{lem:perturbed-degenerate-profile}. Lemma~\ref{lem:existence_shock} establishes the existence and uniqueness of the corresponding viscous shock profile \(u^S\) satisfying \eqref{eq:shock-ode}. Moreover, \eqref{eq:rarefaction-convexity} implies that \(f'\) is strictly increasing on \([u_m,u_+]\). Hence the approximate rarefaction wave \(u^R\) and the composite profile \(\widetilde u\) introduced in Sections~\ref{sec:1.2} and \ref{sec:1.3} are well defined.

\subsubsection*{Step 2. Global existence and uniqueness of a strong solution}

Lemmas~\ref{lem:existence_shock} and~\ref{lem:approximate-rarefaction}, together with the definition \eqref{eq:composite-profile}, imply that
\begin{equation*}
\lim_{\xi\to-\infty}\widetilde u(0,\xi)=u_-,\qquad \lim_{\xi\to+\infty}\widetilde u(0,\xi)=u_+,\qquad u_-\leq\widetilde u(0,\xi)\leq u_+.
\end{equation*}
They also give
\begin{equation*}
\widetilde u(0,\cdot)\in C^2(\mathbb R),\qquad \widetilde u_x(0,\cdot),\widetilde u_{xx}(0,\cdot)\in L^2(\mathbb R).
\end{equation*}
Since \(u_0-\widetilde u(0,\cdot)\in H^1(\mathbb R)\), Proposition~\ref{prop:global-wellposedness}, applied with \(r=\widetilde u(0,\cdot)\), yields a unique global strong solution \(u\) of \eqref{eq:model}--\eqref{eq:initial-data}.

\subsubsection*{Step 3. Uniform stability estimates}

For the representation of \(u\) in the moving frame \(\xi=x-\sigma t\), let \(X\) be the unique local \(C^1\) solution of \eqref{eq:shift-ode} with \(X(0)=0\), as in Section~\ref{sec:3.1}. By \eqref{eq:appendix-maximum-principle}, \eqref{eq:weight}, and \eqref{eq:weight-bounds}, we have \(\|\phi(t,\cdot)\|_{L^\infty}\leq\|u_0\|_{L^\infty}+Cu_m\). Since \(\widetilde u_\xi\geq0\) and \(\int_{\mathbb R}\widetilde u_\xi\,d\xi=u_+-u_-\), we obtain
\begin{equation*}
\begin{split}
|\dot X(t)| &= \kappa\left|\int_{\mathbb R}\widetilde u_\xi\left(a_\eta\phi+\frac12\phi^2\right)d\xi\right| \\
&\leq \kappa(u_+-u_-)\left(\|a_\eta(t)\|_{L^\infty}\|\phi(t)\|_{L^\infty}+\frac12\|\phi(t)\|_{L^\infty}^2\right) \\
&\leq C\left(u_m\|\phi(t)\|_{L^\infty}+\|\phi(t)\|_{L^\infty}^2\right)\leq C\left(u_m^2+\|u_0\|_{L^\infty}^2\right).
\end{split}
\end{equation*}
Thus \(X\) extends uniquely to a function in \(C^1([0,\infty))\).

Fix \(T>0\). Proposition~\ref{prop:large-L2-apriori} applies to \(u\) and \(X\) on \([0,T]\). By \eqref{eq:perturbation} and \eqref{eq:weighted-energy}, the change of variables \(x=\xi-X(t)+\sigma t\) rewrites the weighted energy as
\begin{equation*}
\mathcal E(t)=\frac12\int_{\mathbb R}\bigl((5-\eta)u_m-\widetilde u(t,x-\sigma t+X(t))\bigr)\bigl(u(t,x)-\widetilde u(t,x-\sigma t+X(t))\bigr)^2\,dx.
\end{equation*}
Since \(X(0)=0\), the estimate \eqref{eq:large-L2-apriori} implies $\mathcal E(t)\leq e^{CQ(t)}\bigl(\mathcal E(0)+CQ(t)\bigr)$. Moreover, \eqref{eq:large-L2-apriori-power} and \eqref{eq:weighted-energy-equivalence} yield
\begin{equation}\label{eq:normalized-reference-stability}
\left\|u(t,\cdot)-\widetilde u(t,\cdot-\sigma t+X(t))\right\|_{L^2(\mathbb R)}^2
\leq C\left(\left\|u_0-\widetilde u(0,\cdot)\right\|_{L^2(\mathbb R)}^2+\left(\delta_R/u_m\right)^{8/33}\right)
\end{equation}
for all \(t\geq0\), where the factor \(e^{C(\delta_R/u_m)^{8/33}}\) is absorbed in \(C\) by \(\delta_R/u_m\leq\delta_0\). Since \(T>0\) is arbitrary, the estimate holds globally in time.

\subsubsection*{Step 4. Global \(L^2\) estimate of the shift velocity} 

By \eqref{eq:dissipation}, \eqref{eq:large-L2-apriori-power}, and the weight bounds \eqref{eq:weight-bounds}, for every \(T>0\), we have 
\begin{equation*}
\begin{aligned}
\int_0^T|\dot X(t)|^2\,dt \leq \int_0^T\mathcal D(t)\,dt &\leq C e^{C\left(\delta_R/u_m\right)^{8/33}}\left(\mathcal E(0)+C\left(\delta_R/u_m\right)^{8/33}\right) \\ 
&\leq C\left(\left\|u_0-\widetilde u(0,\cdot)\right\|_{L^2(\mathbb R)}^2+\left(\delta_R/u_m\right)^{8/33}\right).
\end{aligned}
\end{equation*} 
Letting \(T\to\infty\), we obtain \eqref{eq:shift-velocity-bound}.

\subsubsection*{Step 5. Replacement of the shifted approximate rarefaction} 

To complete the proof, we show that
\begin{equation}\label{eq:normalized-rarefaction-replacement}
\left\|u^R(t,\cdot-\sigma t+X(t))-\bar u^r\left(\frac{\cdot}{t}\right)\right\|_{L^2(\mathbb R)}^2\leq C\left(\|u_0-\widetilde u(0,\cdot)\|_{L^2(\mathbb R)}^2+\left(\delta_R/u_m\right)^{8/33}\right)
\end{equation}
for all \(t>0\). By the triangle inequality, we have
\begin{align*}
\left\|u^R(t,\cdot-\sigma t+X(t))-\bar u^r\left(\frac{\cdot}{t}\right)\right\|_{L^2}^2 & \leq C\left\|u^R(t,\cdot-\sigma t+X(t))-u^r(t,\cdot-\sigma t+X(t))\right\|_{L^2}^2 \\
& \quad + C\|u^r(t,\cdot-\sigma t+X(t))-u^r(t,\cdot-\sigma t)\|_{L^2}^2  \\
& \quad + C\left\|u^r(t,\cdot-\sigma t)-\bar u^r\left(\frac{\cdot}{t}\right)\right\|_{L^2}^2 \\
& =: I + II + III.
\end{align*}
For \(I\), taking \(\theta=1/4\) in \eqref{eq:approx-rarefaction-fan-error} and using \eqref{eq:approx-rarefaction-tails} and \eqref{lambdaR}, we obtain
\begin{equation*}
\begin{aligned}
I &\leq C\delta_R^2\int_{-\infty}^0e^{-4|\xi|}\,d\xi +C\delta_R^{1/2}(1+t)^{-3/2}\lambda_R(1+t) +C\delta_R^2\int_{\lambda_R(1+t)}^\infty e^{-4|\xi-\lambda_R(1+t)|}\,d\xi\\
&\leq C\delta_R^2+Cu_m\delta_R^{3/2}(1+t)^{-1/2} \leq C\left(\delta_R/u_m\right)^{3/2}.
\end{aligned}
\end{equation*}
By \eqref{eq:exact-rarefaction-shock-frame} and \eqref{eq:rarefaction-convexity},
\begin{equation*}
\|u^r_\xi(t)\|_{L^2(\mathbb R)}^2 =\frac{1}{1+t}\int_{u_m}^{u_+}\frac{1}{f''(v)}\,dv \leq C\frac{\delta_R}{u_m}\frac{1}{1+t}.
\end{equation*}
Using this, together with the Cauchy--Schwarz inequality and \eqref{eq:shift-velocity-bound}, we have
\begin{equation*}
\begin{aligned}
II &\leq C\frac{\delta_R}{u_m}\frac{|X(t)|^2}{1+t} \leq C\frac{\delta_R}{u_m}\int_0^t|\dot X(s)|^2\,ds \leq C\frac{\delta_R}{u_m}\left(\|u_0-\widetilde u(0,\cdot)\|_{L^2(\mathbb R)}^2+\left(\delta_R/u_m\right)^{8/33}\right).
\end{aligned}
\end{equation*}
Finally, for \(III\), using the explicit formulas for the two rarefaction
profiles and the monotonicity of \(g=(f')^{-1}\), we obtain
\begin{equation*}
\begin{aligned}
\left\|u^r(t,\cdot-\sigma t)-\bar u^r\left(\frac{\cdot}{t}\right)\right\|_{L^1(\mathbb R)} &=\int_0^{\lambda_R}\bigl(u_+-g(\sigma+y)\bigr)\,dy \leq \delta_R\lambda_R.
\end{aligned}
\end{equation*}
Since both profiles take values in \([u_m,u_+]\), we therefore have
\begin{equation*}
\begin{aligned}
III &\leq C\delta_R\left\|u^r(t,\cdot-\sigma t)-\bar u^r\left(\frac{\cdot}{t}\right)\right\|_{L^1(\mathbb R)} \leq C\delta_R^2\lambda_R \leq C\left(\delta_R/u_m\right)^3.
\end{aligned}
\end{equation*}
Combining the estimates for \(I\), \(II\), and \(III\), and using \(\delta_R/u_m\leq\delta_0\), proves \eqref{eq:normalized-rarefaction-replacement}. Together with \eqref{eq:normalized-reference-stability}, this yields \eqref{eq:normalized-contraction-power}.

\subsection{Proof of Theorem~\ref{cor:inviscid-stability}}

Theorem~\ref{cor:inviscid-stability} follows from Theorem~\ref{mainthm} by the inviscid-limit argument of \cite[Section~5]{KV2}; see also \cite[Remark~1.3]{EKK}. In the present scalar setting, the $L^2$ estimate and the $H^1$ control of the shifts simplify the argument.

\subsubsection{Well-prepared initial data and uniform estimates in $\nu$}

We define the rescaled reference profile by
\begin{equation*}
\widetilde u^\nu(t,\xi):=\widetilde u\left(\frac{t}{\nu},\frac{\xi}{\nu}\right).
\end{equation*}
By the definition of \(\widetilde u\) and \(\bar u_0\), \eqref{eq:perturbed-shock-left-tail}, \eqref{eq:perturbed-shock-right-tail}, and \eqref{eq:approx-rarefaction-tails}, we have
\begin{align*}
& \|\widetilde u(0,\cdot)-\bar u_0\|_{L^2(\mathbb R)}^2 \\
& \quad \leq C\int_{-\infty}^0\left(|u^S-u_-|^2+|u^R-u_m|^2\right)\,d\xi +C\int_0^\infty\left(|u_m-u^S|^2+|u_+-u^R|^2\right)\,d\xi \\
& \quad \leq C\delta_S^2\int_{-\infty}^0e^{-2c\delta_S^2|\xi|}\,d\xi+C\delta_R^2\int_{-\infty}^0e^{-4|\xi|}\,d\xi  +C\delta_S^2\int_0^\infty\frac{1}{(1+c\delta_S^2\xi)^2} \, d\xi \\
& \qquad +C\delta_R^2\lambda_R+C\delta_R^2\int_{\lambda_R}^\infty e^{-4|\xi-\lambda_R|}\,d\xi \leq C.
\end{align*}
Here, for \(0\leq\xi\leq\lambda_R\), we used \(0<u^R(0,\xi)-u_m<\delta_R\), and in the last inequality we used \eqref{lambdaR} and \(\delta_R/u_m\leq\delta_0\). Since \(\widetilde u^\nu(0,x)=\widetilde u(0,x/\nu)\) and \(\bar u_0(x)=\bar u_0(x/\nu)\), a change of variables gives
\begin{equation}\label{eq:initial-profile-limit}
\|\widetilde u^\nu(0,\cdot)-\bar u_0\|_{L^2(\mathbb R)}^2=\nu\|\widetilde u(0,\cdot)-\bar u_0\|_{L^2(\mathbb R)}^2\leq C\nu.
\end{equation}
As in \cite[Section~5.1]{KV2}, well-prepared initial data are obtained by mollification: take $u_0^\nu=u_0*\rho_\nu$, where $\rho_\nu(x)=\nu^{-1}\rho(x/\nu)$ is a standard mollifier. By \eqref{eq:perturbed-shock-left-tail}, \eqref{eq:perturbed-shock-right-tail}, and \eqref{eq:approx-rarefaction-derivatives}, we have $\partial_x\widetilde u^\nu(0,\cdot)\in L^2(\mathbb R)$. Since $u_0-\bar u_0\in L^2(\mathbb R)$, \eqref{eq:initial-profile-limit} and standard properties of mollifiers yield \eqref{eq:well-prepared-initial-data}, which proves \textup{(i)}.

Now let $\{u_0^\nu\}_{\nu>0}$ be any sequence of smooth initial data satisfying \eqref{eq:well-prepared-initial-data}. Then, by \eqref{eq:initial-profile-limit},
\begin{equation}\label{eq:well-prepared-initial-data_2}
u_0^\nu-\widetilde u^\nu(0,\cdot)\longrightarrow u_0-\bar u_0\quad\text{in }L^2(\mathbb R).
\end{equation}
Let $u^\nu$ be the corresponding global strong solution of \eqref{eq:model-nu}. Applying Theorem~\ref{mainthm} to $v^\nu(s,y):=u^\nu(\nu s,\nu y)$ gives a shift $Y^\nu$. Set
\begin{equation*}
X^\nu(t):=\nu Y^\nu\left(\frac{t}{\nu}\right).
\end{equation*}
Then $X^\nu(0)=0$. With the comparison profile
\begin{equation*}
U^\nu(t,x):=u^S\left(\frac{x-\sigma t+X^\nu(t)}{\nu}\right)+\bar u^r\left(\frac{x}{t}\right)-u_m,\qquad t>0,
\end{equation*}
scaling \eqref{eq:normalized-contraction-power} and \eqref{eq:shift-velocity-bound} yields
\begin{equation}\label{eq:uniform-inviscid-estimate}
\sup_{t>0}\|u^\nu(t)-U^\nu(t)\|_{L^2(\mathbb R)}^2+\int_0^\infty|\dot X^\nu(t)|^2\,dt\leq C\left(\|u_0^\nu-\widetilde u^\nu(0,\cdot)\|_{L^2(\mathbb R)}^2+\nu\left(\delta_R/u_m\right)^{8/33}\right).
\end{equation}
The exact rarefaction is unchanged by the scaling.

\subsubsection{Weak limits and stability}

By \eqref{eq:uniform-inviscid-estimate}, \eqref{eq:well-prepared-initial-data_2}, and $u_-\leq U^\nu\leq u_+$, the sequence $\{u^\nu\}$ is uniformly bounded in $L^\infty(0,T;L_{\mathrm{loc}}^2(\mathbb{R}))$ for any $T>0$. By weak compactness and a diagonal extraction, we obtain a subsequence satisfying \eqref{eq:weak-inviscid-limit}, proving \textup{(ii)}.

Let $u_\infty$ be any limit obtained in \eqref{eq:weak-inviscid-limit} along a sequence $\nu_n\downarrow0$. For simplicity, we write $\nu$ for $\nu_n$. By \eqref{eq:uniform-inviscid-estimate} and \eqref{eq:well-prepared-initial-data_2}, together with $X^\nu(0)=0$, the shifts $\{X^\nu\}$ are uniformly bounded in $H^1(0,T)$. By compactness and a diagonal extraction, there exist a further subsequence and $X_\infty\in H^1_{\mathrm{loc}}([0,\infty))$ such that
\begin{equation*}
X^\nu\rightharpoonup X_\infty\quad\text{in }H^1(0,T),\qquad X^\nu\longrightarrow X_\infty\quad\text{in }C([0,T])
\end{equation*}
for any $T>0$, with $X_\infty(0)=0$. Define
\begin{equation*}
U^\infty(t,x):=\bar u^s(x-\sigma t+X_\infty(t))+\bar u^r\left(\frac{x}{t}\right)-u_m.
\end{equation*}
The rarefaction components of $U^\nu$ and $U^\infty$ coincide, so \eqref{eq:perturbed-shock-left-tail} and \eqref{eq:perturbed-shock-right-tail} give
\begin{equation}\label{eq:shifted-comparison-profile-limit}
\sup_{0<t\leq T}\|U^\nu(t)-U^\infty(t)\|_{L^2(\mathbb R)}^2\leq C\nu+C(u_m-u_-)^2\|X^\nu-X_\infty\|_{C([0,T])}\longrightarrow0,
\end{equation}
where we used $\|\bar u^s(\cdot+h)-\bar u^s\|_{L^2(\mathbb R)}^2=(u_m-u_-)^2|h|$.
Following \cite[Section~5.2.4]{KV2}, we localize \eqref{eq:uniform-inviscid-estimate} in time and pass to the limit using \eqref{eq:well-prepared-initial-data_2}, \eqref{eq:weak-inviscid-limit}, \eqref{eq:shifted-comparison-profile-limit}, and the weak lower semicontinuity of the $L^2$-norm. This yields \eqref{eq:inviscid-unweighted-stability} for almost every $t>0$.

Similarly, applying the weak lower semicontinuity of the $L^2$-norm to \eqref{eq:uniform-inviscid-estimate} and using \eqref{eq:well-prepared-initial-data_2}, we have
\begin{equation*}
\int_0^T|\dot X_\infty(t)|^2\,dt\leq\liminf_{\nu\to0}\int_0^T|\dot X^\nu(t)|^2\,dt\leq C\|u_0-\bar u_0\|_{L^2(\mathbb R)}^2
\end{equation*}
for all $T>0$. Letting $T\to\infty$, we obtain
\begin{equation*}
\int_0^\infty|\dot X_\infty(t)|^2\,dt\leq C\|u_0-\bar u_0\|_{L^2(\mathbb R)}^2.
\end{equation*}
Since $X_\infty(0)=0$, the Cauchy--Schwarz inequality gives
\begin{equation*}
|X_\infty(t)|=\left|\int_0^t\dot X_\infty(s)\,ds\right|\leq\sqrt{t}\,\|\dot X_\infty\|_{L^2(0,t)}\leq C\sqrt{t}\,\|u_0-\bar u_0\|_{L^2(\mathbb R)}
\end{equation*}
for all $t\geq0$. This proves \eqref{eq:inviscid-shift-bound} and completes \textup{(iii)}.

Finally, we verify the uniqueness assertion. Suppose that $u_0=\bar u_0$. Then \eqref{eq:inviscid-shift-bound} gives $X_\infty(t)=0$ for all $t \geq 0$. Hence the limiting comparison profile $U^\infty$ coincides with the Riemann solution $\bar u$. Therefore, \eqref{eq:inviscid-unweighted-stability} yields
\begin{equation*}
\|u_\infty(t)-\bar u(t)\|_{L^2(\mathbb R)}=0
\end{equation*}
for almost every $t>0$. This implies the uniqueness of the Riemann solution in the class of weak inviscid limits and completes the proof of Theorem~\ref{cor:inviscid-stability}.

\appendix

\section{Proofs of Lemma~\ref{lem:perturbed-degenerate-profile} and Lemma~\ref{lem:existence_shock}} \label{app:profile}

This section contains the deferred proofs of Lemmas~\ref{lem:perturbed-degenerate-profile} and~\ref{lem:existence_shock}. We set
\begin{equation*}
H(v):=f(v)-f(u_m)-f'(u_m)(v-u_m).
\end{equation*}

\begin{proof}[Proof of Lemma~\ref{lem:perturbed-degenerate-profile}]

Since \(H(u_m)=H'(u_m)=0\), Taylor's theorem gives
\begin{equation}\label{eq:perturbed-H-factor}
H(v)=(v-u_m)^2K(v),\qquad K(v):=\int_0^1(1-s)f''\bigl(u_m+s(v-u_m)\bigr) \, ds.
\end{equation}
For a left state $u_-<u_m$, the condition \eqref{eq:perturbed-degenerate-compatibility} is equivalent to $H(u_-)=0$, and hence to \(K(u_-)=0\). For the cubic flux \(f_0(v)=v^3\), the corresponding function \(K_0\) is given by
\begin{equation*}
K_0(v)=6\int_0^1(1-s)\bigl(u_m+s(v-u_m)\bigr) \, ds = v+2u_m,
\end{equation*}
so that $K_0(-2u_m)=0$. We now show that, if \(\varepsilon_f\) is sufficiently small, the function \(K\) associated with \(f\) has a unique zero \(u_-\) near \(-2u_m\).

Writing \(f=f_0+q\), we have
\begin{equation}\label{fKv}
K(v)= K_0(v) + \int_0^1(1-s)q''\bigl(u_m+s(v-u_m)\bigr)\,ds.
\end{equation}
By the definition \eqref{eq:flux-perturbation-size} of \(\varepsilon_f\), the integral on the right-hand side satisfies
\begin{equation*}
\begin{aligned}
\left|\int_0^1(1-s)q''\bigl(u_m+s(v-u_m)\bigr)\,ds\right| &\leq \varepsilon_f\int_0^1(1-s)\bigl((2-s)u_m+s|v|\bigr)\,ds\\
&=\varepsilon_f\left(\frac56u_m+\frac16|v|\right) \leq C\varepsilon_f(u_m+|v|).
\end{aligned}
\end{equation*}
Since \(K_0(-2u_m)=0\), evaluating \eqref{fKv} at \(v=-2u_m\) yields
\begin{equation}\label{eq:perturbed-K-at-root}
|K(-2u_m)| \leq C\varepsilon_fu_m.
\end{equation}
Moreover, differentiating \(K\) and using \(f'''=6+q'''\) and \(6\int_0^1s(1-s)\,ds=1\), we obtain
\begin{equation}\label{eq:perturbed-K-derivative}
|K'(v)-1|=\left|\int_0^1s(1-s)q'''\bigl(u_m+s(v-u_m)\bigr)\,ds\right| \leq \frac16 \lVert q'''\rVert_{L^\infty} \leq C\varepsilon_f.
\end{equation}
For sufficiently small \(\varepsilon_f>0\), this implies
\begin{equation*}
\frac12\leq K'(v)\leq\frac32,\qquad v\in\mathbb R.
\end{equation*}
Consequently, \(K\) is strictly increasing, with
\begin{equation*}
\lim_{v\to-\infty}K(v)=-\infty,\qquad \lim_{v\to+\infty}K(v)=+\infty.
\end{equation*}
Hence there exists a unique \(u_-\in\mathbb R\) such that $K(u_-)=0$. By the mean value theorem, for some \(\theta\) between \(u_-\) and \(-2u_m\), we have
\begin{equation*}
K(u_-)-K(-2u_m)=K'(\theta)(u_-+2u_m).
\end{equation*}
Using \(K(u_-)=0\), \eqref{eq:perturbed-K-at-root}, and \(K'(\theta)\geq1/2\), we obtain
\begin{equation*}
|u_-+2u_m| = \frac{|K(-2u_m)|}{|K'(\theta)|} \leq C\varepsilon_fu_m.
\end{equation*}
This proves \eqref{eq:perturbed-left-state} and, for sufficiently small \(\varepsilon_f>0\), places \(u_-\) in \((-5u_m/2,-3u_m/2)\). Since \(K(u_-)=0\), the compatibility condition \eqref{eq:perturbed-degenerate-compatibility} also holds.

We next prove \eqref{eq:perturbed-shock-factorization} and \eqref{eq:perturbed-factor-bound}. For \(v\neq u_-\), set
\begin{equation*}
b(v):=\frac{K(v)}{v-u_-},
\end{equation*}
and define \(b(u_-):=K'(u_-)\). Since \(K(u_-)=0\), the fundamental theorem of calculus gives
\begin{equation*}
b(v)=\int_0^1K'\bigl(u_-+s(v-u_-)\bigr)\,ds,\qquad v\in\mathbb R.
\end{equation*}
Using \eqref{eq:perturbed-K-derivative}, we obtain
\begin{equation*}
|b(v)-1|\leq C\varepsilon_f,\qquad \frac12\leq b(v)\leq\frac32
\end{equation*}
for sufficiently small \(\varepsilon_f>0\). Combining \(K(v)=(v-u_-)b(v)\) with \eqref{eq:perturbed-H-factor}, we obtain
\begin{equation*}
H(v)=(v-u_-)(v-u_m)^2b(v).
\end{equation*}

Finally, for \(v\geq u_m\), \eqref{eq:flux-perturbation-size} yields
\begin{equation*}
f''(v)=6v+q''(v)\geq6v-\varepsilon_f(u_m+v)\geq(6-2\varepsilon_f)u_m\geq cu_m
\end{equation*}
for sufficiently small \(\varepsilon_f>0\).
\end{proof}

\begin{proof} [Proof of Lemma~\ref{lem:existence_shock}]
The factorization \eqref{eq:perturbed-shock-factorization} and the bound \eqref{eq:perturbed-factor-bound} yield
\begin{equation*}
H(u_-)=H(u_m)=0,\qquad H(v)>0\quad\text{for }u_-<v<u_m.
\end{equation*}
This, together with the Rankine--Hugoniot relation, implies that the shock speed $\sigma$ satisfies the degeneracy condition \eqref{eq:perturbed-degenerate-compatibility} and the strict Oleinik condition
\begin{equation*}
\frac{f(u_m)-f(v)}{u_m-v} < \sigma < \frac{f(v)-f(u_-)}{v-u_-}, \qquad u_-<v<u_m.
\end{equation*}
The standard existence theory for scalar viscous shock profiles yields a unique (up to translation) profile connecting \(u_-\) to \(u_m\); see \cite{KM}. The condition at \(\xi=0\) in \eqref{eq:shock-ode} fixes the translation and hence yields uniqueness. Integrating the shock ODE in \eqref{eq:shock-ode} once and using the end state at \(+\infty\), we obtain
\begin{equation*}
u^S_\xi=H(u^S)=(u^S-u_-)(u^S-u_m)^2b(u^S) >0,
\end{equation*}
which proves \eqref{eq:perturbed-shock-ode}.

It remains to establish the pointwise bounds \eqref{eq:perturbed-shock-left-tail} and \eqref{eq:perturbed-shock-right-tail}. Let \(\delta_S:=u_m-u_-\). Since \(u^S_\xi>0\) and \(u^S(0)=(u_-+u_m)/2\), we have
\begin{equation*}
0<u^S(\xi)-u_-\leq\frac{\delta_S}{2},\qquad \frac{\delta_S}{2}\leq u_m-u^S(\xi)<\delta_S,\qquad \xi\leq0.
\end{equation*}
Using \eqref{eq:perturbed-shock-ode} and \eqref{eq:perturbed-factor-bound}, we obtain
\begin{equation} \label{diffinequS}
c\delta_S^2\bigl(u^S-u_-\bigr)\leq u^S_\xi\leq C\delta_S^2\bigl(u^S-u_-\bigr),\qquad \xi\leq0.
\end{equation}
Dividing by \(u^S-u_-\), integrating from \(\xi\) to \(0\), and using \(u^S(0)-u_-=\delta_S/2\), we obtain
\begin{equation*}
0<u^S(\xi)-u_-\leq\frac{\delta_S}{2}e^{-c\delta_S^2|\xi|},\qquad \xi\leq0.
\end{equation*}
The upper bound in \eqref{diffinequS} then gives
\begin{equation*}
0<u^S_\xi(\xi)\leq C\delta_S^3e^{-c\delta_S^2|\xi|},\qquad \xi\leq0.
\end{equation*}
This proves \eqref{eq:perturbed-shock-left-tail}.

For \(\xi\geq0\), monotonicity gives
\begin{equation*}
\frac{\delta_S}{2}\leq u^S(\xi)-u_-<\delta_S.
\end{equation*}
Set \(r(\xi):=u_m-u^S(\xi)\). Then \(r(0)=\delta_S/2\), and \eqref{eq:perturbed-shock-ode} gives
\begin{equation*}
-r_\xi=(u^S-u_-)r^2b(u^S).
\end{equation*}
Together with \eqref{eq:perturbed-factor-bound}, this implies
\begin{equation*}
\frac{d}{d\xi}\left(\frac1r\right)\geq c\delta_S,\qquad \xi\geq0.
\end{equation*}
Integrating from \(0\) to \(\xi\), we obtain
\begin{equation*}
\frac1{r(\xi)}\geq\frac2{\delta_S}+c\delta_S\xi,
\end{equation*}
and hence
\begin{equation*}
0<r(\xi)\leq\frac{C\delta_S}{1+c\delta_S^2\xi},\qquad \xi\geq0.
\end{equation*}
Finally, \eqref{eq:perturbed-shock-ode} yields
\begin{equation*}
0<u^S_\xi(\xi)\leq C\delta_Sr(\xi)^2\leq\frac{C\delta_S^3}{(1+c\delta_S^2\xi)^2},\qquad \xi\geq0.
\end{equation*}
This proves \eqref{eq:perturbed-shock-right-tail}.
\end{proof}

\section{Proof of the degenerate Hardy--Poincar\'e inequality} \label{sec:poincare}

For \(y\in[0,1]\), we set
\begin{equation}\label{ABm}
A(y):=(7-3y)y(1-y)^2,\qquad B(y):=(1-y)(12-y),\qquad m(y):=3-2y.
\end{equation}
Throughout this section, we call a function \(v\) admissible if
\begin{equation*}
v \in H^1_{\mathrm{loc}}(0,1) \cap L^1(0,1),\qquad \int_0^1A v_y^2\,dy < \infty,\qquad \int_0^1Bv ^2\,dy<\infty.
\end{equation*}
Note that the function \(\Phi\) defined in \eqref{eq:profile-coordinate} is admissible by \(\phi\in H^1(\mathbb R)\) and the identities in \eqref{eq:profile-jacobian}. With these conventions, we restate Proposition~\ref{Poincare}:

\medskip
\noindent\textbf{Proposition~\ref{Poincare} (restated).}
There exists a constant \(c>0\) such that
\begin{equation}\label{eq:degenerate-hardy-poincare}
\int_0^1A\Phi_y^2\,dy-\int_0^1B\Phi^2\,dy+\frac{10}{3}\Big(\int_0^1m\Phi\,dy\Big)^2
\geq c\int_0^1B\Phi^2\,dy
\end{equation}
for any admissible \(\Phi\).

The proof of Proposition~\ref{Poincare} combines two complementary mechanisms. A weighted Hardy inequality adapted to the endpoint degeneracy yields a strict weighted Poincar\'e estimate on the \(B\)-weighted orthogonal complement of the constants. This leaves the constant mode as the only noncoercive direction, which is then controlled by the rank-one term generated by the shift.

\subsection{A weighted Hardy inequality}

\begin{lemma}[Weighted Hardy inequality]\label{lem:weighted-hardy}
There exists a constant $\gamma \in (0,1)$ such that
\begin{equation}\label{eq:dual-Hardy}
\int_0^1\frac{F^2}{A}\,dy\leq \gamma \int_0^1\frac{F_y^2}{B}\,dy
\end{equation}
for any absolutely continuous function \(F\) satisfying \(F(0)=F(1)=0\), where $A$ and $B$ are defined by \eqref{ABm}. In fact, one may take \(\gamma=7/10\).
\end{lemma}

\begin{proof}
Since the case \(\int_0^1F_y^2/B\,dy=+\infty\) is trivial, it suffices to consider the case
\begin{equation} \label{Fy2B}
\int_0^1\frac{F_y^2}{B}\,dy<\infty.
\end{equation}
Set
\begin{equation*}
\tau(y):=y(1-y)^2(10+y+y^2),\qquad h(y):=\frac{F(y)}{\tau(y)},\qquad 0<y<1.
\end{equation*}
For \(0<\varepsilon<1/2\), expanding \(F_y=(\tau h)_y\) and integrating by parts over \([\varepsilon,1-\varepsilon]\), we obtain
\begin{equation}\label{eq:weighted-hardy-identity}
\int_\varepsilon^{1-\varepsilon}\frac{F_y^2}{B} \, dy = \int_\varepsilon^{1-\varepsilon}\frac{\tau^2}{B}h_y^2\,dy + \int_\varepsilon^{1-\varepsilon}\frac{R}{A}F^2\,dy + \left[\frac{\tau_y}{\tau B}F^2\right]_\varepsilon^{1-\varepsilon},
\end{equation}
where
\begin{equation*}
R(y):=-\frac{A(y)}{\tau(y)}\left(\frac{\tau_y(y)}{B(y)}\right)_y.
\end{equation*}
Since \(B>0\) on \((0,1)\), the first term on the right-hand side of \eqref{eq:weighted-hardy-identity} is nonnegative. It therefore remains to control the second term and the boundary term.

We first show that the boundary term in \eqref{eq:weighted-hardy-identity} vanishes as \(\varepsilon\downarrow0\). Since \(F(0)=0\), the Cauchy--Schwarz inequality gives
\begin{equation*}
F(y)^2=\left(\int_0^yF_y(s)\,ds\right)^2\leq\left(\int_0^yB(s)\,ds\right)\left(\int_0^y\frac{F_y(s)^2}{B(s)}\,ds\right).
\end{equation*}
By \eqref{Fy2B}, the second factor tends to zero as \(y\downarrow0\), while \(\int_0^yB(s)\,ds=O(y)\). Hence
\begin{equation*}
F(y)^2=o(y)\qquad\text{as }y\downarrow0.
\end{equation*}
Similarly, since \(F(1)=0\),
\begin{equation*}
F(y)^2 = \left( \int_y^1 F_y (s) \, ds \right)^2 \leq\left(\int_y^1B(s)\,ds\right)\left(\int_y^1\frac{F_y(s)^2}{B(s)}\,ds\right)=o\bigl((1-y)^2\bigr)\qquad\text{as }y\uparrow1.
\end{equation*}
On the other hand, the definition of \(\tau\) gives
\begin{equation*}
\frac{\tau_y}{\tau B}=O(y^{-1})\qquad\text{as }y\downarrow0,\qquad \frac{\tau_y}{\tau B}=O\bigl((1-y)^{-2}\bigr)\qquad\text{as }y\uparrow1.
\end{equation*}
Combining these estimates, we have
\begin{equation} \label{endpoint_van}
\lim_{\varepsilon\downarrow0}\left[\frac{\tau_y}{\tau B}F^2\right]_\varepsilon^{1-\varepsilon} = 0
\end{equation}

We now estimate \(R\). A direct computation yields
\begin{equation*}
R(y)=\frac{(7-3y)(326+24y+179y^2-10y^3)}{(12-y)^2(10+y+y^2)}
\end{equation*}
and
\begin{equation*}
R(y)-\frac{10}{7} =\frac{1574(1-y)^4+1586y(1-y)^3+2281y^2(1-y)^2+2081y^3(1-y)+12y^4}{7(12-y)^2(10+y+y^2)} \geq0.
\end{equation*}
Hence \(R(y)\geq10/7\) on \([0,1]\). Returning to \eqref{eq:weighted-hardy-identity} and discarding its first term, we have
\begin{equation*}
\frac{10}{7}\int_\varepsilon^{1-\varepsilon}\frac{F^2}{A}\,dy \leq \int_\varepsilon^{1-\varepsilon}\frac{F_y^2}{B}\,dy -\left[\frac{\tau_y}{\tau B}F^2\right]_\varepsilon^{1-\varepsilon}.
\end{equation*}
Letting \(\varepsilon\downarrow0\) and using \eqref{endpoint_van}, we obtain
\begin{equation*}
\frac{10}{7} \int_0^1\frac{F^2}{A}\,dy\leq\int_0^1\frac{F_y^2}{B}\,dy.
\end{equation*}
Thus, \eqref{eq:dual-Hardy} holds with \(\gamma=7/10\).
\end{proof}

\subsection{The weighted mean-zero estimate}

We use the weighted Hardy inequality to obtain strict coercivity on the \(B\)-weighted mean-zero subspace.

\begin{lemma}[Weighted mean-zero inequality]\label{lem:weighted-mean-zero}
If $\varphi$ is admissible and satisfies
\begin{equation} \label{meanzero_cond}
\int_0^1B\varphi\,dy=0,
\end{equation}
then
\begin{equation}\label{eq:strict-mean-zero-coercivity}
\int_0^1B\varphi^2\,dy\leq\frac{7}{10}\int_0^1A\varphi_y^2\,dy,\qquad \Big(\int_0^1m\varphi\,dy\Big)^2\leq\frac{17}{150}\int_0^1A\varphi_y^2\,dy,
\end{equation}
where $A$, $B$, and $m$ are defined by \eqref{ABm}.
\end{lemma}

The first estimate in \eqref{eq:strict-mean-zero-coercivity} shows that the local quadratic form is coercive on the \(B\)-weighted orthogonal complement of the constants. Indeed, if \(\int_0^1B\varphi\,dy=0\), then
\begin{equation*}
\int_0^1A\varphi_y^2\,dy-\int_0^1B\varphi^2\,dy\geq\frac37\int_0^1B\varphi^2\,dy.
\end{equation*}
Thus, in the \(B\)-orthogonal decomposition into the constant mode and its complement, the only obstruction to coercivity is the constant mode.

\begin{proof}
Set
\begin{equation*}
F(y):=\int_0^yB(s)\varphi(s)\,ds.
\end{equation*}
Then, since \(B\varphi\in L^1(0,1)\), the function \(F\) is absolutely continuous on \([0,1]\) and satisfies \(F_y=B\varphi\) a.e. Also, \(F(0)=0\), while \(F(1)=0\) by the mean-zero condition \eqref{meanzero_cond}. Integrating by parts, with the endpoint terms justified in Remark~\ref{rem:hp-endpoint-ibp}, and applying the Cauchy--Schwarz inequality, we obtain
\begin{equation*}
\int_0^1B\varphi^2\,dy=-\int_0^1\varphi_yF\,dy\leq\Big(\int_0^1\frac{F^2}{A}\,dy\Big)^{1/2}\Big(\int_0^1A\varphi_y^2\,dy\Big)^{1/2}.
\end{equation*}
Applying Lemma~\ref{lem:weighted-hardy} and using \(F_y^2/B=B\varphi^2\), we conclude that
\begin{equation*}
\int_0^1B\varphi^2\,dy\leq\frac{7}{10}\int_0^1A\varphi_y^2\,dy.
\end{equation*}

To prove the second estimate in \eqref{eq:strict-mean-zero-coercivity}, we construct a function \(G\) such that \(G(0)=G(1)=0\) and
\begin{equation*}
\int_0^1m\varphi\,dy=-\int_0^1G\varphi_y\,dy.
\end{equation*}
The Cauchy--Schwarz inequality will then reduce the estimate to a bound on \(\int_0^1G^2/A\,dy\).

We set \(b:=\int_0^1B\,dy=35/6\) and \(\mu:=\int_0^1m\,dy=2\), and define, for \(y\in[0,1]\),
\begin{equation*}
G(y):=\int_0^y\left(m(s)-\frac{\mu}{b}B(s)\right)\,ds=-\frac1{35}y(1-y)(39-4y).
\end{equation*}
The choice of \(\mu/b\) ensures that \(\int_0^1(m-\frac{\mu}{b}B)\,dy=0\), and hence \(G(0)=G(1)=0\). The mean-zero condition \eqref{meanzero_cond} and the identity \(G_y=m-\frac{\mu}{b}B\) imply
\begin{equation}\label{eq:second-primitive}
\Big(\int_0^1m\varphi\,dy\Big)^2=\Big(\int_0^1G\varphi_y\,dy\Big)^2\leq\Big(\int_0^1\frac{G^2}{A}\,dy\Big)\int_0^1A\varphi_y^2\,dy,
\end{equation}
where the integration by parts is justified by Remark~\ref{rem:hp-endpoint-ibp}.

To complete the estimate, it remains to show that \(\int_0^1G^2/A\,dy\leq17/150\). For \(y\in[0,1]\), define
\begin{equation*}
\rho(y):=\frac{(39-4y)^2}{1225(7-3y)}.
\end{equation*}
Then \(G^2/A=y\rho(y)\). A direct computation gives \(\rho''(y)=15842/[1225(7-3y)^3]>0\). Since \(\rho(0)<9/50\) and \(\rho(1)=1/4\), the convexity of \(\rho\) implies
\begin{equation*}
\rho(y)\leq\frac9{50}(1-y)+\frac14y.
\end{equation*}
Consequently,
\begin{equation*}
\int_0^1\frac{G^2}{A}\,dy\leq\frac9{50}\int_0^1y(1-y)\,dy+\frac14\int_0^1y^2\,dy=\frac{17}{150}.
\end{equation*}
Substituting this bound into \eqref{eq:second-primitive} proves the second inequality in \eqref{eq:strict-mean-zero-coercivity}.
\end{proof}

\begin{remark}[Vanishing of the endpoint terms]\label{rem:hp-endpoint-ibp}
The two integrations by parts in the proof of Lemma~\ref{lem:weighted-mean-zero} are performed first on \([\varepsilon,1-\varepsilon]\), with boundary terms involving \(F\varphi\) and \(G\varphi\). We verify that both vanish as \(\varepsilon\downarrow0\).

For \(0<y<1/2\), the Cauchy--Schwarz inequality gives
\begin{equation*}
|\varphi(y)|\leq|\varphi(1/2)|+\Big(\int_y^{1/2}A\varphi_s^2\,ds\Big)^{1/2}\Big(\int_y^{1/2}\frac{1}{A}\,ds\Big)^{1/2}.
\end{equation*}
Since \(A(y)\sim7y\) as \(y\downarrow0\), and \(A(y)\sim4(1-y)^2\) as \(y\uparrow1\), the same argument at the other endpoint yields
\begin{equation*}
|\varphi(y)|\leq C_\varphi(1+|\log y|^{1/2})\quad\text{as }y\downarrow0,\qquad |\varphi(y)|\leq C_\varphi(1+(1-y)^{-1/2})\quad\text{as }y\uparrow1.
\end{equation*}

For the function \(F\) used in the proof, \(F(0)=F(1)=0\) and
\begin{equation*}
\int_0^1\frac{F_y^2}{B}\,dy=\int_0^1B\varphi^2\,dy<\infty.
\end{equation*}
Hence the endpoint estimates established in the proof of Lemma~\ref{lem:weighted-hardy} apply and give
\begin{equation*}
F(y)^2=o(y)\quad\text{as }y\downarrow0,\qquad F(y)^2=o((1-y)^2)\quad\text{as }y\uparrow1.
\end{equation*}
The explicit formula for \(G\) also shows that
\begin{equation*}
G(y)=O(y)\quad\text{as }y\downarrow0,\qquad G(y)=O(1-y)\quad\text{as }y\uparrow1.
\end{equation*}
Combining these estimates, we have
\begin{equation*}
F(y)\varphi(y)\to0,\qquad G(y)\varphi(y)\to0
\end{equation*}
as \(y\downarrow0\) and as \(y\uparrow1\). Thus the boundary terms in both integrations by parts vanish as \(\varepsilon\downarrow0\).
\end{remark}

\subsection{Proof of Proposition~\ref{Poincare}}

We now combine the mean-zero coercivity with the rank-one term by decomposing an admissible function into its \(B\)-weighted mean and mean-zero component.

Let \(\Phi\) be admissible. Set
\begin{equation*}
d:=\Big(\int_0^1B\,dy\Big)^{-1}\int_0^1B\Phi\,dy,\qquad \varphi:=\Phi-d, \qquad s:=\int_0^1m\varphi\,dy.
\end{equation*}
Then $\varphi$ satisfies the mean-zero condition \eqref{meanzero_cond} and $\varphi_y=\Phi_y$. Since \(\int_0^1B\,dy=35/6\) and \(\int_0^1m\,dy=2\), we have
\begin{equation*}
\int_0^1A\Phi_y^2\,dy=\int_0^1A\varphi_y^2\,dy,\qquad \int_0^1B\Phi^2\,dy=\int_0^1B\varphi^2\,dy+\frac{35}{6}d^2,\qquad \int_0^1m\Phi\,dy=s+2d.
\end{equation*}
Using these identities, we can rewrite the left-hand side of \eqref{eq:degenerate-hardy-poincare}, after completing the square in \(d\) and \(s\), as
\begin{equation}\label{eq:exact-rank-one-completion}
\begin{split}
&\int_0^1A\Phi_y^2\,dy-\int_0^1B\Phi^2\,dy+\frac{10}{3}\Big(\int_0^1m\Phi\,dy\Big)^2\\
&\quad=\int_0^1A\varphi_y^2\,dy-\int_0^1B\varphi^2\,dy+\frac{15}{2}\Big(d+\frac89s\Big)^2-\frac{70}{27}s^2.
\end{split}
\end{equation}
Lemma~\ref{lem:weighted-mean-zero} yields
\begin{equation*}
\int_0^1B\varphi^2\,dy+\frac{70}{27}s^2 \leq\left(\frac{7}{10}+\frac{70}{27}\frac{17}{150}\right)\int_0^1A\varphi_y^2\,dy =\frac{161}{162}\int_0^1A\varphi_y^2\,dy.
\end{equation*}
Substituting this estimate into \eqref{eq:exact-rank-one-completion}, we obtain
\begin{equation}\label{eq:positive-reduction}
\int_0^1A\Phi_y^2\,dy-\int_0^1B\Phi^2\,dy+\frac{10}{3}\Big(\int_0^1m\Phi\,dy\Big)^2
\geq\frac{1}{162}\int_0^1A\varphi_y^2\,dy+\frac{15}{2}\Big(d+\frac89s\Big)^2.
\end{equation}

It remains to control \(\int_0^1B\Phi^2\,dy\) by the right-hand side of \eqref{eq:positive-reduction}. Since
\begin{equation*}
d^2=\Big(d+\frac89s-\frac89s\Big)^2\leq2\Big(d+\frac89s\Big)^2+2s^2,
\end{equation*}
Lemma~\ref{lem:weighted-mean-zero} gives
\begin{equation*}
\begin{split}
\int_0^1B\Phi^2\,dy =\int_0^1B\varphi^2\,dy+\frac{35}{6}d^2 &\leq\frac{7}{10}\int_0^1A\varphi_y^2\,dy + \frac{35}{3}\Big(d+\frac89s\Big)^2+\frac{35}{3}s^2\\
&\leq\frac{91}{45}\int_0^1A\varphi_y^2\,dy+\frac{35}{3}\Big(d+\frac89s\Big)^2.
\end{split}
\end{equation*}
Comparing this estimate with \eqref{eq:positive-reduction} proves \eqref{eq:degenerate-hardy-poincare} for some \(c>0\).

\section{Global well-posedness of the scalar viscous conservation law}\label{app:global-wellposedness}

The following proposition provides the global well-posedness result needed for the proof of Theorem~\ref{mainthm}.

\begin{proposition}[Global well-posedness]\label{prop:global-wellposedness}
Let \(f\in C^2(\mathbb R)\), and let \(r\in C^2(\mathbb R)\) satisfy
\begin{equation*}
\lim_{x\to-\infty}r(x)=r_-,\qquad \lim_{x\to+\infty}r(x)=r_+,\qquad \min\{r_-,r_+\}\leq r(x)\leq\max\{r_-,r_+\},
\end{equation*}
together with \(r_x,r_{xx}\in L^2(\mathbb R)\). If \(u_0-r\in H^1(\mathbb R)\), then the Cauchy problem
\begin{equation}\label{eq:appendix-viscous-cauchy}
u_t+f(u)_x=u_{xx},\qquad u(0,\cdot)=u_0,
\end{equation}
admits a unique global strong solution \(u\) satisfying, for every \(T>0\),
\begin{equation}\label{eq:appendix-strong-regularity}
u-r\in C([0,T];H^1(\mathbb R))\cap L^2(0,T;H^2(\mathbb R)),\qquad u_t\in L^2(0,T;L^2(\mathbb R)).
\end{equation}
Moreover,
\begin{equation}\label{eq:appendix-maximum-principle}
\operatorname*{ess\,inf}_{\mathbb R}u_0\leq u(t,x)\leq\operatorname*{ess\,sup}_{\mathbb R}u_0,\qquad t\geq0,\quad x\in\mathbb R.
\end{equation}
\end{proposition}

\begin{proof}
Set \(v:=u-r\) and \(v_0:=u_0-r\). Then
\begin{equation}\label{eq:appendix-perturbation-equation}
v_t-v_{xx}=N(v),\qquad v(0,\cdot)=v_0,
\end{equation}
where
\begin{equation*}
N(v):=-\bigl(f(v+r)-f(r)\bigr)_x+r_{xx}-f(r)_x=r_{xx}-f'(v+r)(v_x+r_x).
\end{equation*}
Since \(r\) is bounded and \(r_x,r_{xx}\in L^2(\mathbb R)\), we have \(N(0)\in L^2(\mathbb R)\). If \(\|v\|_{H^1},\|w\|_{H^1}\leq R\), then the one-dimensional Sobolev embedding and the mean value theorem yield
\begin{equation*}
\begin{split}
\|N(v)-N(w)\|_{L^2} &\leq C_R\|v-w\|_{L^\infty}\bigl(\|v_x\|_{L^2}+\|r_x\|_{L^2}\bigr)+C_R\|v_x-w_x\|_{L^2} \leq C_R\|v-w\|_{H^1}.
\end{split}
\end{equation*}
Thus \(N:H^1(\mathbb R)\to L^2(\mathbb R)\) is locally Lipschitz.

Let \(A:=I-\partial_{xx}\) on \(L^2(\mathbb R)\), with domain \(D(A)=H^2(\mathbb R)\). Since \(A\) is sectorial and \(D(A^{1/2})=H^1(\mathbb R)\), the equation
\begin{equation*}
v_t+Av=v+N(v)
\end{equation*}
has, by \cite[Theorems~3.3.3 and~3.3.4]{Henry}, a unique maximal solution on \([0,T_*)\) such that
\begin{equation}\label{eq:appendix-local-regularity}
v\in C([0,T_*);H^1(\mathbb R))\cap C((0,T_*);H^2(\mathbb R))\cap C^1((0,T_*);L^2(\mathbb R)).
\end{equation}
Moreover,
\begin{equation}\label{eq:appendix-continuation-alternative}
T_*<\infty\quad\Longrightarrow\quad\limsup_{t\uparrow T_*}\|v(t)\|_{H^1(\mathbb R)}=\infty.
\end{equation}

To exclude finite-time blow-up in \eqref{eq:appendix-continuation-alternative}, we first obtain a uniform bound on the range of \(u\). Set
\begin{equation*}
m:=\operatorname*{ess\,inf}_{\mathbb R}u_0,\qquad M:=\operatorname*{ess\,sup}_{\mathbb R}u_0.
\end{equation*}
Since \(u_0-r\in H^1(\mathbb R)\), its continuous representative vanishes at infinity, and hence \(u_0(x)\to r_\pm\) as \(x\to\pm\infty\). Together with the assumptions on \(r\), this gives \(m\leq r\leq M\). The standard parabolic truncation argument, applied to \((u-M)_+\) and \((m-u)_+\), yields
\begin{equation*}
m\leq u(t,x)\leq M,\qquad 0\leq t<T_*,\quad x\in\mathbb R;
\end{equation*}
see, for example, \cite[Chapter~III, Section~7]{LSU}. This proves \eqref{eq:appendix-maximum-principle} on the maximal existence interval.

We now derive the estimates needed to show that \(T_*=\infty\). We rewrite \eqref{eq:appendix-perturbation-equation} as
\begin{equation}\label{eq:appendix-v-equation-divergence}
v_t+\bigl(f(u)-f(r)\bigr)_x=v_{xx}+h_r,\qquad h_r:=r_{xx}-f(r)_x\in L^2(\mathbb R).
\end{equation}
Note that, since both \(u\) and \(r\) take values in \([m,M]\),
\begin{equation} \label{fur_bd}
|f(u)-f(r)|\leq C|v|,\qquad |f'(u)|\leq C.
\end{equation}
Taking the \(L^2\)-inner product of \eqref{eq:appendix-v-equation-divergence} with \(v\), we obtain
\begin{equation*}
\frac12\frac{d}{dt}\|v\|_{L^2}^2+\|v_x\|_{L^2}^2=\int_{\mathbb R}\bigl(f(u)-f(r)\bigr)v_x\,dx+\int_{\mathbb R}h_rv\,dx
\end{equation*}
and, applying Young's inequality together with \eqref{fur_bd} and $h_r \in L^2$,
\begin{equation}\label{eq:appendix-L2-estimate}
\frac{d}{dt}\|v(t)\|_{L^2}^2+\|v_x(t)\|_{L^2}^2\leq C\bigl(1+\|v(t)\|_{L^2}^2\bigr).
\end{equation}
Similarly, taking the \(L^2\)-inner product of \eqref{eq:appendix-viscous-cauchy} with \(-u_{xx}\) yields
\begin{equation*}
\frac12\frac{d}{dt}\|u_x\|_{L^2}^2+\|u_{xx}\|_{L^2}^2=\int_{\mathbb R}f'(u)u_xu_{xx}\,dx\leq\frac12\|u_{xx}\|_{L^2}^2+C\|u_x\|_{L^2}^2,
\end{equation*}
and hence
\begin{equation}\label{eq:appendix-H1-estimate}
\frac{d}{dt}\|u_x(t)\|_{L^2}^2+\|u_{xx}(t)\|_{L^2}^2\leq C\|u_x(t)\|_{L^2}^2.
\end{equation}
These estimates are first integrated over \((\tau,t)\), with \(0<\tau<t<T_*\), and then extended to \([0,t]\) by \eqref{eq:appendix-local-regularity} and \(\tau\downarrow0\).

Assume, for contradiction, that \(T_*<\infty\). Then Gronwall's inequality applied to \eqref{eq:appendix-L2-estimate} and \eqref{eq:appendix-H1-estimate} yields
\begin{equation*}
\sup_{0\leq t<T_*}\bigl(\|v(t)\|_{L^2}+\|u_x(t)\|_{L^2}\bigr)<\infty.
\end{equation*}
Since \(v_x=u_x-r_x\), this implies that $\|v(t)\|_{H^1}$ remains bounded on $[0,T_*)$, which contradicts \eqref{eq:appendix-continuation-alternative}. Therefore, \(T_*=\infty\).

Finally, fix \(T>0\). Integrating \eqref{eq:appendix-H1-estimate} gives $u_{xx}\in L^2(0,T;L^2(\mathbb R))$. Since \(v_{xx}=u_{xx}-r_{xx}\) and \(v\in C([0,T];H^1(\mathbb R))\) by \eqref{eq:appendix-local-regularity}, we have $v\in L^2(0,T;H^2(\mathbb R))$. Moreover, \eqref{fur_bd} and \(u_x\in C([0,T];L^2(\mathbb R))\) imply
\begin{equation*}
u_t=u_{xx}-f'(u)u_x\in L^2(0,T;L^2(\mathbb R)).
\end{equation*}
This proves \eqref{eq:appendix-strong-regularity} and completes the proof.
\end{proof}

\end{document}